\documentclass[preprint]{elsarticle}
\usepackage{amsmath, amsthm, amsfonts}
\usepackage{url}

\newtheorem{theorem}{\bfseries Theorem}
\newtheorem{lemma}{\bfseries Lemma}
\newtheorem{proposition}{\bfseries Proposition}
\newtheorem{corollary}{\bfseries Corollary}
\theoremstyle{definition}

\newtheorem{remark}{\bfseries Remark}
\newtheorem{assumption}{\bfseries Assumption}

\newtheorem{fact}{\bfseries Fact}
\newtheorem{example}{\bfseries Example}
\DeclareMathOperator{\He}{He}
\DeclareMathOperator{\trace}{Tr}
\DeclareMathOperator{\diag}{Diag}
\DeclareMathOperator{\rank}{rank}

\def\Gcl{G_{\mathrm{\scriptstyle cl}}}
\def\bbC{\mathbb{C}}
\def\bbS{\mathbb{S}}
\begin{document}
\begin{frontmatter}
\title{Reduction of SISO H-infinity Output Feedback Control Problem\tnoteref{t1}}
\tnotetext[t1]{This study was supported by Toyota Riken Specially Promoted Research Program in 2018 (PI: Yoshio Ebihara (Kyoto University)). The first author was supported by JSPS KAKENHI Grant Numbers JP26400203 and JP17H01700.}
\author[1]{Hayato Waki\corref{cor1}}
 \ead{waki@imi.kyushu-u.ac.jp}
\author[2]{Yoshio Ebihara} \ead{ebihara@kuee.kyoto-u.ac.jp}
\author[3]{Noboru Sebe} \ead{sebe@ai.kyutech.ac.jp}
\cortext[cor1]{Corresponding author}
\address[1]{Institute of Mathematics
for Industry, Kyushu University, 744 Motooka, Nishi-ku, Fukuoka 819-0395, Japan}
\address[2]{Department of Electrical Engineering, Kyoto
University, Kyotodaigaku-Katsura, Nishikyo-ku, Kyoto 615-8510, Japan,}
\address[3]{Department of Intelligent and Control Systems,
%Faculty of Computer Science and Systems Engineering, 
Kyushu Institute of Technology, 680-4 Kawazu, Iizuka-shi, Fukuoka 820-8502, Japan}
\begin{abstract}
We consider the linear matrix inequality (LMI) problem of $H_\infty$ output feedback control problem for a generalized plant whose control input, measured output, disturbance input, and controlled output are scalar. We provide an explicit form of the optimal value. This form is the unification of some results in the literature of $H_\infty$ performance limitation analysis. To obtain the form of the optimal value, we focus on the non-uniqueness of perpendicular matrices, which appear in the LMI problem. We use the null vectors of invariant zeros associated with the dynamical system for the expression of the perpendicular matrices. This expression enables us to reduce and simplify the LMI problem. 
Our approach uses some well-known fundamental tools, e.g., the Schur complement, Lyapunov equation, Sylvester equation, and matrix completion. We use these techniques for the simplification of the LMI problem. Also, we investigate the structure of dual feasible solutions and reduce the size of the dual. This reduction is called a facial reduction in the literature of convex optimization.   
\end{abstract}
\end{frontmatter}

%%%%%%%%%%%%
% Contents %
%%%%%%%%%%%%

\section{Introduction}
The importance of $H_{\infty}$ control problems in robust control
was first pointed out by Zames \cite{Zames81}.  
To design $H_\infty$ controllers, \cite{Doyle89} proposed an approach via algebraic
Riccati equations and inequalities, which works fine under some assumptions on a given generalized
plant. Subsequently, an approach by using linear matrix inequalities (LMIs) are proposed in \cite{Gahinet94, Iwasaki94, Scherer97, Masubuchi98}.  
%Both these approaches find a parameter of controller whose $H_\infty$ norm of the closed-loop system is less than a priori given bound and which internally stabilizes the closed-loop system. 
Both these approaches enable us to design a controller
that internally stabilizes the closed-loop system and
makes its $H_\infty$ norm lower than a priori given bound.  

%In contrast, we focus on the minimization of $H_\infty$ norm with the closed-loop system and which internally stabilizes the closed-loop system. In particular, we consider the LMI problem of this minimization throughout this paper. 

In contrast, we derive an explicit form of the infimum $H_\infty$ norm
in $H_\infty$ optimal controller synthesis problems
without any care for controller construction.  
To that end, we focus on the LMI optimization problem associated with the $H_\infty$ optimal controller synthesis. 
Because this minimization is formulated as the infimum, it has no guarantee to have any optimal solutions. In other words, some of the variables in the LMI problem may go to infinity when the objective value approaches its optimal value. Then we may encounter numerical difficulties in such cases. Even if one can construct a controller from a computed solution of the LMI problem, it may be fragile to small changes in the parameters of the controller.

The infimum $H_\infty$ norm is often analytically computed. Such approaches are proposed in the literature of $H_\infty$ performance limitation analysis. For instance, \cite{Chen00} provided the $H_\infty$ performance limitations of sensitivity and complementary sensitivity functions for MIMO linear time-invariant systems. In particular, the Nevanlinna-Pick interpolation was used. After obtaining the infimum, one can compute the desired controller whose $H_\infty$ norm is close to this infimum by applying the existing Riccati, or LMI approaches.

%$H_\infty$ control has attracted much interest since Zames \cite{Zames81} proposed the importance in the control community. As the methodologies of $H_\infty$ control, approaches of Nevanlinna-Pick interpolation, algebraic Riccati equations/inequalities \cite{Doyle89}, and linear matrix inequalities (LMIs) \cite{Iwasaki94, Scherer97, Masubuchi98} are proposed. %Nevanlinna-Pick interpolation is based on complex analysis and functional analysis. The author of \cite{Chen00} used Nevanlinna Pick interpolation to provide a $H_\infty$ performance limitation of the sensitivity and the complementary sensitivity functions of MIMO linear time-invariant systems.  
%The authors of \cite{Doyle89} provided a suboptimality condition for a given MIMO linear time-invariant dynamical system via algebraic Riccati equations/inequalities under invariant zeros in the system. They proposed a bisection based on this suboptimality condition. On the other hand, an LMI problem is obtained from a given plant. A desired controller is computed by solving the LMI problem. Unlike the algebraic Riccati equations/inequalities approach, the assumptions on the invariant zeros are not required in the approach of solving LMI problems. 

\subsection*{Contribution}
We deal with a generalized plant whose input, output, disturbance and controlled output are scalar. The plant is formulated as follows. 
\begin{align}\label{SISO}
&\left\{
\begin{array}{lll}
\dot{x} &=& Ax + b_1w + b_2u\\
z & = & c_1^Tx + d_{11} w + d_{12}u \\
y &=& c_2^Tx + d_{21} w, 
\end{array}
\right. 
\end{align}
where $A\in\mathbb{R}^{n\times n}$, $b_i, c_i\in\mathbb{R}^{n}$ and $d_{ij}\in\mathbb{R}$. In this paper, we refer to \eqref{SISO} as a generalized plant for SISO $H_\infty$ control problem. %an SISO dynamical system. 
The contribution of this paper is to provide an explicit form of 
the optimal value for the well-known LMI problem 
in relation to the SISO $H_{\infty}$ output feedback control problem. The main result can be summarized in the following theorem 
that is obtained by combining 
Theorems \ref{thm:case1}, \ref{thm:case2}, \ref{thm:case3} and \ref{thm:case4} in this paper.  
\begin{theorem}\label{thm:main}
 Let $\gamma^*$ be the optimal value of the LMI problem obtained from $H_{\infty}$ output feedback control for \eqref{SISO} by applying the elimination of variable method. The transfer matrix $G(s)$ for \eqref{SISO} is denoted by 
\[
G(s) = \begin{pmatrix}
G_{zw}(s) & G_{zu}(s)\\
G_{yw}(s) & G_{yu}(s)
\end{pmatrix}. 
\]
Moreover, $\lambda_1, \ldots, \lambda_{m_1}$ (resp. $\omega_1, \ldots, \omega_{m_2}$) denote invariant zeros on the imaginary axis via the realization $(A, b_2, c_1^T, d_{12})$ of $G_{zu}$ (resp. $(A^T, c_2, b_1^T, d_{21})$ of $G_{yw}$). We assume that all of $\lambda_1, \ldots, \lambda_{m_1}$, $\omega_1, \ldots, \omega_{m_2}$ and their complex conjugates are not eigenvalues of $A$. 

\begin{enumerate}
\item \label{main1} If $d_{12}\neq 0$ and $d_{21}\neq 0$, then $\gamma^*$ is equal to
\[
\max\left\{
\hat{\gamma}, |G_{zw}(\lambda_j)| \ (j=1, \ldots, m_1), |G_{zw}(\omega_j)| \ (j=1, \ldots, m_2)
\right\},
\]
where $\hat{\gamma}$ is the maximum eigenvalue of a symmetric matrix defined by unstable invariant zeros in $G_{zu}$ and $G_{yw}$ and their associated vectors. (See \eqref{gammacase1} for the definition of the matrix.) 
\item \label{main2} If $d_{12}= 0$ or $d_{21}= 0$, then $\gamma^*$ is equal to 
\[
\max\left\{
\hat{\gamma}, |G_{zw}(\lambda_j)| \ (j=1, \ldots, m_1), |G_{zw}(\omega_j)| \ (j=1, \ldots, m_2), |G_{zw}(\infty)| 
\right\},
\]
where $G_{zw}(\infty)$ is the value of the transfer function $G_{zw}$ at infinity. 
\end{enumerate}
Here $|G_{zw}(\lambda_j)|$ (resp. $|G_{zw}(\omega_j)|$) is vanished from the above expressions of $\gamma^*$ if the realization of $G_{zu}$ (resp. $G_{yw}$) has no invariant zeros on the imaginary axis. 
\end{theorem}

%In the characterization of $\gamma^*$ in Theorem \ref{thm:main}, we see the absolute values of the transfer function $G_{zw}(s)$ at invariant zeros and infinity. 

In general, the notion of invariant zeros is defined for the realization or state-space representation, not the transfer function. However, for the sake of brevity, we call an invariant zero of the realization $(A, b_2, c_1^T, d_{12})$ (resp. $(A^T, c_2, b_1^T, d_{21})$) {\itshape an invariant zero of $G_{zu}$} (resp. {\itshape $G_{yw}$}) throughout this paper. 

We give remarks on Theorem \ref{thm:main}. 
\begin{remark}\label{remark:main}
\begin{enumerate}
\item\label{R1} Although we have assumed that all invariant zeros $\lambda_j$ and $\omega_j$ on the imaginary axis are not eigenvalues of $A$, we can remove this assumption. Then we can describe $\gamma^*$ by the null vectors associated with invariant zeros $\lambda_j$ and $\omega_j$, instead of $G_{zw}(\lambda_j)$ and $G_{zw}(\omega_j)$ in $\gamma^*$. See Theorem \ref{thm:case3}. 
\item\label{R2} The $H_{\infty}$ control problem for \eqref{SISO} is the problem of finding a controller $K(s)$ which minimizes the $H_{\infty}$ norm of the closed-loop $\Gcl(s, K)$ obtained by connecting $K(s)$ with \eqref{SISO}. Mathematically, this problem can be formulated as follows: 
\begin{align}\label{Hinf}
\gamma^* &= \inf_{K\in\mathcal{K}}\sup_{s\in\sqrt{-1}\mathbb{R}}\sigma_{\max}\left(\Gcl(s, K)\right), 
\end{align}
where $\sqrt{-1}$ indicates the imaginary unit, 
\[
\Gcl(s, K) := G_{zw}(s)+G_{zu}(s)K(s)(1-G_{yu}(s)K(s))^{-1}G_{yw}(s)
\]
and $\mathcal{K}$ is the set of rational functions on $s$ which stabilize $G(s)$ internally. 

We see that when $\lambda$ is an invariant zero on the imaginary axis
   of $G_{zu}$ (resp. $G_{yw}$) of \eqref{SISO}, the value of
   the transfer function $\Gcl(\lambda, K)$ is $G_{zw}(\lambda)$. In
   fact, we have $G_{zu}(\lambda) = 0$ (resp. $G_{yw}(\lambda)=0$) because $\lambda$ is not an
   eigenvalue of $A$. The detail will be provided in Lemma
   \ref{lem:zero}.  
   Therefore, Theorem \ref{thm:main} contains the values of $G_{zw}$ at zeros of $G_{zu}$ and $G_{yw}$ over the imaginary axis. 

\item\label{R3} The performance index $\gamma^*$ is greater than or equal to $|G_{zw}(\infty)| = |d_{11}|$ when at least one of $d_{12}$ and $d_{21}$ is zero. To see this, let $K(s)$ be the transfer function of a controller with a realization $(A_K, b_K, c_K^T, d_K)$. 
%In fact, the state-space representation of the controller for \eqref{SISO} can be parametrized as follows:
%\[
%K(s) \left\{
%\begin{array}{lll}
%\dot{x}_K &=& A_Kx_K + b_Ky, \\
%u &=& c_K^Tx_K + d_Ky. 
%\end{array}
%\right. 
%\]
%The state-space representation of the closed loop can be formulated as 
%\[
%\Gcl(s) 
%\left\{
%\begin{array}{cll}
%\begin{pmatrix}
%\dot{x} \\
%\dot{x_K}
%\end{pmatrix} &=& \begin{pmatrix}
%A + b_2d_Kc_2^T & b_2c_K^T\\
% b_Kc_2^T & A_K
%\end{pmatrix}\begin{pmatrix}
%x\\
%x_K
%\end{pmatrix} + \begin{pmatrix}
%b_2d_Kd_{21}\\
%d_{21}b_K
%\end{pmatrix}w\\
%z &=& \begin{pmatrix}
%c_1^T+d_{12}d_Kc_2^T, d_{12}c_K^T
%\end{pmatrix}\begin{pmatrix}
%x\\
%x_K
%\end{pmatrix} + (d_{11}+d_{12}d_Kd_{21})w
%\end{array}
%\right. 
%\]
   We see from $\Gcl(\infty, K) = d_{11} + d_{12}d_Kd_{21}$ that we can
   reduce the effect of the feedthrough term of \eqref{SISO} 
by the choice of $d_K$ when $d_{12}\neq 0$ and $d_{21}\neq
   0$. Otherwise, we cannot reduce it because the feedthrough term of
   the closed-loop system is $d_{11}$, which is independent in the choice of the parameter $d_K$. Therefore the performance index $\gamma^*$ is greater than or equal to $|G_{zw}(\infty)| = |d_{11}|$ when at least one of $d_{12}$ and $d_{21}$ is zero. 
\end{enumerate}
\end{remark}

To prove the main result, Theorem \ref{thm:main}, we consider the following cases: 
\begin{enumerate}
\item\label{Case1} Both $d_{12}$ and $d_{21}$ are nonzero, and all invariant zeros in $G_{zu}$ and $G_{yw}$ are unstable, but not on the imaginary axis. %, i.e., the real parts of all diagonal elements in $\Lambda$ and $\Omega$ are positive. 
\item\label{Case2} Both $d_{12}$ and $d_{21}$ are nonzero, and at least one of the invariant zeros in $G_{zu}$ or $G_{yw}$ is stable, but all unstable invariant zeros are not on the imaginary axis. %, i.e., at least one of the real parts of the diagonal elements in $\Lambda$ or $\Omega$ is negative.
\item\label{Case3} Both $d_{12}$ and $d_{21}$ are nonzero, and at least one of the invariant zeros in $G_{zu}$ or $G_{yw}$ exists on the imaginary axis. 
\item\label{Case4} At least one of $d_{12}$ and $d_{21}$ is zero. %In this case, an infinite zero exists in $G_{zu}(s)$ or $G_{yw}(s)$ 
\end{enumerate}

All possible generalized plant of the form \eqref{SISO} are 
exactly one of these cases. For simplicity, we will assume in this paper that all invariant zeros are real and distinct from each other. In the analysis of Case \ref{Case1}, we use the Schur complement and the Lyapunov equation. On the other hand, we use not only these mathematical tools but also a technique of the matrix completion problem in Lemma \ref{mcomp} and the dual of the resulting LMI problem in the analysis of Cases \ref{Case2}, \ref{Case3} and \ref{Case4}.

%The contribute of this note is to unify some approaches proposed in \cite{Ebihara16, Waki16} to analyze the $H_{\infty}$ performance limitation for SISO dynamic systems. The transfer function $M$ for SISO dynamic system is dealt with in \cite{Waki16}, and the sensitivity function $S$ is discussed in \cite{Ebihara16}. The $H_{\infty}$ performance limitation in both cases is obtained by parameterizing dual solutions of the resulting $H_{\infty}$ control problem with some vectors associated with unstable zeros in a particular SISO system. We unify both approaches in this note and provide a more straightforward understanding of the expression of dual solutions.  

%The key in this note is to use the (left) vectors associated with zeros in a given SISO system as the computation of the perpendicular complements. The LMI problem corresponded to $H_{\infty}$ control problem becomes a more straightforward form equivalently. We can easily see the same dual problem as in \cite{Waki16}.  

\subsection*{Why do we deal with the dual problem?}
The reason is that we can reduce the dual problem and some techniques developed in Case \ref{Case1} are available. In all of Cases \ref{Case2}, \ref{Case3} and \ref{Case4}, the dual problem is feasible, but not strictly feasible. We exploit this property of the dual of the resulting LMI problem in these cases. More precisely, the dual problem of all these cases is formulated as follows.
\begin{align}\label{DUAL}
&\sup\left\{
L_0\bullet X : L_j\bullet X = b_j \ (j=1, \ldots, m), X\in\mathbb{S}^n_+
\right\}, 
\end{align} 
where $L_0, L_1, \ldots, L_m$ are $n\times n$ symmetric matrices, $b_1, \ldots, b_m\in\mathbb{R}$ and $L_j\bullet X =\trace(L_jX)$ for $j=0, 1, \ldots, m$. Then \eqref{DUAL} has no interior feasible solutions, i.e., no positive definite solutions in \eqref{DUAL}. Hence there exists an orthogonal matrix $P\in\mathbb{R}^{n\times n}$ and a positive integer $r$ such that any dual feasible solution $X$ has the form of 
\begin{align}\label{X}
X &= P\begin{pmatrix}
\tilde{X} & O_{r\times (n-r)}\\
O_{(n-r)\times r}&O_{(n-r)\times (n-r)}
\end{pmatrix}P^T 
\end{align}
for some $\tilde{X}\in\mathbb{S}^{r}_+$. In general, it is difficult to find the nonsingular matrix $P$ in the form \eqref{X} from \eqref{DUAL}. It, however, is relatively easy to compute such a matrix $P$ in all Cases \ref{Case2} to \ref{Case4}. 

Using \eqref{X}, we can reduce the LMI problem that corresponds to the dual \eqref{DUAL}. Substituting this form \eqref{X} to \eqref{DUAL}, we obtain the following problem whose optimal value is equal to that of \eqref{DUAL}. 
\begin{align}\label{DUAL2}
&\sup\left\{
\tilde{L}_0\bullet \tilde{X} : \tilde{L}_j\bullet \tilde{X} = b_j \ (j=1, \ldots, m), \tilde{X}\in\mathbb{S}^r_+
\right\}, 
\end{align} 
where the coefficient matrix $\tilde{L}_j\in\mathbb{S}^r$ is a square submatrix of the matrix $P^TL_jP$ for all $j=0, 1, \ldots, m$. Clearly the size of the positive semidefinite matrix in the LMI problem of \eqref{DUAL2} is smaller than the size of the original LMI problem. After reducing the LMI problem, some techniques developed in Case \ref{Case1} are available to the reduced LMI problem. 

This type of reduction is called {\itshape facial reduction} in the literature of the theory of convex optimization. The facial reduction was proposed in \cite{Borwein81}. In general, the strong duality for convex optimization requires a constrained qualification. Otherwise, the strong duality may fail, i.e., no optimal solutions and/or a positive duality gap. By applying the facial reduction to such convex optimization problems, the reduced problems always satisfy a constrained qualification, and thus the strong duality holds. The facial reduction was already applied in the literature of control theory, e.g., $H_2$ analysis in \cite{Balakrishnan03} and $H_\infty$ state feedback control in \cite{Waki16a,Waki18b}.

\subsection*{Related work}
This study is inspired by \cite{Chen00}, which deals with MIMO systems. We will obtain the same result for SISO systems to \cite{Chen00} in this study. The work \cite{Chen00} used a mathematical tool in complex analysis, while our result is obtained to analysis the LMI problem and its dual. 

Furthermore, this study unifies some of the existing work \cite{Ebihara16a, Ebihara16b, Ebihara16c}. The work \cite{Ebihara16a} obtained a lower bound of the $H_\infty$ performance limitations of $(1+PK)^{-1}P$, where $P$ and $K$ are transfer functions of a SISO linear time-invariant system and a controller, respectively. This lower bound was obtained from a detailed analysis of the resulting LMI problem. The exactness of the lower bound was proved in \cite{Ebihara16c} by using a property in the dual problem. This technique was also used in \cite{Ebihara16b}, which deals with the $H_\infty$ performance limitations of sensitivity and complementary sensitivity functions for a SISO linear time-invariant system. The dual problems play an essential role in both studies. In this study, we extend the analysis obtained in \cite{Ebihara16c} and provide the performance limitation for a more general SISO $H_\infty$ output feedback control problem. The analysis in \cite{Ebihara16a} for the dual problems can be regarded as facial reduction. %, which is known in the literature of the optimization theory. 

The work \cite{Helmersson12} reformulated the resulting LMI problem by using Kronecker canonical form (KCF) (a.k.a. Weierstrass form in this study) obtained from a given generalized plant. This reformulation separates variables in the LMI problem into bounded and unbounded variables. The unbounded variables are removed because they make no effect on the minimum value of the problem. In contrast, the bounded variables remain in the LMI problem. As a result, the size of the problem is reduced, and the numerical performance was improved. 

In comparison with facial reduction, the reduction via KCF deals with the LMI problem, while facial reduction deals with the dual problem. In other words, the reduction via KCF has a {\itshape dual} relation to facial reduction. In fact, one can construct an LMI problem whose dual corresponds to the dual problem reduced via facial reduction. %one can obtain an LMI problem from the dual problem reduced by facial reduction. 
Then one can see that the variables removed in the reduction via KCF also vanish in the LMI problem. In this sense, we can regard facial reduction as the dual approach of the reduction via KCF. 

\cite{SchererPhD} developed a variant of KCF for a given generalized plant and focused on the Riccati equations and inequalities obtained from the plant. A simplification of the Riccati equations and inequalities associated with the plant was provided with using this variant.

\subsection*{Organization of this paper}
The purpose of this paper is to prove Theorem \ref{thm:main}. For this, we consider the four cases, \ref{Case1} to \ref{Case4}. These cases are discussed in Sections \ref{sec:case1} to \ref{sec:case4}, respectively. Invariant zeros play an essential role in reducing the LMI problem. Section \ref{sec:preliminary} devotes the introduction of the concept of invariant zeros and their mathematical formulation. We also present the LMI formulation of $H_\infty$ output feedback control in Section \ref{sec:Hinf}. We also focus on the non-uniqueness of  perpendicular matrices, which appear in the LMI problem of $H_\infty$ output feedback control. For the perpendicular matrices, we use the null vectors associated with invariant zeros in the SISO dynamical system. We can see some existing results related to $H_\infty$ limitation analysis by using Theorem \ref{thm:main} in Section \ref{sec:application}. We give a conclusion of this paper in Section \ref{sec:conclusion}. We introduce other mathematical tools and proofs of some lemmas for proving Theorem \ref{thm:main} in Appendices.

\subsection*{Notation and symbols}
We introduce some notation and symbols used in this paper. Let $\mathbb{C}$ be the set of complex numbers. For $\lambda\in\mathbb{C}$, $\Re(\lambda)$ (resp. $\Im(\lambda))$ denotes the real (resp. imaginary) part of $\lambda$. We partition $\mathbb{C}$ into 
\[
\mathbb{C}_+=\{\lambda\in\mathbb{C} : \Re(\lambda)>0\}, \mathbb{C}_-=\{\lambda\in\mathbb{C} : \Re(\lambda)<0\} \mbox{ and } \mathbb{C}_0 =\{\lambda\in\mathbb{C} : \Re(\lambda)=0\}. 
\]

Let $\mathbb{S}^n$, $\mathbb{S}^n_{+}$ and $\mathbb{S}^n_{++}$ be the sets of $n\times n$ symmetric matrices, $n\times n$ positive semidefinite matrices and $n\times n$ positive definite matrices. For $A, B\in\mathbb{S}^n$, $A\succeq B$ denotes $A-B\in\mathbb{S}^n_+$. We define $A\bullet B = \trace(AB^T) = \sum_{k, \ell=1}^nA_{k\ell}B_{k\ell}$. We define $\He(M) = M+M^T$ for any square matrix $M$. 

We denote by $\sigma_{\max}(A)$ the maximum singular value of a matrix $A$.  
In addition, for a square matrix $A$, 
%we denote by $\lambda(A)$ the set of the eigenvalues of $A$ and 
%further 
we denote by $\lambda_{\max}(A)$ 
the maximum eigenvalue of $A$ when $A$ is symmetric.  
%Finally, for for a square matrix $A$, $\rho(A)$ stands for its spectral radius. 

For a given matrix $G\in{\mathbb{R}}^{n\times m}$ with rank $r$, $G^{\perp}$ denotes an $n\times (n-r)$ matrix which satisfies
$G^TG^{\perp} = O_{m\times (n-r)}$ and $(G^{\perp})^TG^{\perp} \in{\mathbb{S}}^{n-r}_{++}$. We call $G^{\perp}$  {\itshape a perpendicular matrix of $G$} throughout this paper. In general, $G^{\perp}$ is not unique for a given matrix $G$. %\begin{lemma}\label{imaker} _
%For $G\in{\mathbb{R}}^{m\times n}$, we have $\Ima G = \ker (G^{\perp})^T$ and $\ker G^T = \Ima G^{\perp}$.  
%\end{lemma}
$G^{\perp T}$ stands for the transpose of $G^{\perp}$ in this paper.

\section{Preliminaries}\label{sec:preliminary}
In this section, we review the definition and several properties of the invariant zeros of SISO LTI systems. 
Let us consider the following SISO LTI system $G$ described by
\begin{align}\label{fundSISO}
G:\ 
&\left\{
\begin{array}{lll}
\dot{x} & = & Ax + bu\\
y &=& c^Tx + du, 
\end{array}
\right. 
\end{align}
where $A\in\mathbb{R}^{n\times n}$, $b, c\in\mathbb{R}^n$ and $d\in\mathbb{R}$. 
The transfer function of the system $G$ is given by 
$G(s) = c^T(sI_n-A)^{-1}b + d$. 
We say that $\lambda\in\mathbb{C}$ is an invariant zero of 
\eqref{fundSISO} if 
\begin{align}\label{zerorank}
\rank\begin{pmatrix}
A-\lambda I_n & b\\
c^T & d
\end{pmatrix} < n+1. 
\end{align}
In addition, we say that 
an invariant zero $\lambda$ is {\itshape stable} 
if the real part of $\lambda$ is negative, i.e.,
$\lambda\in\mathbb{C}_-$. 
Otherwise, we say that the invariant zero $\lambda$ is {\itshape unstable}. 

We first provide some fundamental facts on the invariant zeros.  
\begin{lemma}\label{lem:zero}
\begin{enumerate}
\item\label{F0} If $\lambda\in\mathbb{C}$ is an invariant zero of
   \eqref{fundSISO} and if $\lambda\not\in\lambda(A)$ then $G(\lambda)=0$. Here $\lambda(A)$ denotes the set of all the eigenvalues of the matrix $A$. 
\item\label{F1} $\lambda\in\mathbb{C}$ is an invariant zero of \eqref{fundSISO} 
   if and only if there exists  
   $\left(\begin{smallmatrix}
   v_{\mathrm{L}}\\
   \hat{v}_{\mathrm{L}}
   \end{smallmatrix}\right)\in\mathbb{C}^{n+1}\setminus\{0\}$
   such that 
   \begin{align}\label{zerorank_stbl}
   \begin{pmatrix}
    v_{\mathrm L}^T & \hat{v}_{\mathrm L}
   \end{pmatrix}
   \begin{pmatrix}
    A & b\\
    c^T & d
   \end{pmatrix} 
   &= \lambda  
   \begin{pmatrix}
    v_{\mathrm L}^T & 0
   \end{pmatrix}. 
   \end{align}
   Similarly, 
   $\lambda\in\mathbb{C}$ is an invariant zero of \eqref{fundSISO} 
   if and only if there exists  
   $\left(
   \begin{smallmatrix}
   v_{\mathrm R}\\
   \hat{v}_{\mathrm R}
   \end{smallmatrix}
   \right)\in\mathbb{C}^{n+1}\setminus\{0\}$
   such that 
   \begin{align}\label{zerorank_stbr}
   \begin{pmatrix}
    A & b\\
    c^T & d
   \end{pmatrix} 
   \begin{pmatrix}
   v_{\mathrm R} \\ \hat{v}_{\mathrm R}
   \end{pmatrix}
   &= \lambda  
   \begin{pmatrix}
    v_{\mathrm R} \\ 0
   \end{pmatrix}. 
   \end{align}
 \item\label{F2} 
   If $(A, b)$ in \eqref{fundSISO} is controllable 
   then $\hat{v}_{\mathrm L}\in\bbC$ in \eqref{zerorank_stbl} is nonzero. 
   Similarly, if $(A, b)$ in \eqref{fundSISO} is stabilizable and if 
   $\lambda\in\bbC_+\cup\bbC_0$, 
   then $\hat{v}_{\mathrm L}\in\bbC$ is nonzero. 
 \item\label{F3} 
   If $(A, c^T)$ in \eqref{fundSISO} is observable
   then $\hat{v}_{\mathrm R}\in\bbC$ in \eqref{zerorank_stbr} is nonzero. 
   Similarly, if $(A, c^T)$ in \eqref{fundSISO} is detectable and if 
   $\lambda\in\mathbb{C}_+\cup\mathbb{C}_0$, 
   then $\hat{v}_{\mathrm R}\in\bbC$ is nonzero. 
   %
% \item\label{F3} If an invariant zero $\lambda\in\mathbb{C}$ of
%   \eqref{fundSISO} 
%   is controllable and observable mode
%   \footnote{\textcolor{blue}{``$\lambda$ is not an eigenvalue of $A$." is better?}}, 
%   then $\lambda$ is the zero of the transfer 
%   function $G(s) := c^T(\lambda I_n-A)^{-1}b + d$ of \eqref{fundSISO}. 
\end{enumerate}
\end{lemma}

\begin{proof}
The validity of the assertion \ref{F0} readily follows since for
 $\lambda\not\in\lambda(A)$
we have 
\begin{align*}
 \rank\begin{pmatrix}
A-\lambda I_n & b\\
c^T & d
\end{pmatrix} &=
 \rank\begin{pmatrix}
A-\lambda I_n & 0\\
c^T & c^T(\lambda I_n-A)^{-1}b+d.  
\end{pmatrix} \\
&=
 \rank\begin{pmatrix}
A-\lambda I_n & 0\\
c^T & G(\lambda)
\end{pmatrix}. 
\end{align*}
The validity of the assertion \ref{F1} is obvious.  
For the proof of the assertion \ref{F2}, 
we first consider the case where $(A, b)$ is controllable, 
i.e., $\rank\begin{pmatrix}
A-sI_n& b
\end{pmatrix}=n\ (\forall s\in\bbC)$.  
Suppose $\hat{v}_{\mathrm L} = 0$ for contradiction. 
Then it follows from \eqref{zerorank_stbl} that 
$v_{\mathrm L}^TA = \lambda v_{\mathrm L}^T$ and $v_{\mathrm L}^Tb = 0$. 
This equation contradicts the controllability of $(A, b)$. 
Therefore $\hat{v}_{\mathrm L} \neq 0$. 
We next consider the case where $(A, b)$ is stabilizable, 
i.e., $\rank\begin{pmatrix}
A-sI_n& b
\end{pmatrix}=n\ (\forall s\in\bbC_+\cup\bbC_0)$.  
Suppose $\hat{v}_{\mathrm L} = 0$ for contradiction. 
Then it follows from \eqref{zerorank_stbl} that 
$v_{\mathrm L}^TA = \lambda v_{\mathrm L}^T$ and $v_{\mathrm L}^Tb = 0$ for $\lambda\in\bbC_+\cup\bbC_0$. 
This equation contradicts the stabilizability of $(A, b)$. 
Therefore again $\hat{v}_{\mathrm L} \neq 0$. 
The assertion \ref{F3} can be proved similarly to the proof of
the assertion \ref{F2}.  
\end{proof}

In the following, 
we call $\left(
   \begin{smallmatrix}
   v_{\mathrm L}\\
   \hat{v}_{\mathrm L}
   \end{smallmatrix}
   \right)\in\mathbb{C}^{n+1}$ that satisfies \eqref{zerorank_stbl}
{\itshape the left null vector associated
with the invariant zero $\lambda$ of $G$}. 
Similarly, 
we call $\left(
   \begin{smallmatrix}
   v_{\mathrm R}\\
   \hat{v}_{\mathrm R}
   \end{smallmatrix}
   \right)\in\mathbb{C}^{n+1}$ that satisfies \eqref{zerorank_stbr}
{\itshape the right null vector associated
with the invariant zero $\lambda$ of $G$}.

As we see in the next theorem, invariant zeros in \eqref{fundSISO} plays an essential role in a canonical quasi-diagonal form of the Rosenbrock system matrix $\left(\begin{smallmatrix} sI_n-A & -b\\
-c^T & -d\end{smallmatrix}\right)$ of \eqref{fundSISO}. The canonical form is known as the Weierstrass form in \cite[eq. (3.19)]{Lewis86}. This theorem follows from \cite[Theorem 3 in Chapter XII]{Gantmacher89}.

\begin{theorem}\label{Weierstrass}
Assume $\left(
\begin{smallmatrix}
b\\
d
\end{smallmatrix}
\right)\neq 0$ \mbox{ or } $\left(
\begin{smallmatrix}
c\\
d
\end{smallmatrix}
\right)\neq 0$. Then there exist non-singular matrices $P, Q\in\mathbb{C}^{(n+1)\times(n+1)}$, a nonnegative integer $r$, a Jordan matrix $\Lambda\in\mathbb{C}^{(n-r)\times (n-r)}$ and a nilpotent $N\in\mathbb{C}^{(r+1)\times (r+1)}$ such that 
\begin{align}
  \label{cqf}
  P\begin{pmatrix}
  s I_n-A & -b\\
  -c^T & -d
  \end{pmatrix}Q& = \begin{pmatrix}
  s I_{n-r} - \Lambda & O\\
  O & s N - I_{r+1}
  \end{pmatrix}. 
\end{align}
Furthermore, the followings hold. 
\begin{enumerate}
  \item\label{a1} All eigenvalues of $\Lambda$ are invariant zeros of \eqref{fundSISO}. 
  \item\label{a2} The matrix $N$ consists of only one Jordan cell, i.e., $N^{r} \neq O$ and $N^{r+1} = O$. 
  \item\label{a3} %The nonnegative integer $r$ is equal to the {\itshape relative degree of \eqref{fundSISO}}, {i.e.}, 
  If $d\neq 0$, then $r=0$. 
  \item\label{a4} If $d=0$, then $c^Tb=0$, $c^TAb = 0$, $\ldots, c^TA^{r-2}b = 0$ and $c^TA^{r-1}b \neq 0$. 
  \end{enumerate}
\end{theorem}
\begin{proof}
\eqref{cqf} follows from \cite[Theorem 3 in Chapter XII]{Gantmacher89} and Jordan decomposition. We here prove all assertions. We notice that \eqref{cqf} is the identity on $s$. Then we obtain
\begin{align}\label{sterm}
  \begin{pmatrix}
I_n & 0\\
0^T&0
\end{pmatrix}Q &= P^{-1}\begin{pmatrix}
I_{n-r} & O\\
O & N
\end{pmatrix}, \\
\label{const}
\begin{pmatrix}
A & b\\
c^T & d
\end{pmatrix}Q& = P^{-1}\begin{pmatrix}
\Lambda & O \\
O & I_{r+1}
\end{pmatrix}. 
\end{align}

For the assertion \ref{a1}, we denote $P^{-1}, Q\in\mathbb{C}^{(n+1)\times (n+1)}$ by
\[
P^{-1} = \bordermatrix{
 & (n-r) & (r+1) \cr 
n& P_{11} & P_{12}\cr
1&P_{21} & P_{22}
}, Q = \bordermatrix{
 & (n-r) & (r+1) \cr 
n& Q_{11} & Q_{12}\cr
1& Q_{21} & Q_{22}
}. 
\]
%where $P_{11}\in\mathbb{C}^{n\times (n-r)}$, $P_{21}\in\mathbb{C}^{n-r}$, $P_{12}\in\mathbb{C}^{n\times (r+1)}$ and $P_{22}\in\mathbb{C}^{r+1}$. 
We obtain the following equations from \eqref{sterm} and \eqref{const}, respectively. 
\[P_{11} = Q_{11}, O=P_{21} \mbox{ and } 
\begin{pmatrix}
A & b\\
c^T & d
\end{pmatrix}\begin{pmatrix}
Q_{11}\\ 
Q_{21}
\end{pmatrix} = \begin{pmatrix}
P_{11}\Lambda \\
P_{21}\Lambda
\end{pmatrix}. 
\]
Eliminating $P_{11}$ and $P_{21}$ from those equations, we obtain
\begin{align}\label{eq:finitezero}
\begin{pmatrix}
A & b\\
c^T& d
\end{pmatrix}\begin{pmatrix}
Q_{11} \\
Q_{21}
\end{pmatrix}&=\begin{pmatrix}
Q_{11} \\
O
\end{pmatrix}\Lambda. 
\end{align}
% We can apply the eigendecomposition to $\Lambda$ in \eqref{cqf}. In fact, there exists a non-singular matrix $U\in\mathbb{C}^{(n-r)\times (n-r)}$ and a Jordan matrix $\Lambda\in\mathbb{C}^{(n-r)\times (n-r)}$ such that $\Lambda = U^{-1}\Lambda U$. In particular the diagonal elements in $\Lambda$ are eigenvalues of $\Lambda$. Thus there exists $\hat{Q}_1\in\mathbb{C}^{n\times (n-r)}$ and $\hat{Q}_2\in\mathbb{C}^{n-r}$ such that 
% \begin{align}\label{finitezero}
% \begin{pmatrix}
% A & b\\
% c^T& d
% \end{pmatrix}\begin{pmatrix}
% \hat{Q}_{1} \\
% \hat{Q}_{2}
% \end{pmatrix}&= \begin{pmatrix}
% \hat{Q}_{1} \\
% O
% \end{pmatrix}\Lambda
% \end{align}
We see from \eqref{eq:finitezero} that all eigenvalues of $\Lambda$ are invariant zeros of \eqref{fundSISO}. 

For the assertion \ref{a2}, it is sufficient to prove $\rank{N} = r$ because $N$ is a nilpotent. This follows from \eqref{sterm}. In fact, we obtain $\rank{I_n} = \rank{I_{n-r}} + \rank{N}$ from \eqref{sterm} because both $P$ and $Q$ are non-singular. 

For simplicity of the proof of the assertions \ref{a3} and \ref{a4}, we restrict the form of $N$. As we have already seen, the rank of $N$ is $r$. For any nilpotent $N$ with rank $r$, there exists a non-singular matrix $\hat{P}\in\mathbb{C}^{(r+1)\times (r+1)}$ such that 
\begin{align}\label{eq:nilpotent0}
\begin{pmatrix}
0 & 1 & & &0\\
0 & 0 & 1 & &\\
& & \ddots&\ddots &\\
& && \ddots & 1\\
0 & &&&0
\end{pmatrix} &= \hat{P}^{-1}N\hat{P}. 
\end{align}
This is obtained from the Jordan decomposition of $N$. Thus we set $N$ as the matrix in the left-hand side of \eqref{eq:nilpotent0} and will prove the assertions \ref{a3} and \ref{a4}.

For the assertion \ref{a3}, we suppose to the contrary that $r>0$. We obtain from \eqref{sterm} and \eqref{const} 
\begin{align}\label{eq:nilpotent}
\begin{pmatrix}
A & b\\
c^T & d
\end{pmatrix}\begin{pmatrix}
Q_{12} \\
Q_{22}
\end{pmatrix}N &= \begin{pmatrix}
Q_{12}\\
O
\end{pmatrix}. 
\end{align}
 We denote $Q_{12}$ and $Q_{22}$ by 
$Q_{12} = \begin{pmatrix}
q_1 & \cdots & q_{r+1}
\end{pmatrix}$ and $Q_{22} = \begin{pmatrix}
\hat{q}_1 & \cdots & \hat{q}_{r+1}
\end{pmatrix}$. Substituting $Q_{12}$ and $Q_{22}$ to \eqref{eq:nilpotent}, we then obtain
\begin{align}\label{eq:nilpotent2}
\begin{pmatrix}
A & b\\
c^T & d
\end{pmatrix}\begin{pmatrix}
0 &q_1 &\cdots &q_{r} \\
0 &\hat{q}_1 &\cdots &\hat{q}_{r}
\end{pmatrix} = \begin{pmatrix}
q_1 &q_2 &\cdots &q_{r+1} \\
0 & 0&\cdots &0
\end{pmatrix}
\end{align}
We see from this equation that $q_{1}=0$ and $\hat{q}_{1}d = 0$. Since we assumed $d\neq 0$, $\hat{q}_{1} = 0$. This contradicts the fact that $Q$ is non-singular. Thus $r=0$. In particular, we see that $Q_{12}=0$ and $Q_{22}\neq 0$ if $r=0$.  

For the assertion \ref{a4}, we focus on \eqref{eq:nilpotent2}. Then we obtain
\begin{align}\label{eq:nilpotent3}
  &\left\{
  \begin{array}{lcll}
  Aq_k +b\hat{q}_k&=& q_{k+1} & (k=1, 2, \ldots, r), \\
  c^Tq_k &=& 0 & (k=1, 2, \ldots, r), \\
  q_{1} &=& 0. &
  \end{array}
  \right. 
\end{align}
It follows from the third equation that we have $\hat{q}_{1}\neq 0$. Otherwise the matrix $Q$ has the zero column and thus $Q$ is singular. 

Next, eliminating $q_k$ at the left-hand side in the first equations of \eqref{eq:nilpotent3}, we obtain 
\begin{align}
  \label{eq:firsteq}
  \hat{q}_{1}A^{k-1}b + \hat{q}_{2}A^{k-2}b + \cdots + \hat{q}_{k}b&= q_{k+1}
\end{align}
for $k= 1, \ldots, r$. We can prove $c^Tb = 0$, $c^TAb = 0, \ldots, c^TA^{r-2}b = 0$ by using the induction on $k$ and the second equations of \eqref{eq:nilpotent3}.  

Finally, we prove $c^TA^{r-1}b \neq 0$. From \eqref{eq:firsteq}, we have $\hat{q}_{1}c^TA^{r-1}b = c^Tq_{r+1}$. If $c^Tq_{r+1} = 0$, then we obtain the contradiction. In fact, $c^TQ_{11} = 0$ and $c^TQ_{12} = 0$. We then obtain $\begin{pmatrix} c^T& 0\end{pmatrix}Q = 0$. Since $c\neq 0$, this contradicts to the fact that $Q$ is non-singular. 
\end{proof}

\begin{remark}\label{remark1}
We can summarize Theorem \ref{Weierstrass} as follows: There exist a non-singular matrix $Q\in\mathbb{C}^{(n+1)\times (n+1)}$, a nonnegative integer $r$, a Jordan matrix $\Lambda\in\mathbb{C}^{(n-r)\times (n-r)}$ and a nilpotent $N\in\mathbb{C}^{(r+1)\times (r+1)}$ such that 
\begin{align}\label{cqf2}
  \begin{pmatrix}
  A & b\\
  c^T & d
  \end{pmatrix} \begin{pmatrix}
  Q_{11} & Q_{12}\\
  Q_{21} & Q_{22}
  \end{pmatrix}\begin{pmatrix}
  I_{n-r} & O\\
  O & N
  \end{pmatrix}& = \begin{pmatrix}
  Q_{11} & Q_{12}\\
  O & O
  \end{pmatrix}\begin{pmatrix}
  \Lambda & O \\
  O & I_{r+1}
  \end{pmatrix}.
\end{align}
We can observe the following from \eqref{cqf2}. 
\begin{enumerate}
\item\label{b1} If $\lambda$ is an eigenvalue of $\Lambda$ in \eqref{cqf2} with the (algebraic) multiplicity $m$, then we call it {\itshape the invariant zero of $G$ with the multiplicity $m$}. We can see that \eqref{fundSISO} has $(n-r)$ invariant zeros including their multiplicity.  

In analogy to the multiplicity, we can define the geometric multiplicity $m_{\mathrm g}$ of the invariant zero. If both multiplicity do not coincide, then we {\itshape cannot} select $m$ linearly independent right null vectors associated with the invariant zero. We, however, can define {\itshape the generalized right null vectors} $\left(
\begin{smallmatrix}
q_{m_{\mathrm g}+1}\\
\hat{q}_{m_{\mathrm g}+1}
\end{smallmatrix}
\right), \ldots, \left(
\begin{smallmatrix}
q_{m}\\
\hat{q}_{m}
\end{smallmatrix}
\right)$ in a similar manner to the generalized eigenvectors. Then all the (generalized) right null vectors $\left(
\begin{smallmatrix}
q_{1}\\
\hat{q}_{1}
\end{smallmatrix}
\right), \ldots, \left(
\begin{smallmatrix}
q_{m}\\
\hat{q}_{m}
\end{smallmatrix}
\right)$ are linearly independent. For instance, if $m>1$ and $m_{\mathrm g} =1$, then we have one right null vector and $(m-1)$ generalized right null vectors associated with the invariant zero $\lambda$ as follows. 
\[
\begin{pmatrix}
A & b\\
c^T& d
\end{pmatrix}\begin{pmatrix}
q_1\\
\hat{q}_1
\end{pmatrix} = \lambda \begin{pmatrix}
q_1\\
0
\end{pmatrix}, 
\begin{pmatrix}
A & b\\
c^T& d
\end{pmatrix}\begin{pmatrix}
q_k\\
\hat{q}_k
\end{pmatrix} = \lambda \begin{pmatrix}
q_k\\
0
\end{pmatrix} + \begin{pmatrix}
q_{k-1}\\
0
\end{pmatrix} \ (k=2, \ldots, m). 
\]
Throughout this paper, we refer to the generalized null vector as {\itshape the right null vector associated with the invariant zero $\lambda$} for brevity. We can see that the submatrix $\left(
\begin{smallmatrix}
Q_{11}\\
Q_{21}
\end{smallmatrix}
\right)$ obtained by collecting these $(n-r)$ right null vectors associated with all invariant zeros of $G$. 

\item\label{b2} Assertions \ref{a3} and \ref{a4} in Theorem \ref{Weierstrass} imply that the nonnegative integer $r$ is equal to the {\itshape relative degree of \eqref{fundSISO}}. In addition, wa say that $G$ has an infinite invariant zero if $r>0$. 
\item\label{b3} We see from the proof of Theorem \ref{Weierstrass} that we can take $\left(
\begin{smallmatrix}
0\\
1
\end{smallmatrix}
\right)$ as the first column of the submatrix $\left(
\begin{smallmatrix}
Q_{12}\\
Q_{22}
\end{smallmatrix}
\right)$. Thus $Q_{11}$ is of full column rank. Otherwise we obtain a contradiction to the fact that $Q$ is non-singular. 
\end{enumerate}

%We can obtain similar results to 3 and 4 of Lemma \ref{lem:zero} for the generalized null vectors. We give a proof in Appendix \ref{subapp:genvec34}.
%\begin{lemma}\label{lemma:genvec34}
%If $(A, c^T)$ in \eqref{fundSISO} is observable,  
%   then $\hat{q}_{k}\in\bbC$ \ $(k=m_{\mathrm g}+1, \ldots, m)$ are nonzero. 
%   Similarly, if $(A, c^T)$ in \eqref{fundSISO} is detectable and if 
%   $\lambda\in\bbC_+\cup\bbC_0$, 
%   then $\hat{q}_{k}\in\bbC$ \ $(k=m_{\mathrm g}+1, \ldots, m)$ are nonzero.
%\end{lemma}

We discussed the right generalized null vectors associated with the invariant zeros of $G$ from \eqref{cqf2}. Similarly, we can also introduce the left generalized null vectors associated with the invariant zeros of $G$ from \eqref{cqf}. For this, we consider the dual dynamical system of \eqref{fundSISO}.
\begin{align}\label{dualfundSISO}
G_{\scriptstyle \mathrm{d}}:\ 
&\left\{
\begin{array}{lll}
\dot{x}_{\mathrm{d}} & = & A^Tx_{\mathrm{d}} + cu_{\mathrm{d}}\\
y_{\mathrm{d}} &=& b^Tx_{\mathrm{d}} + du_{\mathrm{d}}, 
\end{array}
\right. 
\end{align}
Applying Theorem \ref{Weierstrass} to \eqref{dualfundSISO}, we then obtain the following Wierestrass form.
\begin{align*}
  \begin{pmatrix}
  A^T & c\\
  b^T & d
  \end{pmatrix} \begin{pmatrix}
  Q_{11} & Q_{21}\\
  Q_{12} & Q_{22}
  \end{pmatrix}\begin{pmatrix}
  I_{n-r} & O\\
  O & N
  \end{pmatrix}& = \begin{pmatrix}
  Q_{11} & Q_{12}\\
  O & O
  \end{pmatrix}\begin{pmatrix}
  \Lambda & O \\
  O & I_{r+1}
  \end{pmatrix}.
\end{align*}
By taking the transpose, then we obtain
\begin{align}\label{cqf3}
 \begin{pmatrix}
  I_{n-r} & O\\
  O & N^T
  \end{pmatrix}\begin{pmatrix}
  Q_{11}^T & Q_{21}^T\\
  Q_{12}^T & Q_{22}^T
  \end{pmatrix} \begin{pmatrix}
  A & b\\
  c^T & d
  \end{pmatrix} & = \begin{pmatrix}
  \Lambda^T & O \\
  O & I_{r+1}
  \end{pmatrix}\begin{pmatrix}
  Q_{11}^T & O \\
  Q_{12}^T & O
  \end{pmatrix}. 
\end{align}
In an analogous way to the generalized right null vector, we can define the generalized left null vector(s) from \eqref{cqf3}. We refer to those as left null vector(s) associated with an invariant zero as well as the case of the generalized right null vector(s) throughout this paper. 
\end{remark}
We have seen the flexibility in choosing $\left(\begin{smallmatrix}
Q_{12} \\ Q_{22}
\end{smallmatrix}
\right)$ in the proof of the assertions \ref{a3} and \ref{a4} in Theorem \ref{Weierstrass}. We give possible forms of the submatrix $\left(\begin{smallmatrix}
Q_{12} \\ Q_{22}
\end{smallmatrix}
 \right)$ of $Q$ in \eqref{cqf2} and \eqref{cqf3}. % in Lemma \ref{infiniteRvec}.
 %Before, we discuss the condition for the transfer function $G$ of \eqref{fundSISO} to be not identically. This fact can be proved by using Leverrier-Faddeev method in \cite{Hou98}.
% \begin{lemma}\label{idzero}
% $G(s)$ is identically zero if and only if $d = 0$ and 
% $c^TA^kb = 0$ for all $k=0, 1, \ldots, n-1$. 
% \end{lemma}
We give a proof in \ref{subapp:lemma2}. 

\begin{lemma}\label{lem:infinitevec}
Let $r$ be the relative degree of \eqref{fundSISO}. %Assume that the transfer function $G(s)$ of \eqref{fundSISO} is not identically zero. 
Then we can set the submatrix $\left(\begin{smallmatrix}
Q_{12} \\ Q_{22}
\end{smallmatrix}
\right)$ of $Q$ in \eqref{cqf2} as follows:
\begin{align}\label{infiniteRvec}
\begin{pmatrix}
Q_{12} \\ Q_{22}
\end{pmatrix} &= \left\{
\begin{array}{cl}
\begin{pmatrix}
0 \\
1 
\end{pmatrix}& \mbox{ if } r=0\\
\begin{pmatrix}
0 &b &Ab &\cdots &A^{r-2}b & A^{r-1}b \\
1 &0& 0& \cdots &0& 0
\end{pmatrix} & \mbox{ if } r > 0. 
\end{array}
\right.
\end{align}
Similarly, we can set the submatrix $\left(\begin{smallmatrix}
Q_{12} \\ Q_{22}
\end{smallmatrix}\right)$ of $Q$ in \eqref{cqf3} as follows:
\[
\begin{pmatrix}
Q_{12} \\ Q_{22}
\end{pmatrix} = \left\{
\begin{array}{cl}
\begin{pmatrix}
0 \\
1 
\end{pmatrix}& \mbox{ if } r=0\\
\begin{pmatrix}
0 &c &A^Tc &\cdots &(A^T)^{r-2}c &(A^T)^{r-1}c \\
1 &0& 0& \cdots & 0&0
\end{pmatrix} & \mbox{ if } r > 0. 
\end{array}
\right.
\]
\end{lemma}

\section{$H_{\infty}$ output feedback control problem and its reformulation}\label{sec:Hinf}

\subsection{Invariant zeros of $G_{zu}$ and $G_{yw}$}\label{subsec:zero}
Applying the Laplace transform to \eqref{SISO}, we obtain 
\begin{align}\label{transform}
\begin{pmatrix}
Z(s)\\
Y(s)
\end{pmatrix}& = G(s)\begin{pmatrix}
W(s)\\
U(s)
\end{pmatrix} =\begin{pmatrix}
G_{zw}(s) & G_{zu}(s)\\
G_{yw}(s) & G_{yu}(s)
\end{pmatrix}\begin{pmatrix}
W(s)\\
U(s)
\end{pmatrix}. 
\end{align}
Here we assume that $x(0) = 0$, and each element in $G(s)$ 
can be described by $A, b_j, c_j$ and $d_{ij}$ as follows: 
\begin{align}\label{Gall}
&\begin{array}{@{}l}
 G_{zw}(s) = c_1^T(s I_n - A)^{-1}b_1 + d_{11},\quad G_{zu}(s) = c_1^T(s I_n - A)^{-1}b_2 + d_{12}, \\
 G_{yw}(s) = c_2^T(s I_n - A)^{-1}b_1 + d_{21},\quad G_{yu}(s) = c_2^T(s I_n - A)^{-1}b_2. 
\end{array}
\end{align}

By following the elimination-of-variables method \cite{Gahinet94,
Iwasaki94},  
the $H_{\infty}$ control problem \eqref{Hinf} for the generalized plant $G$ given by \eqref{SISO} 
can be formulated as follows:  
\begin{align}\label{LMI}
&\left\{
\begin{array}{cl}
\displaystyle\inf_{\gamma, X, Y} & \gamma\\
\mbox{subject to} & -\begin{pmatrix}
b_2\\
d_{12}\\
0
\end{pmatrix}^{\perp T}\begin{pmatrix}
\He(AX) & Xc_1& b_1 \\
c_1^TX&-\gamma & d_{11}\\
b_1^T & d_{11} & -\gamma 
\end{pmatrix}\begin{pmatrix}
b_2\\
d_{12}\\
0
\end{pmatrix}^{\perp} \in\mathbb{S}_+^{n+1}, \\
& -\begin{pmatrix}
c_2\\
d_{21}\\
0
\end{pmatrix}^{\perp T}\begin{pmatrix}
\He(YA) & Yb_1 & c_1 \\
b_1^TY&-\gamma & d_{11}\\
c_1^T & d_{11} & -\gamma 
\end{pmatrix}\begin{pmatrix}
c_2\\
d_{21}\\
0
\end{pmatrix}^{\perp} \in\mathbb{S}_+^{n+1}, \\
& \begin{pmatrix}
X & -I_n\\
-I_n & Y
\end{pmatrix}\in\mathbb{S}^{2n}_+. 
\end{array}
\right.
\end{align}
It should be noted that suboptimal $H_\infty$ controllers can be
reconstructed by using suboptimal solutions of this SDP.  
The elimination-of-variables method is to solve the above SDP by primal-dual interior-point methods or the ellipsoid methods.  
In the following, we denote the infimal value of the SDP \eqref{LMI} by $\gamma^*$,  
which is consistent with the notation in \eqref{Hinf}. 

Before getting into the specific treatments of \eqref{LMI},  
we give a fundamental result on its duality.  
We prove this result in \ref{subapp:proof1} on the basis of 
the discussion in \ref{app:sdp}. 
\begin{theorem}\label{thm:duality}
The following statements are equivalent:
\begin{itemize}
 \item[(I)] $(A, b_2)$ in \eqref{SISO} is stabilizable and $(A, c_2^T)$ in
	  \eqref{SISO} is detectable.  
 \item[(II)] LMI problem \eqref{LMI} is strictly feasible, i.e., there exists a solution $(\gamma, X, Y)$ such that
\begin{align*}
-\begin{pmatrix}
b_2\\
d_{12}\\
0
\end{pmatrix}^{\perp T}\begin{pmatrix}
\He(AX) & Xc_1& b_1 \\
c_1^TX&-\gamma & d_{11}\\
b_1^T & d_{11} & -\gamma 
\end{pmatrix}\begin{pmatrix}
b_2\\
d_{12}\\
0
\end{pmatrix}^{\perp} &\in\mathbb{S}_{++}^{n+1}, \\
 -\begin{pmatrix}
c_2\\
d_{21}\\
0
\end{pmatrix}^{\perp T}\begin{pmatrix}
\He(YA) & Yb_1 & c_1 \\
b_1^TY&-\gamma & d_{11}\\
c_1^T & d_{11} & -\gamma 
\end{pmatrix}\begin{pmatrix}
c_2\\
d_{21}\\
0
\end{pmatrix}^{\perp} &\in\mathbb{S}_{++}^{n+1} \mbox{ and } \\
 \begin{pmatrix}
X & -I_n\\
-I_n & Y
\end{pmatrix}&\in\mathbb{S}^{2n}_{++}. 
\end{align*}
\end{itemize}
In particular, if \eqref{LMI} is strictly feasible, then it follows
from the strong duality theorem (in 
Theorem \ref{sdt}) in \ref{app:sdp} 
that the duality gap between \eqref{LMI} and its dual
is zero, and that the dual has an optimal solution.\end{theorem}

The key idea in working with the SDP \eqref{LMI} in this paper is the
treatments of the perpendicular matrices
\[
\begin{pmatrix}
b_2\\
d_{12}\\
0
\end{pmatrix}^{\perp},\  
\begin{pmatrix}
c_2\\
d_{21}\\
0
\end{pmatrix}^{\perp}. 
\]
These perpendicular matrices are usually constructed by directly working on the vectors
$(b_2^T\ d_{12}\ 0)^T$ and $(c_2^T\ d_{21}\ 0)^T$.  
In stark contrast, in this paper, 
we use the left and right null vectors associated with the invariant
zeros of the systems  
$G_{zu}$ and $G_{yw}$ given by \eqref{Gall}. 

Recall that $\lambda\in\mathbb{C}$ is an invariant zero of
$G_{zu}$ given by \eqref{Gall} if $\lambda$ is an eigenvalue of a Jordan matrix $\Lambda$ in the following Wierestrass form for $G_{zw}$
\begin{align}
  %\label{cqf_of_Gzw}
  \label{zeroGzu}
    \begin{pmatrix}
  I_{n-r_1} & O\\
  O & N_1^T
  \end{pmatrix}\begin{pmatrix}
  S^T & f\\
  \hat{S}^T & \hat{f}
  \end{pmatrix}\begin{pmatrix}
  A & b_2\\
  c_1^T & d_{12}
  \end{pmatrix}& = \begin{pmatrix}
  \Lambda^T & O \\
  O & I_{r_1+1}
  \end{pmatrix}\begin{pmatrix}
  S^T & O\\
  \hat{S}^T & O
  \end{pmatrix}.
\end{align}
Here $r$ is nonnegative integer, $S\in\mathbb{C}^{n\times (n-r_1)}$, $f\in\mathbb{C}^{n-r_1}$, $\hat{S}\in\mathbb{C}^{n\times (r_1+1)}$, $\hat{f}\in\mathbb{C}^{r_1+1}$ and $N_1\in\mathbb{C}^{(r_1+1)\times (r_1+1)}$ is a nilpotent. This is obtained from \eqref{cqf3}. In particular, it follows from Theorem \ref{Weierstrass} and Remark \ref{remark1} that $r_1$ is equal to the relative degree of the transfer function of $G_{zu}$ and the matrix $\left(\begin{smallmatrix}
\hat{S}\\
\hat{f}^T
\end{smallmatrix}
\right)$ can be set as follows.
\[
\begin{pmatrix}
\hat{S}\\
\hat{f}^T
\end{pmatrix} = \left\{
\begin{array}{cl}
\begin{pmatrix}
0\\
1
\end{pmatrix} & (r_1= 0)\\
 \begin{pmatrix}
0 &c &(A^T)c &\cdots &(A^T)^{r-2}c & (A^T)^{r-1}c \\
1 &0& 0& \cdots & 0&0
\end{pmatrix} & (r_1\neq 0)
\end{array}
\right. 
\]

Similarly, recall 
that $\omega\in\mathbb{C}$ is an invariant zero of
$G_{yw}$ given by \eqref{Gall} if $\omega$ is an eigenvalue of a Jordan matrix $\Omega$ in the following ``transposed version'' of 
 Wierestrass form for $G_{yw}$
\begin{align}\label{tgdef}
 \begin{pmatrix}
  I_{n-r_2} & O\\
  O & N_2^T
  \end{pmatrix}\begin{pmatrix}
  T^T & g\\
  \hat{T}^T & \hat{g}
  \end{pmatrix}
  \begin{pmatrix}
  A^T & c_1\\
  b_2^T & d_{21}
  \end{pmatrix}  & = \begin{pmatrix}
  \Omega^T & O \\
  O & I_{r_2+1}
  \end{pmatrix}\begin{pmatrix}
  T^T & O\\
  \hat{T}^T & O
  \end{pmatrix}.
\end{align}
This is obtained from \eqref{cqf2}. Symbols $T$, $\hat{T}$, $g$, $\hat{g}$, $N_2$ and $\Omega$ in \eqref{tgdef} can be defined in a similar manner to invariant zeros of $G_{zu}$. In particular, $\left(
\begin{smallmatrix}
\hat{T}\\
\hat{g}^T
\end{smallmatrix}
\right)$ can be set as follows. 
\[
\begin{pmatrix}
\hat{T}\\
\hat{g}^T
\end{pmatrix} = \left\{
\begin{array}{cl}
\begin{pmatrix}
0\\
1
\end{pmatrix} & (r_2= 0)\\
 \begin{pmatrix}
0 &b &Ab &\cdots &A^{r-2}b &A^{r-1}b\\
1 &0& 0& \cdots & 0&0
\end{pmatrix} & (r_2\neq 0)
\end{array}
\right. 
\]
We use \eqref{zeroGzu} and \eqref{tgdef} for concise descriptions of LMI problems. Throughout the paper, we make the following assumptions 
on the generalized plant $G$ given by \eqref{SISO} and \eqref{Gall}:  
\begin{assumption}\label{A}
\begin{enumerate}
\item[(a)]\label{A1} $(A, b_2)$ is stabilizable and $(A, c_2^T)$ is detectable. 
\item[(b)]\label{A2} All invariant zeros of $G_{zu}$ are real. Similarly for $G_{yw}$. 
%\item[(c)]\label{A3} All invariant zeros of $G_{zu}$ are distinct.
%   Similarly for $G_{yw}$.
\end{enumerate}
\end{assumption}

Assumption \ref{A}-(a) is quite natural in controller design since otherwise, there are no internally stabilizing controllers.  
On the other hand, 
we proceed with the discussions in each section under Assumption \ref{A}-(b), but this is just for simplicity.  
Assumption \ref{A}-(b) implies that $\left(\begin{smallmatrix}S\\ f^T\end{smallmatrix}\right)$ in
\eqref{zeroGzu} and $\left(\begin{smallmatrix}T\\ g^T\end{smallmatrix}\right)$ in \eqref{tgdef} are real and hence
facilitates the descriptions of matrix inequality conditions. 
%Assumption \ref{A}-(c) is also imposed for simplicity. 
%We can remove this assumption, but the discussions become more complicated and longer. For this, we add this assumption. 
%In each section that follows, we remove these assumptions 
%after stating the main results.  
%Finally, Assumption \ref{A}-(c) seems to be necessary to 
%ensure the full-columnness of the null vectors associated with
%the invariant zeros of $G_{zu}$ and $G_{yw}$.  
Under Assumption \ref{A} and from Theorem \ref{Weierstrass} and Remark \ref{remark1}, we can readily obtain 
the next results with respect to 
the invariant zeros of $G_{zu}$ and $G_{yw}$.  
\begin{lemma}\label{lem:gplant}
\begin{enumerate}
\item\label{F9} Let $r_1\ (0\le r_1\le n)$ be the relative degree
   of $G_{zu}$. 
   Then $G_{zu}$ has $(n-r_1)$ real  invariant zero(s). 
   % with their multiplicity.  
   The left null vectors $\left(
   \begin{smallmatrix}
   S\\
   f^T
   \end{smallmatrix}
   \right)$ associated with all the invariant zeros of $G_{zu}$ is of full column rank. In particular, $S$ is also of full column rank. Furthermore, if $d_{12}\neq 0$, then $S$ is square and thus is non-singular. 
   %If we denote the left null vector associated with the invariant zero
   %$\lambda_i$ by $(s_i^T\ f_i)\in\bbR^{1\times (n+1)}$ for $i=1, \ldots, n-r_1$, 
   %then $(n-r_1)$ vectors $s_1, \ldots, s_{n-r_1}$ are linearly independent. 
\item\label{F10} 
Let $r_2\ (0\le r_2\le n)$ be the relative degree
   of $G_{yw}$. 
   Then $G_{yw}$ has $(n-r_2)$ real invariant zero(s).  
   % with their multiplicity.  
   The right null vectors $\left(
   \begin{smallmatrix}
   T\\
   g^T
   \end{smallmatrix}
   \right)$ associated with all the invariant zeros of $G_{yw}$ is of full column rank. In particular, $T$ is also of full column rank. Furthermore, if $d_{21}\neq 0$, then $T$ is square and thus is non-singular. 
%Let $r_2\ (0\le r_2\le n)$ be the relative degree
%   of $G_{yw}$. 
%   Then $G_{yw}$ has $n-r_2$ real distinct invariant zero(s) 
%   $\omega_i\ (i=1,\ldots,n-r_2)$.  
%   If we denote the right null vector associated with the invariant zero
%   $\omega_i$ by $(t_i^T\ g_i)^T\in\bbR^{n+1}$ for $i=1, \ldots, n-r_2$, 
%   then $(n-r_2)$ vectors $t_1, \ldots, t_{n-r_2}$ are linearly independent. 
\end{enumerate}
\end{lemma}

Finally, we partition $\Lambda$ and $\Omega$ into the parts of 
stable and unstable invariant zeros as follows:
\begin{align}\label{Partition}
\Lambda &= \begin{pmatrix}
\Lambda_- &* \\
O_{k_1 \times (n-r_1-k_1)}& \Lambda_+
\end{pmatrix} \mbox{ and } \Omega = \begin{pmatrix}
\Omega_- &*\\
O_{k_2 \times (n-r_2-k_2)} & \Omega_+
\end{pmatrix}.  
\end{align}
Here, $k_1$ (resp. $k_2$) denotes the number of unstable invariant zeros of 
$G_{zu}$ (resp. $G_{yw}$), 
$\Lambda_+$ (resp. $\Omega_+$) is a Jordan matrix with unstable invariant 
zeros of $G_{zu}$ (resp. $G_{yw}$), and 
$\Lambda_-$ (resp. $\Omega_-$) is a Jordan matrix with stable invariant 
zeros of $G_{zu}$ (resp. $G_{yw}$).  
We also partition $S, f$ and $T, g$ conformably as follows
\begin{align}\label{Partition.vec}
\begin{pmatrix}
S^T_\circ& f_\circ
\end{pmatrix}\begin{pmatrix}
A & b_2\\
c_1^T & d_{12}
\end{pmatrix} = \Lambda_\circ^T\begin{pmatrix}
S^T_\circ& 0
\end{pmatrix} \mbox{ and }\begin{pmatrix}
T^T_\circ& g_\circ
\end{pmatrix}\begin{pmatrix}
A^T & c_2\\
b_1^T & d_{21}
\end{pmatrix} = \Omega_\circ^T\begin{pmatrix}
T^T_\circ& 0
\end{pmatrix}
\end{align}
where $\circ = +$ or $-$. 
When we deal with invariant zeros on the imaginary axis explicitly, 
we use another partition of $\Lambda$ and $\Omega$, 
see Section \ref{sec:case3}. 

\subsection{Simplification of the LMI problem via invariant zeros and null vectors}\label{subsec:reduction}

We assume that $d_{12}\neq 0$ and $d_{21}\neq 0$ in this subsection.  
Since this is also assumed in Sections \ref{sec:case1}, \ref{sec:case2}
and \ref{sec:case3}, 
the simplification in this subsection is valid except for Section
\ref{sec:case4}. 

If $d_{12}\neq 0$ and $d_{21}\neq 0$, 
it follows from Lemma \ref{lem:gplant} that 
both $S$ and $T$ in \eqref{zeroGzu} and \eqref{tgdef} are non-singular. 
Using this property, we have
\[
\begin{pmatrix}
b_2\\
d_{12}\\
0
\end{pmatrix}^{\perp} = \begin{pmatrix}
S & 0\\
f^T & 0\\
0^T & 1
\end{pmatrix} \mbox{ and }\begin{pmatrix}
c_2\\
d_{21}\\
0
\end{pmatrix}^{\perp} = \begin{pmatrix}
T & 0\\
g^T & 0\\
0^T & 1
\end{pmatrix}. 
\]
By using this fact, we have 
\begin{align*}
&-\begin{pmatrix}
b_2\\
d_{12}\\
0
\end{pmatrix}^{\perp T}\begin{pmatrix}
\He(AX) & Xc_1 & b_1 \\
c_1^TX&-\gamma &d_{11} \\
b_1^T & d_{11} & -\gamma 
\end{pmatrix}\begin{pmatrix}
b_2\\
d_{12}\\
0
\end{pmatrix}^{\perp}\\ %&=-\begin{pmatrix}
%S & 0\\
%f^T & 0\\
%0^T & 1
%\end{pmatrix}^T\begin{pmatrix}
%\He(AX) & Xc_1 & b_1 \\
%c_1^TX&-\gamma & d_{11}\\
%b_1^T & d_{11} & -\gamma 
%\end{pmatrix} \begin{pmatrix}
%S & 0\\
%f^T & 0\\
%0^T & 1
%\end{pmatrix}\\
%=& - \begin{pmatrix}
%\He((S^TA+fc_1^T)XS) -\gamma ff^T & *\\
%b_1^TS + d_{11}f^T & -\gamma
%\end{pmatrix}\\
=& - \begin{pmatrix}
\He(\Lambda^T S^TXS) -\gamma ff^T & *\\
b_1^TS + d_{11}f^T & -\gamma
\end{pmatrix}, \\
&-\begin{pmatrix}
c_2\\
d_{21}\\
0
\end{pmatrix}^{\perp T}\begin{pmatrix}
\He(YA) & Yb_1 & c_1 \\
b_1^TY&-\gamma & d_{11}\\
c_1^T & d_{11} & -\gamma 
\end{pmatrix}\begin{pmatrix}
c_2\\
d_{21}\\
0
\end{pmatrix}^{\perp}\\%&= -\begin{pmatrix}
%T & 0\\
%g^T & 0\\
%0^T & 1
%\end{pmatrix}^T
%\begin{pmatrix}
%\He(YA) & Yb_1 & c_1 \\
%b_1^TY&-\gamma &d_{11} \\
%c_1^T & d_{11} & -\gamma 
%\end{pmatrix}
%\begin{pmatrix}
%T & 0\\
%g^T & 0\\
%0^T & 1
%\end{pmatrix}\\
%=& - \begin{pmatrix}
%\He((T^TA^T+gb_1^T)YT) -\gamma gg^T & *\\
%c_1^TT + d_{11}g^T & -\gamma
%\end{pmatrix}\\
=& -\begin{pmatrix}
\He(\Omega^TT^TYT) -\gamma gg^T & *\\
c_1^TT + d_{11}g^T & -\gamma
\end{pmatrix}, 
\end{align*}
where $*$ indicates the transpose of the lower triangular part. In addition, since $S$ and $T$ are non-singular, we have 
\[
\begin{pmatrix}
X & -I_n\\
-I_n & Y
\end{pmatrix}\in\mathbb{S}^{2n}_+ \iff \begin{pmatrix}
S^TXS & -S^TT\\
-T^TS & T^TYT
\end{pmatrix}\in\mathbb{S}^{2n}_+. 
\]
Hence, by replacing $S^TXS$ and $T^TYT$ by $\hat{X}$ and $\hat{Y}$,
respectively, the SDP \eqref{LMI} can be reduced to
\begin{align}\label{LMI2}
&\left\{
\begin{array}{cl}
\displaystyle\inf_{\gamma, \hat{X}, \hat{Y}} & \gamma\\
\mbox{subject to} & - \begin{pmatrix}
\He(\Lambda^T \hat{X}) -\gamma ff^T & *\\
h_1^T & -\gamma
\end{pmatrix} \in\mathbb{S}_+^{n+1}, \begin{pmatrix}
\hat{X} & *\\
-J& \hat{Y}
\end{pmatrix}\in\mathbb{S}^{2n}_+,\\
& - \begin{pmatrix}
\He(\Omega^T \hat{Y}) -\gamma gg^T & *\\
h_2^T & -\gamma
\end{pmatrix} \in\mathbb{S}_+^{n+1}, 
\end{array}
\right.
\end{align}
where $J:= T^TS$, $h_1:=S^Tb_1+d_{11}f$ and $h_2:= T^Tc_1+d_{11}g$. %Here $*$ stands for the transpose of the lower part of the matrix. 

In Sections \ref{sec:case1}, \ref{sec:case2}, \ref{sec:case3} and \ref{sec:case4}, 
we analyze the infimal value $\gamma^*$ given by \eqref{LMI2} separately 
according to the following four cases: 
\begin{description}
\item[Case 1] (Section \ref{sec:case1}) Both $d_{12}$ and $d_{21}$ are nonzero, 
   and all the invariant zeros of $G_{zu}$ and $G_{yw}$ are unstable, 
   but not on the imaginary axis. 
\item[Case 2] (Section \ref{sec:case2}) Both $d_{12}$ and $d_{21}$ are nonzero, 
   and at least one of the invariant zeros of $G_{zu}$ or
   $G_{yw}$ is stable, 
   but all the unstable invariant zeros are not on the imaginary axis. 
\item[Case 3] (Section \ref{sec:case3}) Both $d_{12}$ and $d_{21}$ are nonzero, 
    and at least one of the invariant zeros in $G_{zu}$ or
    $G_{yw}$ 
    exists on the imaginary axis. 
\item[Case 4] (Section \ref{sec:case4}) At least one of $d_{12}$ and $d_{21}$ is
   zero. 
   In this case, an infinite invariant zero exists in $G_{zu}$ or $G_{yw}$. 
\end{description}

\section{Analysis of Case 1}\label{sec:case1}
In this section, we assume that all invariant zeros of $G_{zu}$ and
$G_{uw}$ given by \eqref{Gall} are (strictly) positive. 
This is represented equivalently by $\Lambda = \Lambda_+$ and $\Omega = \Omega_+$. 
Then we have $S = S_+, f = f_+$ and $T=T_+$, $g=g_+$, and thus 
$J = J_+ := T_+^TS_+$, 
$h_1 = h_{1+} := S_+^Tb_1+d_{11}f_+$ and 
$h_2 = h_{2+}: = T_+^Tc_1+d_{11}g_+$. 

Under Assumption \ref{A}, 
Theorem \ref{thm:duality} ensures the existence of an optimal solution 
to the dual of \eqref{LMI2}, 
while the following lemma ensures the existence of 
an optimal solution to \eqref{LMI2}. 
We give a proof of Lemma \ref{lem:exist} in 
\ref{subapp:proof2}. 
\begin{lemma}\label{lem:exist}
LMI problem \eqref{LMI2} has an optimal solution. 
\end{lemma}

In addition, we can obtain the next result with respect to LMI problem 
\eqref{LMI2}. We give a proof in \ref{subapp:lemma5}. 
\begin{lemma}\label{lem:zerocase1} 
The optimal value $\gamma^*$ of LMI problem \eqref{LMI2} is zero 
if and only if $h_{1+} = 0$, $h_{2+} = 0$ and $J_+ = O_{n\times n}$.  
\end{lemma}

Since we have already dealt with the case $\gamma^*=0$ explicitly in Lemma \ref{lem:zerocase1}, 
we assume that the optimal value $\gamma^*$ of \eqref{LMI2} is 
(strictly) positive in the remainder of this section. 
Then we can apply the Schur complement to the first and second LMIs in
\eqref{LMI2} 
and obtain 
\begin{align*}
\He((-\Lambda_+)^T X) + \gamma f_+f_+^T- h_{1+}h_{1+}^T/\gamma &\in\mathbb{S}^{n}_+ \mbox{ and }\\
\He((-\Omega_+)^T Y) +\gamma g_+g_+^T- h_{2+}h_{2+}^T/\gamma &\in\mathbb{S}^{n}_+ 
\end{align*}
for $\gamma >0$.  
It follows that LMI problem \eqref{LMI2} can be reformulated as
\begin{align}\label{LMI3}
&\left\{
\begin{array}{cl}
\displaystyle\inf_{\gamma, X, Y, \tilde{X}, \tilde{Y}} & \gamma\\
\mbox{subject to} 
& \He((-\Lambda_+)^T X) + 
\gamma f_+f_+^T- h_{1+}h_{1+}^T/\gamma -\tilde{X} = O_n, 
\tilde{X}\in\mathbb{S}^n_+, \\
& \He((-\Omega_+)^T Y) + 
\gamma g_+g_+^T- h_{2+}h_{2+}^T/\gamma -\tilde{Y} = O_n, 
\tilde{Y}\in\mathbb{S}^n_+, \\
& \begin{pmatrix}
X & -J_+^T\\
-J_+& Y
\end{pmatrix}\in\mathbb{S}^{2n}_+,\quad  
\gamma >0.
\end{array}
\right.
\end{align}
Since both $(-\Lambda_+)^T$ and $(-\Omega_+)^T$ are Hurwitz stable,
and since the first and second equalities can be seen as the Lyapunov
equations, 
we can solve them explicitly as follows:
\begin{align}
\label{Xhat}X &= 
\int_{0}^{\infty}\exp(-\Lambda_+^T t)
\left( \gamma f_+f_+^T- h_{1+}h_{1+}^T/\gamma 
-\tilde{X}\right)\exp(-\Lambda_+ t)\, dt, \\
\label{Yhat}Y &= 
\int_{0}^{\infty}\exp(-\Omega_+^T t)
\left( \gamma g_+g_+^T- h_{2+}h_{2+}^T/\gamma 
-\tilde{Y}\right)\exp(-\Omega_+ t)\, dt. 
\end{align}
See e.g. \cite[Remark, page 78]{Boyd94} for this explicit form of the Lyapunov
equations. 
In relation to \eqref{Xhat} and \eqref{Yhat}, 
we define $F_+, G_+, H_{1+}, H_{2+}\in\mathbb{S}^n$ by 
\begin{align}\label{FGH}
&\left\{
\begin{array}{ccl}
F_+ &=& \displaystyle\int_{0}^{\infty}\exp(-\Lambda_+^T t)
f_+f_+^T\exp(-\Lambda_+ t)\, dt, \\
G_+ &=& \displaystyle\int_{0}^{\infty}\exp(-\Omega_+^T t)
g_+g_+^T\exp(-\Omega_+ t)\, dt, \\
H_{1+} &=& \displaystyle\int_{0}^{\infty}\exp(-\Lambda_+^T t)
h_{1+}h_{1+}^T\exp(-\Lambda_+ t)\, dt, \\
H_{2+} &=& \displaystyle\int_{0}^{\infty}\exp(-\Omega^T_+ t)
h_{2+}h_{2+}^T\exp(-\Omega_+ t)\, dt. 
\end{array}
\right. 
\end{align}
We remark that $F_+, G_+, H_{1+}, H_{2+}$ are positive semidefinite. 
In particular, $F_+$ and $G_+$ are positive definite because 
both pairs $(\Lambda_+, f_+)$ and $(\Omega_+, g_+)$ are 
controllable under Assumption \ref{A1}. 
This is proved in Lemma \ref{lem:cntobs}
of \ref{subapp:proof3}. 

By using $F_+, G_+, H_{1+}, H_{2+}\in\mathbb{S}^n$, 
we can rewrite \eqref{LMI3} as
\begin{align}\label{LMI4}
&\left\{
\begin{array}{cl}
\displaystyle\inf_{\gamma, \tilde{X}, 
\tilde{Y}, \bar{X}, \bar{Y}} & \gamma\\
\mbox{subject to} 
& \hat{X} = \displaystyle\int_0^{\infty}
\exp(-\Lambda_+^T t)\tilde{X}\exp(-\Lambda_+ t)\, dt, \tilde{X}\in\mathbb{S}^n_+,\\
&\hat{Y} = \displaystyle\int_0^{\infty}
\exp(-\Omega_+^T t)\tilde{Y}\exp(-\Omega_+ t)\, dt, \tilde{Y}\in\mathbb{S}^n_+, \\
& \begin{pmatrix}
\gamma F_+ -\displaystyle\frac{1}{\gamma}H_{1+} -\hat{X} & -J_+^T\\
-J_+& \gamma G_+ -\displaystyle\frac{1}{\gamma}H_{2+} -\hat{Y}
\end{pmatrix}\in\mathbb{S}^{2n}_+, 
\gamma >0.  
\end{array}
\right.
\end{align}
Then we can readily prove that \eqref{LMI4} is 
equivalent to the next LMI problem :
\begin{align}\label{LMI5}
&\displaystyle\inf%_{\gamma}
\left\{ \gamma : \begin{pmatrix}
\gamma F_+ -\displaystyle\frac{1}{\gamma}H_{1+} & -J_+^T\\
-J_+& \gamma G_+ -\displaystyle\frac{1}{\gamma}H_{2+} 
\end{pmatrix}\in\mathbb{S}^{2n}_+, \gamma >0
\right\}. 
\end{align}
In fact, it is clear that if \eqref{LMI4} with the objective value $\gamma=\gamma_0$
is feasible by $(\gamma_0, \tilde{X}, \tilde{Y}, \hat{X},\hat{Y})$, 
then \eqref{LMI5} with $\gamma=\gamma_0$ 
is also feasible since $\hat{X}$ and $\hat{Y}$ 
are both positive semidefinite. 
On the other hand, if \eqref{LMI5} is feasible with the objective value 
$\gamma=\gamma_0$, then \eqref{LMI4} with $\gamma=\gamma_0$
is also feasible by
$(\gamma_0, \tilde{X}, \tilde{Y}, \hat{X}, \hat{Y})=(\gamma_0, O_n, O_n, O_n,O_n)$.  

To summarize the results in this section, 
we arrive at the next theorem 
that is the first main result of this paper.  
\begin{theorem}\label{thm:case1}
Let us consider Case \ref{Case1} stated at the final part of Section \ref{sec:Hinf}. 
Then the optimal value $\gamma^*$ of LMI problem \eqref{LMI2} is equal 
to the maximum eigenvalue of the matrix $E\in\mathbb{S}^{4n}$ defined by
\begin{align}\label{gammacase1}
E&:= \begin{pmatrix}
O& F_+^{-1/2}J_+^TG_+^{1/2} & F_+^{-1/2}H^{1/2}_{1+} & O \\
G_+^{-1/2}J_+F_+^{-1/2} & O& O &G_+^{-1/2}H^{1/2}_{2+}\\
H^{1/2}_{1+}F^{-1/2}_+ & O & O &O \\
O & H^{1/2}_{2+}G_+^{-1/2} & O & O
\end{pmatrix}. 
\end{align}
Here, $F_+, G_+\in\bbS_{++}^n$ and $H_{1+}, H_{2+}\in\bbS_+^n$ are given by
 \eqref{FGH}.  
Moreover, an optimal solution $(\gamma, X, Y)$ of the SDP
\eqref{LMI2} can be given explicitly by 
\[
\gamma = \gamma^*,\quad 
X = \gamma F_+ - \frac{1}{\gamma}H_{1+},\quad 
Y = \gamma G_+ - \frac{1}{\gamma}H_{2+}. 
\]
\end{theorem}
\begin{proof}
To prove $\gamma^* = \hat{\gamma}:=\lambda_{\max}(E)$, 
we use the fact that $F_+$ and $G_+$ are positive definite. 
By using the Schur complement, we have 
\begin{align*}
&\begin{pmatrix}
\gamma F_+ -\displaystyle\frac{1}{\gamma}H_{1+} & -J_+^T\\
-J_+& \gamma G_+ -\displaystyle\frac{1}{\gamma}H_{2+} 
\end{pmatrix}\in\mathbb{S}^{2n}_+ \\
&\iff 
\begin{pmatrix}
\gamma I_n - \displaystyle\frac{1}{\gamma}F_{+}^{-1/2} H_{1+}F_{+}^{-1/2}& *\\ %-F_+^{-1/2}J_+^TG_+^{-1/2}\\
-G_+^{-1/2}J_+F_{+}^{-1/2} & \gamma I_n - \displaystyle\frac{1}{\gamma}G_{+}^{-1/2} H_{2+}G_{+}^{-1/2}
\end{pmatrix}\in\mathbb{S}^{2n}_+, \\
&\iff \begin{pmatrix}
\gamma I_n & *& * & * \\ %-F_+^{-1/2}J_+^TG_+^{1/2} & -F_+^{-1/2}H^{1/2}_{1+} & O \\
-G_+^{-1/2}J_+F_+^{-1/2} & \gamma I_n & *&*\\%O &-G_+^{-1/2}H^{1/2}_{2+}\\
-H_{1+}^{1/2}F_+^{-1/2} & O & \gamma I_n &*\\%O \\
O & -H^{1/2}_{2+}G_+^{-1/2} & O & \gamma I_n
\end{pmatrix}\in\mathbb{S}^{4n}_+ \iff \gamma \ge \hat{\gamma} = \lambda_{\max}(E). 
%\lambda_{\max}\begin{pmatrix}
%O& F_+^{-1/2}J_+^TG_+^{1/2} & F_+^{-1/2}H^{1/2}_{1+} & O \\
%G_+^{-1/2}J_+F_+^{-1/2} & O& O &G_+^{-1/2}H^{1/2}_{2+}\\
%H^{1/2}_{1+}F^{-1/2}_+ & O & O &O \\
%O & H^{1/2}_{2+}G_+^{-1/2} & O & O
%\end{pmatrix}.  
\end{align*}
We note that the maximum eigenvalue of $E$ is nonnegative since $E$ is indefinite. 
%
%\[
%-E = P^{-1}EP \mbox{ where } P = \diag(I_n, -I_n, -I_n, I_n). 
%\]
%
%This implies that $-\lambda$ is also an eigenvalue of $E$ 
%if $\lambda$ is. 
Hence the maximum eigenvalue of $E$ is nonnegative, 
and thus we conclude that the optimal value $\gamma^*$ of
\eqref{LMI2} is equal to the maximum eigenvalue of $E$.  
In addition, since \eqref{LMI5} is equivalent to \eqref{LMI2}, 
an optimal solution of \eqref{LMI2} is obtained from \eqref{Xhat} and \eqref{Yhat}.  
Therefore we obtain the result.
\end{proof}

Before closing this section, 
we provide an explicit way to compute the matrix $E$ in Theorem \ref{thm:case1}. 
First, we compute all the invariant zeros 
$\lambda$ and $\omega$ of $G_{zu}$ and $G_{yw}$ and
their null vectors $\left(
\begin{smallmatrix}
S\\
f^T
\end{smallmatrix}
\right)$ and $\left(
\begin{smallmatrix}
T\\
g^T
\end{smallmatrix}
\right)$ in \eqref{zeroGzu} and
\eqref{tgdef}, respectively. 
Collecting them, we define 
$\Lambda$, $\Omega$, $S$, $f$, $T$ and $g$ as in 
\eqref{zeroGzu} and \eqref{tgdef}. 
%It follows from \ref{F2} of Lemma \ref{lem:zero} 
%that we may set $f = g = (1, \ldots, 1)^T\in\mathbb{R}^n$ 
%under the assumption that all invariant zeros on $G_{zu}$ and $G_{yw}$
%are unstable, but not on the imaginary axis. 
We remark that we have $\Lambda_+=\Lambda$, $\Omega_+=\Omega$, $S_+=S$,
$f_+=f$, $T_+=T$ and $g_+=g$ under this assumption. 
Second, we compute $J_+$, $h_{1+}$ and $h_{2+}$ by
\[
J_+ = T^T_+S_+,\quad h_{1+} = S_+^Tb_1+d_{11}f_+,\quad h_{2+} = T_+^Tc_1 + d_{11} g_+. 
\]
Next, we solve the following Lyapunov equation 
to determine the the symmetric matrix $F_+$: 
\[
(-\Lambda_+^T)F_+ +F_+(-\Lambda_+) = -f_+f_+^T. 
\]
As $(\Lambda_+, f_+)$ is controllable, 
the solution of the above Lyapunov equation is positive definite.  
Similarly, we solve the following Lyapunov equations 
to determine the symmetric matrices $G_+$, $H_{1+}$ and $H_{2+}$,
respectively:
\begin{align*}
(-\Omega_+^T)G_+ +G_+(-\Omega_+) &= -g_+g_+^T, \\ 
(-\Lambda_+^T)H_{1+} +H_{1+}(-\Lambda_+) &= -h_{1+}h_{1+}^T, \\
(-\Omega_+^T)H_{2+} +H_{2+}(-\Omega_+) &= -h_{2+}h_{2+}^T.  
\end{align*}
Finally, we compute $F_+^{-1/2}$, $G_+^{-1/2}$, $H_{1+}^{1/2}$ and
$H_{2+}^{1/2}$ by eigenvalue decomposition. 
Then we can obtain the matrix $E$ in \eqref{gammacase1}.

\section{Analysis of Case 2}\label{sec:case2}
When dealing with Case 2 stated at the final part of Section \ref{sec:Hinf}, 
we cannot obtain Theorem \ref{thm:case1} by a similar discussion to 
Section \ref{sec:case1}. 
The difficulty lies in the fact that we cannot represent the solutions of the Lyapunov equations with respect to $\Lambda$ and $\Omega$ since they contain negative eigenvalues in Case 2. 
However, we can overcome this difficulty by investigating the structure of feasible solutions of the dual of \eqref{LMI2}. 

The following lemma provides the mathematical formulation of the dual of \eqref{LMI2}. We give a proof in \ref{subapp:lemma6}. 
\begin{lemma}\label{lem:dual}
The dual of \eqref{LMI2} can be formulated as follows:
\begin{align}\label{dual2}
&\left\{
\begin{array}{cl}
\displaystyle\sup_{W_{ij}, Z_{ij}, V_{ij}} 
& 2(h_1^T\bullet Z_{21} + h_2^T\bullet V_{21} + J\bullet W_{21})\\
\mbox{subject to} & f^TZ_{11}f + Z_{22} + g^TV_{11}g + V_{22} = 1, \\
& W_{11} = \He(\Lambda Z_{11}), W_{22} = \He(\Omega V_{11}), \\
& \begin{pmatrix}
Z_{11} & Z_{21}^T\\
Z_{21} & Z_{22}
\end{pmatrix}\in\mathbb{S}^{n+1}_+, \begin{pmatrix}
V_{11} & V_{21}^T\\
V_{21} & V_{22}
\end{pmatrix}\in\mathbb{S}^{n+1}_+, \begin{pmatrix}
W_{11} & W_{21}^T\\
W_{21} & W_{22}
\end{pmatrix}\in\mathbb{S}^{2n}_+. 
\end{array}
\right.
\end{align}
Moreover, the duality gap between \eqref{LMI2} and \eqref{dual2} is zero, and \eqref{dual2} has an optimal solution. 
\end{lemma}

The following lemma provides the structure of solutions of \eqref{dual2}. For this, we partition $Z_{11}$, $V_{11}$ and $W$ of a feasible solution $(Z, V, W)$ of \eqref{dual2} as follows:
\begin{align*}
Z_{11} &= \bordermatrix{
&(n-k_1) & k_1\cr
(n-k_1) & Z_{11}^1 & (Z_{11}^2)^T\cr
k_1 &Z_{11}^2&Z_{11}^3\cr
}, V_{11} = \bordermatrix{
&(n-k_2) & k_2\cr
(n-k_2) & V_{11}^1 & (V_{11}^2)^T\cr
k_2 &V_{11}^2&V_{11}^3\cr
}, \\
W_{11} &= \bordermatrix{
&(n-k_1) & k_1\cr
(n-k_1) & W_{11}^1 & (W_{11}^2)^T\cr
k_1 &W_{11}^2&W_{11}^3\cr
} \mbox{ and }  
W_{22} = \bordermatrix{
&(n-k_2) & k_2\cr
(n-k_2) & W_{22}^1 & (W_{22}^2)^T\cr
k_2 &W_{22}^2&W_{22}^3\cr
}. 
\end{align*}
\begin{lemma}\label{lem:dstructure}
Any feasible solution $(Z, V, W)$ of \eqref{dual2} has the form of  
\begin{align*}
W_{11} &= \begin{pmatrix}
O_{(n-k_1)\times (n-k_1)} & O_{k_1\times (n-k_1)}\\
O_{(n-k_1)\times k_1} & W_{11}^3
\end{pmatrix}, Z_{11} = \begin{pmatrix}
O_{(n-k_1)\times (n-k_1)} & O_{k_1\times (n-k_1)}\\
O_{(n-k_1)\times k_1} & Z_{11}^3
\end{pmatrix}, \\
W_{22} &= \begin{pmatrix}
O_{(n-k_2)\times (n-k_2)} & O_{k_2\times (n-k_2)}\\
O_{(n-k_2)\times k_2} & W_{22}^3
\end{pmatrix}, V_{11} = \begin{pmatrix}
O_{(n-k_2)\times (n-k_2)} & O_{k_1\times (n-k_2)}\\
O_{(n-k_2)\times k_2} & V_{11}^3
\end{pmatrix}, \\
Z_{21} &= \begin{pmatrix}
O_{1\times (n-k_1)} & Z_{21}^2
\end{pmatrix}, W_{11}^3 = \He(\Lambda_+ Z_{11}^3), \\
%\end{align*}
%Similarly, %if all eigenvalues of $\Omega_-$ are in $\mathbb{C}_-$, then 
%we have 
%\begin{align*}
V_{21} &= \begin{pmatrix}
O_{1\times (n-k_2)} & V_{21}^2
\end{pmatrix} \mbox { and } W_{22}^3 = \He(\Omega_+ V_{11}^3). 
\end{align*}
Furthermore, it follows from the structure of 
$W_{11}$ and $W_{22}$ that we have  
\[
W_{21} = \begin{pmatrix}
O_{(n-k_2)\times (n-k_1)} & O_{(n-k_2)\times k_1}\\
O_{k_2\times (n-k_1)} & W_{21}^3
\end{pmatrix}. 
\]
\end{lemma}
\begin{proof}
We prove only the structure of $Z_{11}$ and $W_{11}$. 
We focus on $W_{11}=\He(\Lambda Z_{11})$, $W_{11}\in\mathbb{S}^n_{+}$
and $Z_{11}\in\mathbb{S}^n_+$. Then $W_{11}=\He(\Lambda Z_{11})$ is
equivalently written as
\begin{align*}
\begin{pmatrix}
W_{11}^1 & (W_{11}^2)^T\\
W_{11}^2&W_{11}^3
\end{pmatrix} &= \He\left(
\begin{pmatrix}
\Lambda_- & \\
 & \Lambda_+
\end{pmatrix}
\begin{pmatrix}
Z_{11}^1 & *\\
Z_{11}^2 & Z_{11}^3
\end{pmatrix}\right)\\
&= \begin{pmatrix}
\He(\Lambda_-Z_{11}^1) & *\\%(Z_{11}^2)^T\Lambda_+^T + \Lambda_-(Z_{11}^2)^T \\
\Lambda_+Z_{11}^2 + Z_{11}^2 \Lambda_-^T&\He(\Lambda_+Z_{11}^3)
\end{pmatrix}. 
\end{align*}
We remark that $\Lambda_-$ is Hurwitz stable 
because all the eigenvalues of $\Lambda_-$ are negative. %in $\mathbb{C}_-$. 
Since $W_{11}^1 = \He(\Lambda_-Z_{11}^1)$ 
can be seen as the Lyapunov equation, 
and since $W_{11}^1\in\mathbb{S}^{n-k_1}_+$ and $\Lambda_-$ is Hurwitz stable, 
we have 
\[
Z_{11}^1 = -\int_0^{\infty} \exp(\Lambda_-t)W_{11}^1\exp(\Lambda_-^Tt)\, dt. 
\]
It follows from the positive semidefiniteness of 
$Z_{11}^1$ and $W_{11}^1$ that $Z_{11}^1$ must be 
the %$(n-k_1)\times (n-k_1)$
zero matrix. 
Substituting this into the Lyapunov equation, 
we obtain $W_{11}^1 = O$. %$W_{11}^1 = O_{(n-k_1)\times (n-k_1)}$. 
Consequently, $Z_{11}^2$ and $W_{11}^2$ are also the $k_1\times (n-k_1)$ 
zero matrix because $Z_{11}\in\mathbb{S}^n_+$ and 
$W_{11}^1\in\mathbb{S}^n_+$, respectively.
\end{proof}

Substituting the structure of dual solutions $(Z, V, W)$ 
to the first equality constraint 
%$f^TZ_{11}f + Z_{22} + g^TV_{11}g + V_{22} = 1$ 
in \eqref{dual2}, we obtain
\begin{align*}
&f^TZ_{11}f + Z_{22} + g^TV_{11}g + V_{22}\\
=& \begin{pmatrix}
f_-\\
f_+
\end{pmatrix}^T\begin{pmatrix}
O_{(n-k_1)\times (n-k_1)} & O_{(n-k_1)\times k_1}\\
O_{k_1\times(n-k_1)} & Z_{11}^3
\end{pmatrix}\begin{pmatrix}
f_-\\
f_+
\end{pmatrix} + Z_{22}\\
&\quad + \begin{pmatrix}
g_-\\
g_+
\end{pmatrix}^T\begin{pmatrix}
O_{(n-k_2)\times (n-k_2)} & O_{(n-k_2)\times k_2}\\
O_{k_2\times (n-k_2)} & V_{11}^3
\end{pmatrix}\begin{pmatrix}
g_-\\
g_+
\end{pmatrix} + V_{22}\\
=& f^T_+ Z_{11}^3 f_+ + Z_{22} + g_+^TV_{11}^3 g_+ + V_{22} = 1. 
\end{align*}
Moreover, we have 
\begin{align*}
J\bullet W_{21} &= \mbox{Trace}(T^TSW_{21}^T) =\mbox{Trace}\left(
T^TS
%\begin{pmatrix}
%T_- \\
%T_+
%\end{pmatrix}\begin{pmatrix}
%S^T_- & S^T_+
%\end{pmatrix}
\begin{pmatrix}
O_{(n-k_1)\times (n-k_2)} & O_{(n-k_1)\times k_2}\\
O_{k_1\times (n-k_2)} & (W_{21}^3)^T
\end{pmatrix}\right)\\
&= \mbox{Trace}(T_+S^T_+(W_{21}^3)^T)) = J_+\bullet W_{21}^3, \\
h_1^T\bullet Z_{21}&= \begin{pmatrix}
h_{1-}^T & h_{1+}^T
\end{pmatrix}\bullet \begin{pmatrix}
O_{1\times (n-k_1)} & Z_{21}^2
\end{pmatrix} = h_{1+}^T\bullet Z_{21}^2, \mbox{ and }\\
h_2^T\bullet V_{21}&= \begin{pmatrix}
h_{2-}^T & h_{2+}^T
\end{pmatrix}\bullet \begin{pmatrix}
O_{1\times (n-k_2)} & V_{21}^2
\end{pmatrix} = h_{2+}^T\bullet V_{21}^2. 
\end{align*}
Therefore \eqref{dual2} is equivalent to the following optimization problem: 
\begin{align}\label{dual3}
&\left\{
\begin{array}{cl}
\displaystyle\sup_{\hat{W}_{ij}, \hat{Z}_{ij}, \hat{V}_{ij}} & 2(h_{1+}^T\bullet \hat{Z}_{21} + h_{2+}^T\bullet \hat{V}_{21} + J_+\bullet \hat{W}_{21})\\
\mbox{subject to} & f_+^T\hat{Z}_{11}f_+ + Z_{22} + g_+^T\hat{V}_{11}g_+ + V_{22} = 1, \\
& \hat{W}_{11} = \He(\Lambda_+ \hat{Z}_{11}), \hat{W}_{22} = \He(\Omega_+ \hat{V}_{11}), \begin{pmatrix}
\hat{W}_{11} & \hat{W}_{21}^T\\
\hat{W}_{21} & \hat{W}_{22}
\end{pmatrix}\in\mathbb{S}^{k_1+k_2}_+, \\
& \begin{pmatrix}
\hat{Z}_{11} & \hat{Z}_{21}^T\\
\hat{Z}_{21} & \hat{Z}_{22}
\end{pmatrix}\in\mathbb{S}^{k_1+1}_+, \begin{pmatrix}
\hat{V}_{11} & \hat{V}_{21}^T\\
\hat{V}_{21} & \hat{V}_{22}
\end{pmatrix}\in\mathbb{S}^{k_2+1}_+. 
\end{array}
\right.
\end{align}

\begin{remark}\label{facialred}
We have successfully reduced the dual \eqref{dual2} when at least either $G_{zw}$ or $G_{yw}$ has stable invariant zeros. This reduction corresponds to {\itshape facial reduction} in the literature of the optimization theory, which was proposed by Borwein and Wolkowicz in \cite{Borwein81} for general convex cone programming problems. Thereafter, \cite{Ramana97} and \cite{Ramana97SIOPT} proposed facial reduction for SDP problems. 

Facial reduction for SDP problems is a finitely iterative algorithm. It works for non-strictly feasible SDP problems like dual \eqref{dual2}, and generates a strictly feasible SDP problem whose optimal value is equal to the original. 

The number of minimal iterations of facial reduction is an important concept in convex analysis and is called {\itshape the degree of singularity of the SDP problem}. The degree is used for the error bound analysis of SDP feasibility problems in \cite{Sturm00b} and perturbation analysis of SDP in \cite{Cheung14}. In the proof of Lemma \ref{lem:dstructure}, we can see that the facial reduction spends only one iteration in Case \ref{Case2}. It is proved in \cite{Waki18} that the same fact holds for $H_\infty$ output feedback control problem for MIMO dynamical system. %On the other hand, as we will see in Section \ref{sec:case4}, i.e., in the case where $d_{12}=0$ or $d_{21}=0$, the facial reduction spends more than one iteration.  
\end{remark}
The next lemma provides the dual of \eqref{dual3}. 
\begin{lemma}\label{lem:dualdual2}
The dual of \eqref{dual3} can be reformulated as follows: 
\begin{align}\label{LMI4case2}
&\left\{
\begin{array}{cl}
\displaystyle\inf_{\gamma, \hat{X}, \hat{Y}} & \gamma\\
\mbox{subject to} & - \begin{pmatrix}
\He(\Lambda_+^T \hat{X}) -\gamma f_+f_+^T & h_{1+}\\
h_{1+}^T & -\gamma
\end{pmatrix} \in\mathbb{S}_+^{k_1+1}, \\
& - \begin{pmatrix}
\He(\Omega_+^T \hat{Y}) -\gamma g_+g_+^T & h_{2+}\\
h_{2+}^T & -\gamma
\end{pmatrix} \in\mathbb{S}_+^{k_2+1}, \begin{pmatrix}
\hat{X} & -J_+^T\\
-J_+& \hat{Y}
\end{pmatrix}\in\mathbb{S}^{k_1+k_2}_+, 
\end{array}
\right.
\end{align}
Moreover the duality gap between \eqref{LMI4case2} and 
\eqref{dual3} is zero, 
and both \eqref{LMI4case2} and \eqref{dual3} have optimal solutions.  
\end{lemma}
\begin{proof}
We can prove by a similar manner in Lemma \ref{lem:dual} 
the fact that the dual of \eqref{LMI4case2} is \eqref{dual3}. 
In fact, it is well-known that the dual of the dual problem is exactly the primal problem. The proof on the zero duality gap 
between \eqref{LMI4case2} and \eqref{dual3} is provided in \ref{subapp:proof3}. 
In addition, we can prove the existence of 
optimal solutions of \eqref{LMI4case2} and \eqref{dual3} 
by similar manners to the proofs in \ref{subapp:proof2} and \ref{subapp:proof3}.  
\end{proof}

Since all eigenvalues of both $\Lambda_+$ and $\Omega_+$ are positive, 
we obtain the same result as Theorem \ref{thm:case1} 
by applying the discussion in Section \ref{sec:case1}. 
We summarize the result in this subsection as follows: 
\begin{theorem}\label{thm:case2}
Let us consider Case \ref{Case2} stated at the final part of Section \ref{sec:Hinf}. 
Then the optimal value $\gamma^*$ of LMI problem \eqref{LMI2} is equal 
to the maximum eigenvalue of the symmetric matrix $E$ defined by  
\begin{align}\label{gammacase2}
E&:= \begin{pmatrix}
O& F_+^{-1/2}J_+^TG_+^{1/2} & F_+^{-1/2}H^{1/2}_{1+} & O \\
G_+^{-1/2}J_+F_+^{-1/2} & O& O &G_+^{-1/2}H^{1/2}_{2+}\\
H^{1/2}_{1+}F^{-1/2}_+ & O & O &O \\
O & H^{1/2}_{2+}G_+^{-1/2} & O & O
\end{pmatrix}. 
\end{align}
Here $F_+, H_{1+}\in\mathbb{S}^{k_1}_+$ and $G_+, H_{2+}\in\mathbb{S}^{k_2}$ are defined by
\begin{align*}
&\left\{
\begin{array}{ccl}
F_+ &=& \displaystyle\int_{0}^{\infty}\exp(-\Lambda_+^T t)
f_+f_+^T\exp(-\Lambda_+ t)\, dt,\\ 
G_+ &=& \displaystyle\int_{0}^{\infty}\exp(-\Omega_+^T t)
g_+g_+^T\exp(-\Omega_+ t)\, dt, \\
H_{1+} &=& \displaystyle\int_{0}^{\infty}\exp(-\Lambda_+^T t)
h_{1+}h_{1+}^T\exp(-\Lambda_+ t)\, dt,\\ 
H_{2+} &=& \displaystyle\int_{0}^{\infty}\exp(-\Omega^T_+ t)
h_{2+}h_{2+}^T\exp(-\Omega_+ t)\, dt. 
\end{array}
\right.
\end{align*}
In particular, $F_+$ and $G_+$ are positive definite because of Assumption \ref{A1}. 
\end{theorem}
\begin{proof}
All the optimal values of \eqref{LMI2}, \eqref{dual2}, \eqref{dual3} 
and \eqref{LMI4case2} are equivalent. 
In fact, the equivalence between \eqref{LMI2} and  
\eqref{dual2} follows from Theorem \ref{thm:duality}. 
The optimal values of \eqref{dual2} is equal 
to the optimal value of \eqref{dual3} 
because we obtain \eqref{dual3} from \eqref{dual2} 
by investigating the structure of solutions of \eqref{dual2}. 
The equivalence between the optimal values of \eqref{dual3} and \eqref{LMI4case2} 
follows from Lemma \ref{lem:dualdual2}. 
Finally, we can prove that the optimal value of \eqref{LMI4case2} is given by $\lambda_{\max}(E)$. In fact, we can prove that if $\gamma^*=0$, then $h_{1+}=0$, $h_{2+}=0$ and $J_+=O$, and thus $E$ is the zero matrix. Clearly, $\gamma^*=\lambda_{\max}(E)$. Otherwise, since $\gamma^*>0$, we can prove $\gamma^*=\lambda_{\max}(E)$ in a similar manner to the proof of Theorem \ref{thm:case1}. 
\end{proof}

We remark that the size of the matrix $E$ in \eqref{gammacase1} is $4n$, 
while in \eqref{gammacase2} the size is $2(k_1+k_2)$.  
When both of $G_{zu}$ and $G_{yw}$ have no stable invariant zeros, 
then \eqref{gammacase2} is equal to \eqref{gammacase1}. 

We obtain the following corollary from Theorem \ref{thm:case2}. 
\begin{corollary}\label{coro:closedform}
If all the invariant zeros of $G_{zu}$ are stable, 
then the optimal value $\gamma^*$ of \eqref{LMI2} is equal to
\begin{align}\label{Gzustable}
\lambda_{\max}\begin{pmatrix}
 O & G_+^{-1/2}H_{2+}^{1/2}\\
H_{2+}^{1/2}G_+^{-1/2} & O
\end{pmatrix}. 
\end{align}
Similarly, if all the invariant zeros in $G_{yw}$ are stable, 
then the optimal value $\gamma^*$ of \eqref{LMI2} is equal to
\begin{align}
\lambda_{\max}\begin{pmatrix}
 O & F_+^{-1/2}H_{1+}^{1/2}\\
H_{1+}^{1/2}F_+^{-1/2} & O
\end{pmatrix}. 
\end{align}
Finally, if both of $G_{zu}$ and $G_{yw}$ have no unstable zeros, 
then $\gamma^*$ is equal to zero. 
\end{corollary}
\begin{proof}
We prove \eqref{Gzustable} only.  
Intuitively, \eqref{Gzustable} directly follows from 
Theorem \ref{thm:case2} because $F_+$ and $H_{1+}$ both vanish in the
present case.  
The proof can be made more rigorous as follows.  
Since all the invariant zeros in $G_{zu}$ are stable, 
we have $\Lambda = \Lambda_-$, $S=S_-$ and $f=f_-$. 
It follows from Lemma \ref{lem:dstructure} that 
any feasible solution $(Z, V, W)$ of \eqref{dual2} has the form of 
\begin{align*}
Z &= \begin{pmatrix}
O_n & 0\\
0^T &Z_{22}
\end{pmatrix}, V = \begin{pmatrix}
O_{n-k_2} & O_{(n-k_2)\times k_2} & 0\\
O_{k_2\times (n-k_2)} & \hat{V}_{11}& \hat{V}_{21}^T\\
0^T &\hat{V}_{21}&V_{22}
\end{pmatrix} \mbox{ and } \\
 W &= \begin{pmatrix}
O_{n} & O_{n\times (n-k_2)} & O_{n\times k_2}\\
O_{(n-k_2)\times n} & O_{n-k_2} & O_{(n-k_2)\times k_2}\\
O_{k_2\times n} & O_{k_2\times (n-k_2)} & \hat{W}_{22}
\end{pmatrix}. 
\end{align*}
Substituting them into \eqref{dual2}, it can be reformulated as
\begin{align}\label{dual3_tmp}
&\left\{
\begin{array}{cl}
\displaystyle\sup_{Z_{22}, \hat{V}_{ij}, V_{22}, \hat{W}_{22}} & 2h_{2+}^T\bullet \hat{V}_{21}\\
\mbox{subject to}& Z_{22} + g_+^T\hat{V}_{11}g_+ + V_{22} = 1, \hat{W}_{22} = \He(\Omega_+ \hat{V}_{11}), \\
&\begin{pmatrix}
\hat{V}_{11} & \hat{V}_{21}^T\\
\hat{V}_{21} & \hat{V}_{22}
\end{pmatrix}\in\mathbb{S}^{k_2+1}_+, Z_{22}\ge0, \hat{W}_{22} \in\mathbb{S}^{k_2}_+. 
\end{array}
\right. 
\end{align}
By following a similar line to Lemma \ref{lem:dualdual2}, 
we obtain the following dual problem:
\begin{align}\label{LMI4case2_tmp}
&\displaystyle\inf%_{\gamma, \hat{Y}}
\left\{\gamma : \gamma \ge 0, - \begin{pmatrix}
\He(\Omega_+^T \hat{Y}) -\gamma g_+g_+^T & h_{2+}\\
h_{2+}^T & -\gamma
\end{pmatrix} \in\mathbb{S}_+^{k_2+1}, \hat{Y}\in\mathbb{S}^{k_2}_+
\right\}. 
\end{align}
We can prove the duality gap between \eqref{LMI4case2_tmp} and
 \eqref{dual3_tmp} is zero.  
For \eqref{LMI4case2_tmp}, it follows from a similar manner 
in the proof of Theorem \ref{thm:case2} that 
\[
\gamma^*=\hat{\gamma} = \lambda_{\max}\begin{pmatrix}
 O & G_+^{-1/2}H_{2+}^{1/2}\\
H_{2+}^{1/2}G_+^{-1/2} & O
\end{pmatrix}. 
\]
%
%
%When no unstable invariant zeros in both $G_{zu}(s)$ and $G_{yw}(s)$ exist, we have $\Lambda = \Lambda_-$, $S=S_-$, $f=f_-$, $\Omega=\Omega_-$, $T=T_-$ and $g=g_-$. It follows from Lemma \ref{lem:dstructure} any feasible solution $(Z, V, W)$ of \eqref{dual2} has the form of 
%\[
%Z = \begin{pmatrix}
%O_n & 0\\
%0^T &Z_{22}
%\end{pmatrix}, V = \begin{pmatrix}
%O_n & 0\\
%0^T &V_{22}
%\end{pmatrix} \mbox{ and } W = O_{2n}. 
%\]
%Substituting them into \eqref{dual2}, it can be reformulated as 
%\begin{align*}%\label{dual3_allstable}
%&\displaystyle\sup_{Z_{22}, V_{22}}\left\{0 : Z_{22}+ V_{22} = 1, Z_{22}\ge 0, V_{22}\ge 0\right\}. 
%\end{align*}
%The optimal value of this problem is zero and it follows from the proof of Theorem \ref{thm:case2} that it is equal to the optimal value of \eqref{LMI2}. 
%
\end{proof}

\section{Analysis of Case 3}\label{sec:case3}

In this section, we deal with Case 3, stated in the final part of Section
\ref{sec:Hinf}. For simplicity, we assume the following.
\begin{itemize}
\item We allow both $G_{zu}$ and $G_{yw}$ to have complex invariant zeros on the imaginary axis. %, that is, we do not require (b) of Assumption \ref{A}.
\item All the invariant zeros on the imaginary axis are distinct from each other.
\item Both $G_{zu}$ or $G_{yw}$ do not have $0$ as invariant zero. 
\end{itemize}
In particular, the first assumption corresponds to remove (b) of Assumption \ref{A}. Otherwise, we deal with only $0$ as invariant zeros on the imaginary axis. Thus this assumption makes the discussions in Case 3 more general. Other assumptions are imposed to improve the readability. However, we emphasize that the result in Theorem \ref{thm:case3} is still valid without assuming them. 

Under (a) of Assumption \ref{A} and these additional assumptions, we consider the case where $G_{zu}$ (resp. $G_{yw}$) has $2m_1$ (resp. $2m_2$) invariant zeros $\lambda_1, \ldots, \lambda_{m_1}$ (resp. $\omega_1, \ldots, \omega_{m_2}$) and their complex conjugates $\bar{\lambda}_1$, $\ldots$, $\bar{\lambda}_{m_1}$ (resp. $\bar{\omega}_1$, $\ldots$, $\bar{\omega}_{m_2}$) on the imaginary axis. Moreover, these invariant zeros are distinct from each other. 
The null vectors $(s_j^T, f_j)$ \ $(j=1, \ldots, m_1)$ and 
$(t_j^T, g_j)$ \ $(j=1, \ldots, m_2)$ associated with 
the invariant zeros 
$\lambda_j\ (j=1,\cdots,m_1)$ and 
$\omega_j\ (j=1,\cdots,m_2)$
can be written, respectively, by 
\begin{align*}
\begin{pmatrix}
s_j\\
f_j
\end{pmatrix} &= \begin{pmatrix}
s_j^r\\
f_j^r
\end{pmatrix} + \sqrt{-1}\begin{pmatrix}
s_j^i\\
f_j^i
\end{pmatrix} \ (j=1, \ldots, m_1), \\
\begin{pmatrix}
t_j\\
g_j
\end{pmatrix} &= \begin{pmatrix}
t_j^r\\
g_j^r
\end{pmatrix} + \sqrt{-1}\begin{pmatrix}
t_j^i\\
g_j^i
\end{pmatrix} \ (j=1, \ldots, m_2). 
\end{align*}
Here $s_j^r$, $s_j^i$, $t_j^r$, and $t_j^i$ are in $\mathbb{R}^n$ and
$f_j^r$, $f_j^i$, $g_j^r$, and $g_j^i$ are in $\mathbb{R}$. 
Note that $|f_j|^2 = (f_j^r)^2+(f_j^i)^2$ and $|g_j|^2= (g_j^r)^2+(g_j^i)^2$ are nonzero due to Assumption \ref{A1} and \ref{F2} of Lemma \ref{lem:zero}.  
Then we have 
\begin{align}
\label{sf0}
\begin{pmatrix}
(s_1^r)^T & f_1^r\\
(s_1^i)^T & f_1^i\\
\vdots & \vdots\\
(s_{m_1}^r)^T & f_{m_1}^r\\
(s_{m_1}^i)^T & f_{m_1}^i
\end{pmatrix}\begin{pmatrix}
A & b_2\\
c_1^T & d_{12}
\end{pmatrix} &= \begin{pmatrix}
F(\lambda_1) & & \\
& \ddots & \\
& & F(\lambda_{m_1})
\end{pmatrix}\begin{pmatrix}
(s_1^r)^T & 0\\
(s_1^i)^T & 0\\
\vdots & \vdots\\
(s_{m_1}^r)^T & 0\\
(s_{m_1}^i)^T & 0
\end{pmatrix}, \\
\label{tg0}
\begin{pmatrix}
(t_1^r)^T & g_1^r\\
(t_1^i)^T & g_1^i\\
\vdots & \vdots\\
(t_{m_2}^r)^T & g_{m_2}^r\\
(t_{m_2}^i)^T & g_{m_2}^i
\end{pmatrix}\begin{pmatrix}
A^T & c_2\\
b_1^T & d_{21}
\end{pmatrix} &= \begin{pmatrix}
F(\omega_1) & &\\
& \ddots &\\
& & F(\omega_{m_2})
\end{pmatrix}\begin{pmatrix}
(t_1^r)^T & 0\\
(t_1^i)^T & 0\\
\vdots & \vdots\\
(t_{m_2}^r)^T & 0\\
(t_{m_2}^i)^T & 0%\\
%(t_{m_2+1})^T & 0\\
\end{pmatrix}, 
\end{align}
where $F(\lambda)$ is defined by 
\[
F(\lambda) = \begin{pmatrix}
0 & \Im(\lambda)\\
-\Im(\lambda) & 0
\end{pmatrix}. 
\]
Here $\Im(\lambda)$ denotes the imaginary part of $\lambda\in\mathbb{C}$. 
Note that $F(\lambda)$ is non-singular when $\lambda \neq0$. For simplicity, we denote \eqref{sf0} and \eqref{tg0} by 
\begin{align}\label{zeroOnimg}
\begin{pmatrix}
S^T_0 & f_0
\end{pmatrix}\begin{pmatrix}
A & b_2\\
c_1^T & d_{12}
\end{pmatrix} = \Lambda_0^T
\begin{pmatrix}
S_0^T & 0
\end{pmatrix} \mbox{ and }
\begin{pmatrix}
T^T_0 & g_0
\end{pmatrix}\begin{pmatrix}
A^T & c_2\\
b_1^T & d_{21}
\end{pmatrix} = \Omega_0^T\begin{pmatrix}
T_0^T & 0
\end{pmatrix}. 
\end{align}
We remark that the sizes of $\Lambda_0$ and $\Omega_0$ are $2m_1$ and $2m_2$, respectively. %, due to Assumption \ref{A3}. 
In addition, we can reformulate \eqref{zeroGzu} and \eqref{tgdef} as 
\begin{align*}
\begin{pmatrix}
S^T_0 & f_0\\
S^T & f
\end{pmatrix}\begin{pmatrix}
A & b_2\\
c_1^T & d_{12}
\end{pmatrix} &= \begin{pmatrix}
\Lambda_0^T & \\
& \Lambda^T
\end{pmatrix}
\begin{pmatrix}
S_0^T & 0\\
S^T & 0
\end{pmatrix}, \\
\begin{pmatrix}
T^T_0 & g_0\\
T^T & g
\end{pmatrix}\begin{pmatrix}
A^T & c_1\\
b_2^T & d_{21}
\end{pmatrix} &= \begin{pmatrix}
\Omega_0^T & \\
& \Omega^T
\end{pmatrix}
\begin{pmatrix}
T_0^T & 0\\
T^T & 0
\end{pmatrix}. 
\end{align*}
As we have assumed that $d_{12}\neq0$ and $d_{21}\neq0$, it follows from Lemma \ref{lem:gplant} that $(S_0, S)$, $(T_0, T)\in\mathbb{R}^{n\times n}$ are non-singular, and we have 
\[
\begin{pmatrix}
b_2\\
d_{12}\\
0
\end{pmatrix}^{\perp} = \begin{pmatrix}
S_0 &S & 0\\
f_0^T &f^T & 0\\
0^T &0^T & 1
\end{pmatrix} \mbox{ and }\begin{pmatrix}
c_2\\
d_{21}\\
0
\end{pmatrix}^{\perp} = \begin{pmatrix}
T_0&T & 0\\
g_0^T&g^T & 0\\
0^T &0^T & 1
\end{pmatrix}. 
\]
%\eqref{LMI} can be formulated as 
By following similar lines leading to \eqref{LMI2},
we can reformulate \eqref{LMI} as in
\begin{align}\label{LMI_img2}
&\left\{
\begin{array}{cl}
\displaystyle\inf_{\gamma, \hat{X}_{ij}, \hat{Y}_{ij}} & \gamma\\
\mbox{subject to} & - \begin{pmatrix}
\He\left(\begin{pmatrix}
\Lambda^T_0 & \\
& \Lambda^T
\end{pmatrix}
\hat{X}\right) -\gamma \begin{pmatrix}
f_0\\
f
\end{pmatrix} \begin{pmatrix}
f_0\\
f
\end{pmatrix}^T &\hat{h}_1 \\
\hat{h}_1^T&-\gamma
\end{pmatrix} \in\mathbb{S}_+^{n+1}, \\
& - \begin{pmatrix}
\He\left(\begin{pmatrix}
\Omega^T_0 & \\
& \Omega^T
\end{pmatrix}
\hat{Y}\right) -\gamma \begin{pmatrix}
g_0\\
g
\end{pmatrix} \begin{pmatrix}
g_0\\
g
\end{pmatrix}^T &\hat{h}_2 \\
\hat{h}_2^T&-\gamma
\end{pmatrix} \in\mathbb{S}_+^{n+1}, \\
& \begin{pmatrix}
\hat{X} & -\hat{J}^T\\
-\hat{J} & \hat{Y}
\end{pmatrix}\in\mathbb{S}^{2n}_+, 
\end{array}
\right.
\end{align}
where we define $\hat{X}$, $\hat{Y}$, $\hat{h}_1$ and $\hat{h}_2$ as follows: 
\begin{align*}
\hat{X} &= \begin{pmatrix}
S_0^T\\
S^T
\end{pmatrix}X\begin{pmatrix}
S_0 & S
\end{pmatrix}, \hat{Y} = \begin{pmatrix}
T_0^T\\
T^T
\end{pmatrix}Y\begin{pmatrix}
T_0 & T
\end{pmatrix}, \hat{J} = \begin{pmatrix}
T^T_0\\
T^T
\end{pmatrix}\begin{pmatrix}
S_0 & S
\end{pmatrix}, J = T^TS, \\
%\begin{pmatrix}
%\hat{X} & -\hat{D}^T\\
%-\hat{D} & \hat{Y}
%\end{pmatrix} &= \begin{pmatrix}
%S^T_0 & \\
%S^T& \\
%& T^T_0\\
%& T^T
%\end{pmatrix}\begin{pmatrix}
%X & -I_n\\
%-I_n & Y
%\end{pmatrix}\begin{pmatrix}
%S_0 & S & & \\
%& &T_0 & T
%\end{pmatrix}, \\
\hat{h}_1^T &= \begin{pmatrix}
h_{10}^T & h_1^T
\end{pmatrix}, h_1^T = b_1^TS + d_{11}f^T, 
\hat{h}_2^T = \begin{pmatrix}
h_{20}^T & h_2^T
\end{pmatrix}, h_2^T = c_1^TT + d_{11}g^T, \\
h_{10}^T& = \begin{pmatrix}
(h_{10}^r)_1 & (h_{10}^i)_1 & \cdots & (h_{10}^r)_{m_1} & (h_{10}^i)_{m_1}
\end{pmatrix} = b_1^TS_0 + d_{11}f_0^T, \\
h_{20}^T& = \begin{pmatrix}
(h_{20}^r)_1 & (h_{20}^i)_1 & \cdots & (h_{20}^r)_{m_2} & (h_{20}^i)_{m_2}
\end{pmatrix} = c_1^TT_0 + d_{11}g_0^T. 
\end{align*}
Applying Lemma \ref{lem:dual}, its dual is formulated as follows: 
\begin{align}\label{dual_img2}
&\left\{
\begin{array}{cl}
\displaystyle\sup & \begin{pmatrix}
& \hat{h}_1\\
\hat{h}_1^T &  
\end{pmatrix}\bullet Z+ \begin{pmatrix}
& \hat{h}_2\\
\hat{h}_2^T & 
\end{pmatrix}\bullet V + \begin{pmatrix}
& \hat{J}^T\\
\hat{J} & 
\end{pmatrix}\bullet W\\
\mbox{subject to} & \begin{pmatrix}
f_0f_0^T & * & \\
f f_0^T & ff^T& \\
& & 1
\end{pmatrix}\bullet Z + \begin{pmatrix}
g_0g_0^T & * & \\
gg_0^T & gg^T& \\
& & 1
\end{pmatrix}\bullet V = 1, \\ 
& W_{11} = \He\left(\begin{pmatrix}
\Lambda_0 & \\
 & \Lambda
\end{pmatrix}
\begin{pmatrix}
Z_{11} & *\\
Z_{21} &Z_{22}
\end{pmatrix}
\right), W=\begin{pmatrix}
W_{11} & *\\
W_{21} & W_{22}
\end{pmatrix}\in\mathbb{S}^{2n}_+, \\
&W_{22} = \He\left(
\begin{pmatrix}
\Omega_0 & \\
 & \Omega
\end{pmatrix}
\begin{pmatrix}
V_{11} & *\\
V_{21} &V_{22}
\end{pmatrix}
\right), \\
& Z = \begin{pmatrix}
Z_{11} & *&*\\
Z_{21} & Z_{22} & *\\
Z_{31} & Z_{32} & Z_{33}
\end{pmatrix}\in\mathbb{S}^{n+1}_+, V = \begin{pmatrix}
V_{11} & *&*\\
V_{21} & V_{22} & *\\
V_{31} & V_{32} & V_{33}
\end{pmatrix}\in\mathbb{S}^{n+1}_+. %, Z_{11}\in\mathbb{S}^{2m_1}, V_{11}\in\mathbb{S}^{2m_2}. 
\end{array}
\right.
\end{align}
We remark that the duality gap between \eqref{LMI_img2} and \eqref{dual_img2} is zero, and \eqref{dual_img2} has an optimal solution. In fact, these facts follows from Theorem \ref{thm:duality} because the proof on the zero duality gap between \eqref{LMI} and its dual in Theorem \ref{thm:duality} is independent in the computation of the perpendicular matrices in \eqref{LMI}. 

The following lemma is useful to reduce \eqref{dual_img2}. We give a proof in \ref{subapp:lemma9}. 
\begin{lemma}\label{reduction}
Any feasible solution $(W_{ij}, Z_{ij}, V_{ij})$ has the form of  
\begin{align}
\label{Z11}
Z_{11} &= \diag(z_1, z_1, \ldots, z_{m_1}, z_{m_1}), %\\
%\label{Z21}
Z_{21} = O_{(n-2m_1)\times 2m_1}, \\
\label{V11}
V_{11} &= \diag(v_1, v_1, \ldots, v_{m_2}, v_{m_2}), %\\
%\label{Z21}
V_{21} = O_{(n-2m_2)\times 2m_1}, \\
\label{W11}
(W_{11})_{ij} &= 0 \ (i, j=1, \ldots, 2m_1), (W_{22})_{ij} = 0 \ (i, j=1, \ldots, 2m_2), \\
%\label{W21}
(W_{21})_{ij} &= 0 \ (i = 1, \ldots, n, j =1, \ldots, 2m_1 \mbox{ and } i=1, \ldots, 2m_2, j=2m_1+1, \ldots, n), \nonumber
\end{align}
where $\diag(a_1, \ldots, a_n)$ stands for the diagonal matrix with the diagonal elements $a_1, \ldots, a_n$. 
\end{lemma}

\begin{remark}\label{remark2}
We have reduced the size of the matrix variable $W$ in \eqref{dual_img2} in Lemma \ref{reduction}. This reduction also corresponds to the facial reduction for SDP as well as Lemma \ref{lem:dstructure}.  In addition, we can also apply Lemma \ref{lem:dstructure} when at least either $G_{zu}$ or $G_{yw}$ has stable invariant zeros. 
\end{remark}

By using \eqref{Z11}, \eqref{V11} and \eqref{W11}, 
we can reformulate \eqref{dual_img2} as follows:  
\begin{align}\label{dual_img3}
&\left\{
\begin{array}{cl}
\displaystyle\sup & \begin{pmatrix}
& &h_{10}\\
& & h_1\\
h_{10}^T & h_1^T &  
\end{pmatrix}\bullet Z+ \begin{pmatrix}
& &h_{20}\\
& & h_2\\
h_{20}^T & h_2^T &  
\end{pmatrix}\bullet V + \begin{pmatrix}
& J^T\\
J & 
\end{pmatrix}\bullet \hat{W}\\
\mbox{subject to} &  
\begin{pmatrix}
f_0f_0^T & f_0f^T & \\
f f_0^T & ff^T& \\
& & 1
\end{pmatrix}\bullet Z + \begin{pmatrix}
g_0g_0^T & g_0g^T & \\
gg_0^T & gg^T& \\
& & 1
\end{pmatrix}\bullet V = 1, \\ 
%\displaystyle\sum_{j=1}^{m_1}(f_{0 (2j-1)}^2+f_{0 (2j)}^2)z_j + f^TZ_{22}f + Z_{33} \\
%&{} + \displaystyle\sum_{j=1}^{m_2}(g_{0 (2j-1)}^2+g_{0 (2j)}^2)v_j + g^TV_{22}g + V_{33} = 1, \\
& \hat{W}_{11} = \He\left( \Lambda Z_{22}\right), \hat{W}_{22} = \He\left(\Omega V_{22}\right),  \hat{W}=\begin{pmatrix}
\hat{W}_{11} & \hat{W}_{21}^T\\
\hat{W}_{21} & \hat{W}_{22}
\end{pmatrix}\in\mathbb{S}^{n_0}_+, \\
& Z = \begin{pmatrix}
Z_{11} & O & Z_{31}^T\\
O & Z_{22} & Z_{32}^T\\
Z_{31} & Z_{32} & Z_{33}
\end{pmatrix}\in\mathbb{S}^{n+1}_+, V = \begin{pmatrix}
V_{11} & O & V_{31}^T\\
O & V_{22} & V_{32}^T\\
V_{31} & V_{32} & V_{33}
\end{pmatrix}\in\mathbb{S}^{n+1}_+, \\
& Z_{11}= \diag(z_1, z_1, \ldots, z_{m_1}, z_{m_1}), z_j\in\mathbb{R} \ (j=1, \ldots, m_1),\\
&V_{11}=\diag(v_1, v_1, \ldots, v_{m_2}, v_{m_2}), v_j\in\mathbb{R} \ (j=1, \ldots, m_2),
\end{array}
\right.
\end{align}
where $n_0 = 2n - 2(m_1 +m_2)$. 
%Applying Lemma \ref{lem:dstructure} to \eqref{dual_img3}, we obtain 
%\begin{align}\label{dual_img4}
%&\left\{
%\begin{array}{cl}
%\displaystyle\sup & \begin{pmatrix}
%& &*\\
%& &*\\
%h_{10}^T & h_{1+}^T &  
%\end{pmatrix}\bullet \hat{Z}+ \begin{pmatrix}
%& &*\\
%& &*\\
%h_{20}^T & h_{2+}^T &  
%\end{pmatrix}\bullet \hat{V} + \begin{pmatrix}
%&*\\
%J_+ & 
%\end{pmatrix}\bullet \hat{W}\\
%\mbox{subject to} &  
%\begin{pmatrix}
%f_0f_0^T &* & \\
%f_+ f_0^T & f_+f_+^T& \\
%& & 1
%\end{pmatrix}\bullet \hat{Z} + \begin{pmatrix}
%g_0g_0^T & * & \\
%g_+g_0^T & g_+g_+^T& \\
%& & 1
%\end{pmatrix}\bullet \hat{V} = 1, \\ 
%%\displaystyle\sum_{j=1}^{m_1}(f_{0 (2j-1)}^2+f_{0 (2j)}^2)z_j + f_+^T\hat{Z}_{22}f_+ + Z_{33} \\
%%&{} + \displaystyle\sum_{j=1}^{m_2}(g_{0 (2j-1)}^2+g_{0 (2j)}^2)v_j + g_+^T\hat{V}_{22}g_+ + V_{33} = 1, \\
%& \hat{W}=\begin{pmatrix}
%\hat{W}_{11} & *\\
%\hat{W}_{21} & \hat{W}_{22}
%\end{pmatrix}\in\mathbb{S}^{k_1+k_2}_+, \\
%& \hat{Z} = \begin{pmatrix}
%Z_{11} & * &*\\
%O & \hat{Z}_{22} & *\\
%Z_{31} & \hat{Z}_{32} & Z_{33}
%\end{pmatrix}\in\mathbb{S}^{n_1}_+,  \hat{V} = \begin{pmatrix}
%V_{11} & * & *\\
%O & \hat{V}_{22} & *\\
%V_{31} & \hat{V}_{32} & V_{33}
%\end{pmatrix}\in\mathbb{S}^{n_2}_+, \\
%& Z_{11}= \diag(z_1, z_1, \ldots, z_{m_1}, z_{m_1}), \hat{W}_{11} = \He\left( \Lambda_+ \hat{Z}_{22}\right), \\
%& V_{11}=\diag(v_1, v_1, \ldots, v_{m_2}, v_{m_2}), \hat{W}_{22} = \He\left(\Omega_+\hat{V}_{22}\right), \\
%&z_j\in\mathbb{R} \ (j=1, \ldots, m_1), v_j\in\mathbb{R} \ (j=1, \ldots, m_2), 
%\end{array}
%\right.
%\end{align}
%where $n_1 = 2m_1+k_1+1$ and $n_2=2m_2+k_2+1$. 
Let $\mathcal{F}$ be the feasible region of \eqref{dual_img3}. 
%We introduce the notation $\mathcal{F}$ of the feasible region of \eqref{dual_img4} for the next lemma. 
%\begin{align*}
%\mathcal{F} &:= \left\{
%(\hat{Z}, \hat{V}, \hat{W}, z, v): %\in\mathbb{S}^{n_1}_+\times\mathbb{S}^{n_2}_+\times \mathbb{S}^{k_1+k_2}_+\times\mathbb{R}^{m_1+1}\times\mathbb{R}^{m_2+1} :
%\begin{array}{l}
%\begin{pmatrix}
%f_0f_0^T & * & \\
%f_+ f_0^T & f_+f_+^T& \\
%& & 1
%\end{pmatrix}\bullet \hat{Z} + \begin{pmatrix}
%g_0g_0^T & * & \\
%g_+g_0^T & g_+g_+^T& \\
%& & 1
%\end{pmatrix}\bullet \hat{V} = 1, \\ 
% \hat{W}_{11} = \He\left( \Lambda_+ \hat{Z}_{22}\right), \hat{W}_{22} = \He\left(\Omega_+\hat{V}_{22}\right), \\
% \hat{W}=\begin{pmatrix}
%\hat{W}_{11} & *\\
%\hat{W}_{21} & \hat{W}_{22}
%\end{pmatrix}\in\mathbb{S}^{k_1+k_2}_+, z\in\mathbb{R}^{m_1}, v\in\mathbb{R}^{m_2}, \\
% \hat{Z} = \begin{pmatrix}
%Z_{11} & * & *\\
%O & \hat{Z}_{22} & *\\
%Z_{31} & \hat{Z}_{32} & Z_{33}
%\end{pmatrix}\in\mathbb{S}^{n_1}_+,  \hat{V} = \begin{pmatrix}
%V_{11} & * &*\\
%O & \hat{V}_{22} & *\\
%V_{31} & \hat{V}_{32} & V_{33}
%\end{pmatrix}\in\mathbb{S}^{n_2}_+, \\
%(Z_{11})_{k, \ell} = 0, (V_{11})_{k, \ell} = 0 \ (k\neq \ell), \\
%(Z_{11})_{2j-1, 2j-1} = (Z_{11})_{2j, 2j} = z_j \ (j=1, \ldots, m_1), \\
%(V_{11})_{2j-1, 2j-1} = (V_{11})_{2j, 2j} = v_j \ (j=1, \ldots, m_2)
%\end{array}
%\right\}, 
%\end{align*}
The next lemma shows the dual of \eqref{dual_img3}. We give a proof in \ref{subapp:lemma10}. 
\begin{lemma}\label{lem:dualdual}
The dual of \eqref{dual_img3} can be formulated as follows: 
\begin{align}\label{LMI_img3}
&\left\{
\begin{array}{cl}
\displaystyle\inf_{\gamma, \hat{X}, \hat{Y}, U_{ij}^X, U_{ij}^Y} & \gamma\\
\mbox{subject to} & - \begin{pmatrix}
U_{11}^X -\gamma f_0f_0^T &(U_{21}^X)^T-\gamma f_0f^T &h_{10} \\
U_{21}^X-\gamma ff^T_0&\He(\Lambda^T \hat{X})-\gamma ff^T &h_{1} \\
h_{10}^T & h_{1}^T& -\gamma
\end{pmatrix} \in\mathbb{S}_+^{n+1}, \\
& - \begin{pmatrix}
U_{11}^Y -\gamma g_0g_0^T &(U_{21}^Y)^T-\gamma g_0g^T &h_{20} \\
U_{21}^Y-\gamma gg^T_0&\He(\Omega^T \hat{Y})-\gamma gg^T &h_{2} \\
h_{20}^T & h_{2}^T& -\gamma
\end{pmatrix} \in\mathbb{S}_+^{n+1}, \\
&\begin{pmatrix}
\hat{X} & -J^T\\
-J & \hat{Y}
\end{pmatrix}\in\mathbb{S}^{n_0}_+, U_{11}^X\in\mathbb{S}^{2m_1}, U_{11}^Y\in\mathbb{S}^{2m_2}, \\
& U_{21}^X\in\mathbb{R}^{(n-2m_1)\times 2m_1}, U_{21}^Y\in\mathbb{R}^{(n-2m_2)\times 2m_2}, \\
& (U_{11}^X)_{2j-1, 2j-1} + (U_{11}^X)_{2j, 2j} = 0 \ (j=1, \ldots, m_1), \\
& (U_{11}^Y)_{2j-1, 2j-1} + (U_{11}^Y)_{2j, 2j} = 0 \ (j=1, \ldots, m_2).
\end{array}
\right.
\end{align}
Moreover, the duality gap between \eqref{LMI_img3} and \eqref{dual_img3} is zero, and \eqref{dual_img3} has an optimal solution. 
\end{lemma}

We focus on the first inequality constraint in \eqref{LMI_img3} and can
see that all the off-diagonal elements of $U^X_{11}$ and all the
elements of $U_{21}^X$ do not appear in the other constraints in
\eqref{LMI_img3}. Hence it is enough to compute them after finding $\gamma$, $\hat{X}$
and all the diagonal elements of $U_{11}^X$. Proposition \ref{reduction_img}
shown below gives a simplification of LMI problem \eqref{LMI_img3} based on this
idea. For this, we use the following lemma. This lemma plays an essential role in the proof of Proposition \ref{reduction_img} and can be directly proved in a similar manner to the proof in \cite[Appendix D]{Waki18b}%. We give a proof of this lemma in \ref{app:mcomp}. 
\begin{lemma}\label{mcomp}
Let $k, \ell$ be positive integers. We assume that the two matrices $\left(\begin{smallmatrix}
U_{11} & U_{31}^T\\
U_{31} & U_{33}
\end{smallmatrix}\right)\in\mathbb{S}^{k+1}$ and $\left(\begin{smallmatrix}
U_{22} & U_{32}^T\\
U_{32} & U_{33}
\end{smallmatrix}\right)\in\mathbb{S}^{\ell+1}$ are positive semidefinite. 
%the following two matrices are positive semidefinite: 
%\[
%\begin{pmatrix}
%U_{11} & U_{31}^T\\
%U_{31} & U_{33}
%\end{pmatrix}\in\mathbb{S}^{k+1}_+ \mbox{ and } \begin{pmatrix}
%U_{22} & U_{32}^T\\
%U_{32} & U_{33}
%\end{pmatrix}\in\mathbb{S}^2_+. 
%\]
Then there exists $U_{21}\in\mathbb{R}^{\ell\times k}$ such that the matrix $\left(\begin{smallmatrix}
U_{11} & U_{21}^T&U_{31}^T\\
U_{21} & U_{22} &U_{32}^T\\
U_{31}& U_{32} & U_{33}
\end{smallmatrix}\right)$ is positive semidefinite. %the following matrix is positive semidefinite: 
%\begin{align*}%\label{matrix}
%\begin{pmatrix}
%U_{11} & U_{21}^T & U_{31}^T\\
%U_{21} & U_{22} & U_{32}^T\\
%_{31} & U_{32} & U_{33}
%\end{pmatrix}. %, \mbox{ where } Z_{31} = \frac{Z_{32}}{Z_{22}} Z_{21}.  
%\end{align*}
%In addition, we have $Z_{11}\neq O_n$, 
%\[
%\rank\begin{pmatrix}
%Z_{11} & Z_{21}^T\\
%Z_{21} & Z_{22}
%\end{pmatrix} = \rank Z_{11} \mbox{ and } \rank\begin{pmatrix}
%Z_{22} & Z_{32}^T\\
%Z_{32} & Z_{33}
%\end{pmatrix} \le 1, 
%\]
%if and only if \eqref{matrix} satisfies the maximal rank property. 
\end{lemma}

\begin{proposition}\label{reduction_img}
The optimal value of \eqref{LMI_img3} is equivalent to the following LMI problem:
\begin{align}\label{LMI_img5}
&\left\{
\begin{array}{cl}
\displaystyle\inf_{\begin{subarray}{c}\gamma, \hat{X}, \hat{Y}, \\ u_j^X, u_j^Y\end{subarray}} & \gamma\\
\mbox{subject to} & - \begin{pmatrix}
\He(\Lambda^T \hat{X})-\gamma f_{+}f_+^T &h_{1} \\
h_{1}^T& -\gamma
\end{pmatrix} \in\mathbb{S}_+^{n-2m_1+1}, \\
& - \begin{pmatrix}
\He(\Omega^T \hat{Y})-\gamma gg^T &h_{2} \\
h_{2}^T& -\gamma
\end{pmatrix} \in\mathbb{S}_+^{n-2m_2+1}, \begin{pmatrix}
\hat{X} & -J^T\\
-J & \hat{Y}
\end{pmatrix}\in\mathbb{S}^{n_0}_+, \\
& \begin{pmatrix}
\gamma (f^i)_{j}^2 -u_{j}^X& *\\
-(h_{10}^i)_{j} & \gamma
\end{pmatrix}\in\mathbb{S}^2_+, \begin{pmatrix}
\gamma (f^r)_{j}^2 +u_{j}^X&*\\
-(h_{10}^r)_{j} & \gamma
\end{pmatrix}\in\mathbb{S}^2_+\ (j=1, \ldots, m_1), \\
& \begin{pmatrix}
\gamma (g^i)_{j}^2 -u_{j}^Y&*\\
-(h_{20}^i)_{j} & \gamma
\end{pmatrix}\in\mathbb{S}^2_+, \begin{pmatrix}
\gamma (g^r)_{j}^2 +u_{j}^Y& *\\
-(h_{20}^r)_{j} & \gamma
\end{pmatrix}\in\mathbb{S}^2_+ \ (j=1, \ldots, m_2). 
\end{array}
\right.
\end{align}
\end{proposition}

\begin{proof}%[Proof of Proposition \ref{reduction_img}]
Since the matrix in the left side of each constraint in \eqref{LMI_img5} is a submatrix in the left side of each constraint in \eqref{LMI_img3}, any feasible solution of \eqref{LMI_img3} is also feasible for \eqref{LMI_img5} with the same objective value. It is sufficient to prove that any feasible solution of \eqref{LMI_img5} is feasible for \eqref{LMI_img3} with the same objective value. Let $(\gamma, \hat{X}, \hat{Y}, u_j^X, u_j^Y)$ be a feasible solution of \eqref{LMI_img5}. Then the solution satisfies 
\begin{align*}
\begin{pmatrix}
\gamma (f^i)_{j}^2 -u_{j}^X& -(h_{10}^i)_{j}\\
-(h_{10}^i)_{j} & \gamma
\end{pmatrix}&\in\mathbb{S}^2_+, \begin{pmatrix}
\gamma (f^r)_{j}^2 +u_{j}^X& -(h_{10}^r)_{j}\\
-(h_{10}^r)_{j} & \gamma
\end{pmatrix}\in\mathbb{S}^2_+\ (j=1, \ldots, m_1). 
\end{align*}
Applying Lemma \ref{mcomp} to the above matrices repeatedly, we see that there exists $v_{kj}\in\mathbb{R}$ \ $(j=1, \ldots, 2m_1, k=j+1, \ldots, 2m_1+1)$ such that the following matrix is positive semidefinite: %$U^X\in\mathbb{S}^{2m_1}$ such that 
\[
%-\begin{pmatrix}
%U^X -\gamma f_0f_0^T & h_{10}\\
%h_{10}^T &-\gamma
%\end{pmatrix}\in\mathbb{S}^{2m_1+1}. 
\begin{pmatrix}
\gamma (f^r)_1^2 +u_1^X & v_{21}&v_{31} &\cdots &v_{2m_1, 1} &-(h_{10}^r)_1\\
 v_{21}&\gamma (f^i)_1^2-u_1^X &v_{32} &\cdots &v_{2m_1, 2} &-(h_{10}^i)_1\\
v_{31}&v_{32}&\ddots &\ddots &\vdots &\vdots \\
\vdots&\vdots &\ddots &\gamma (f^r)_{m_1}^2+u_{m_1}^X &v_{2m_1, 2m_1-1} &-(h_{10}^r)_{m_1}\\
v_{2m_1, 1}&v_{2m_1, 2} &\cdots &v_{2m_1, 2m_1-1} &\gamma (f^i)_{m_1}^2 -u_{m_1}^X &-(h_{10}^i)_{m_1}\\
-(h_{10}^r)_1&-(h_{10}^i)_1 &\cdots &-(h_{10}^r)_{m_1}&-(h_{10}^i)_{m_1}& \gamma
\end{pmatrix}
\] 
We define $U^X_{11}\in\mathbb{S}^{2m_1}$ as follows: 
\[
(U^X_{11})_{kj} = \left\{
\begin{array}{cl}
-u_{(k+1)/2}^X & \mbox{ if } k = j \mbox{ and } k : \mbox{ odd number}, \\
u_{k/2}^X & \mbox{ if } k = j \mbox{ and } k : \mbox{ even number}, \\
\gamma (f_0)_{k}(f_0)_{j} - v_{kj}& \mbox{ if }k > j, \\
\gamma (f_0)_{k}(f_0)_{j} - v_{jk}& \mbox{ o.w.}
%\gamma (f_0)_{k}(f_0)_{j} - v_{kj} & \mbox{ o.w.}
\end{array}
\right. 
\]
Then we have 
\begin{align}\label{eq1}
-\begin{pmatrix}
U_{11}^X -\gamma f_0f_0^T & h_{10}\\
h_{10}^T &-\gamma
\end{pmatrix}&\in\mathbb{S}^{2m_1+1}_+. 
\end{align}
Similarly, there exists $U_{11}^Y\in\mathbb{S}^{2m_2}$ such that 
$-\left(\begin{smallmatrix}
U^Y_{11} -\gamma g_0g_0^T & h_{20}\\
h_{20}^T &-\gamma
\end{smallmatrix}\right)\in\mathbb{S}^{2m_2+1}_+
$. Moreover, we have 
\begin{align}\label{eq2}
- \begin{pmatrix}
\He(\Lambda^T \hat{X})-\gamma ff^T &h_{1} \\
h_{1}^T& -\gamma
\end{pmatrix} &\in\mathbb{S}_+^{n-2m_1+1}. 
\end{align}
From \eqref{eq1}, \eqref{eq2} and Lemma \ref{mcomp}, there exists $Z_{21}^X\in\mathbb{R}^{(n-2m_1)\times 2m_1}$ such that
\[
-\begin{pmatrix}
U^X_{11}-\gamma f_0f_0^T &(Z_{21}^X)^T &h_{10} \\
Z_{21}^X & \He(\Lambda^T \hat{X})-\gamma f_{+}f_+^T &h_{1} \\
h_{10}^T & h_{1}^T & -\gamma
\end{pmatrix}\in\mathbb{S}^{n+1}_+. 
\]
As well as the above, we have $Z_{21}^Y\in\mathbb{R}^{(n-2m_2)\times 2m_2}$ that satisfies
\[
-\begin{pmatrix}
U^Y_{11}-\gamma g_0g_0^T &(Z_{21}^Y)^T &h_{20} \\
Z_{21}^Y & \He(\Omega^T \hat{Y})-\gamma gg^T &h_{2} \\
h_{20}^T & h_{2}^T & -\gamma
\end{pmatrix}\in\mathbb{S}^{n+1}_+. 
\]
We define $U^X_{21}$ and $U^Y_{21}$ by $U^X_{21} = Z_{21}^X + \gamma f_+f_0^T$ and $U^Y_{21} = Z_{21}^Y + \gamma g_+g_0^T$. Then the solution 
$(\gamma, \hat{X}_{22}, \hat{Y}_{22}, U^{X}_{11}, U^X_{21}, U^Y_{11}, U^Y_{21})$
is feasible for \eqref{LMI_img3}, and thus the optimal value of \eqref{LMI_img5} is equal to the optimal value of \eqref{LMI_img3}. 
\end{proof}

We can simplify \eqref{LMI_img5} by applying the following lemma:
\begin{lemma}\label{reduction_img2}
Let $f_1, f_2, h_1, h_2, \gamma\in\mathbb{R}$. There exists $p\in\mathbb{R}$ such that \begin{align}
\label{eq3}
\begin{pmatrix}
\gamma f_1^2 - p & -h_1\\
-h_1 & \gamma
\end{pmatrix}\in\mathbb{S}^2_+ \mbox{ and} \begin{pmatrix}
\gamma f_2^2 + p & -h_2\\
-h_2 & \gamma
\end{pmatrix}\in\mathbb{S}^2_+
\end{align}
if and only if $\gamma$ satisfies 
\begin{align}\label{result}
\gamma &\ge \sqrt{
\frac{h_1^2 + h_2^2}{f_1^2 + f_2^2}
}. 
\end{align}
\end{lemma}
\begin{proof}
If $\gamma = 0$, then $h_1=h_2=0$, and thus the equivalence is obvious.  
We assume $\gamma > 0$. We see that 
\eqref{eq3} holds if and only if
$-\gamma f^2_2 + \frac{h_2^2}{\gamma}\le p\le \gamma f_1^2 - \frac{h_1^2}{\gamma}$. 
We obtain \eqref{result} from this inequality. On the other hand, if \eqref{result} holds, then we define $p$ by
$p = \frac{1}{2}\left(\gamma ( f_1^2-f^2_2) - \frac{h_1^2 - h_2^2}{\gamma}\right)$. 
Then \eqref{eq3} holds by direct computation. 
\end{proof}
Applying Lemma \ref{reduction_img2} to \eqref{LMI_img5}, it can be reformulated as follows:
\begin{align}\label{LMI_img6}
&\left\{
\begin{array}{cl}
\displaystyle\inf_{\gamma, \hat{X}, \hat{Y}} & \gamma\\
\mbox{subject to} & - \begin{pmatrix}
\He(\Lambda^T \hat{X})-\gamma ff^T &h_{1} \\
h_{1}^T& -\gamma
\end{pmatrix} \in\mathbb{S}_+^{n-2m_1+1}, \\
& - \begin{pmatrix}
\He(\Omega^T \hat{Y})-\gamma gg^T &h_{2} \\
h_{2}^T& -\gamma
\end{pmatrix} \in\mathbb{S}_+^{n-2m_2+1}, \\ 
&\begin{pmatrix}
\hat{X} & -J^T\\
-J & \hat{Y}
\end{pmatrix}\in\mathbb{S}^{n_0}_+, \\
& \gamma \ge\displaystyle\max_{j=1, \ldots, m_1}\left\{
\left|\frac{(h_{10})_{j}}{f_{j}}\right|\right\}, \gamma\ge\displaystyle\max_{j=1, \ldots, m_2}\left\{\left|\frac{(h_{20})_{j}}{g_{j}}\right|\right\}. 
\end{array}
\right.
\end{align}
Here we used $(h_{10}^r)_j^2 + (h_{10}^i)_j^2 = |(h_{10})_j|^2$ and $(f^r)_{j}^2 + (f^i)_{j}^2= |f_j|^2$ etc for the last inequalities. 

It should be noted that the optimal values of optimization problems, which appear in Section \ref{sec:case3}, that is, \eqref{LMI_img2}, \eqref{dual_img2}, \eqref{dual_img3}, \eqref{LMI_img3}, \eqref{LMI_img5} and \eqref{LMI_img6} are equal to $\gamma^*$. In fact, the optimal value of \eqref{LMI_img2} is $\gamma^*$. Since \eqref{dual_img2} is the dual of \eqref{LMI_img2}, the optimal value of \eqref{dual_img2} is $\gamma^*$. In addition, it follows from Lemma \ref{lem:dualdual} that the optimal value of \eqref{LMI_img3} is equal to the optimal value of \eqref{dual_img3}. From Proposition \ref{reduction_img} and Lemma \ref{reduction_img2}, the optimal value of \eqref{LMI_img6} is equal to the optimal values of \eqref{LMI_img5} and \eqref{LMI_img3}. Hence all the optimal values of these optimization problems are equal to $\gamma^*$.

From \eqref{LMI_img6}, we see that 
\[
\gamma^* = \max\left\{\tilde{\gamma}, \max_{j=1, \ldots, m_1}\left\{
\left|\frac{(h_{10})_{j}}{f_{j}}\right|\right\}, \displaystyle\max_{j=1, \ldots, m_2}\left\{\left|\frac{(h_{20})_{j}}{g_{j}}\right|\right\}\right\}, 
\]
where $\tilde{\gamma}$ is the optimal value of the following LMI problem
\begin{align}\label{LMI_img7}
&\left\{
\begin{array}{cl}
\displaystyle\inf_{\gamma, \hat{X}, \hat{Y}} & \gamma\\
\mbox{subject to} & - \begin{pmatrix}
\He(\Lambda^T \hat{X})-\gamma ff^T &h_{1} \\
h_{1}^T& -\gamma
\end{pmatrix} \in\mathbb{S}_+^{n-2m_1+1}, \\
& - \begin{pmatrix}
\He(\Omega^T \hat{Y})-\gamma gg^T &h_{2} \\
h_{2}^T& -\gamma
\end{pmatrix} \in\mathbb{S}_+^{n-2m_2+1}, \\ 
&\begin{pmatrix}
\hat{X} & -J^T\\
-J & \hat{Y}
\end{pmatrix}\in\mathbb{S}^{n_0}_+. 
\end{array}
\right.
\end{align}
It should be noted that if $G_{zu}$ (resp. $G_{yw}$) has a stable invariant zero, then Theorem \ref{thm:case2} is available to obtain an expression of $\tilde{\gamma}$. Otherwise, Theorem \ref{thm:case1} is available. Therefore, we obtain the following theorem from \eqref{LMI_img6}:

\begin{theorem}\label{thm:case3}
Let us consider Case \ref{Case3} stated at the final part of Section \ref{sec:Hinf}. 
We assume that invariant zeros $\lambda_1, \ldots, \lambda_{m_1}$ in
 $G_{zu}$ and $\omega_1, \ldots, \omega_{m_2}$ in $G_{yw}$ of
 \eqref{SISO} exist on the imaginary axis.  
Let $\hat{\gamma}:=\lambda_{\max}(E)$ where 
$E$ is given by \eqref{gammacase2} in Theorem \ref{thm:case2}.  
Then the optimal value $\gamma^*$ of \eqref{LMI_img2} is equal to  
\[
 \max\left\{\hat{\gamma}, 
%\sqrt{
%\displaystyle\frac{(h_{10}^r)_{i}^2 + (h_{10}^i)_{i}^2}{(f^r)_{i}^2 + (f^i)_{i}^2} 
%}, 
%\sqrt{
%\displaystyle\frac{(h_{20}^r)_{j}^2 + (h_{20}^i)_{j}^2}{(g^r)_{j}^2 + (g^i)_{j}^2} 
%}, 
\left|\frac{s_j^Tb_1}{f_j} + d_{11}\right| \ (j=1, \ldots, m_1), \left|\frac{t_j^Tc_1}{g_j} + d_{11}\right| \ (j=1, \ldots, m_2)\right\}. 
\]
In particular, if all invariant zeros of $G_{zu}$ and $G_{yw}$ 
on the imaginary axis are not the eigenvalues of $A$, respectively, then $\gamma^*$ is equal to 
\[
 \max\left\{\hat{\gamma}, 
%\sqrt{
%\displaystyle\frac{(h_{10}^r)_{i}^2 + (h_{10}^i)_{i}^2}{(f^r)_{i}^2 + (f^i)_{i}^2} 
%}, 
%\sqrt{
%\displaystyle\frac{(h_{20}^r)_{j}^2 + (h_{20}^i)_{j}^2}{(g^r)_{j}^2 + (g^i)_{j}^2} 
%}, 
\left|G_{zw}(\lambda_j)\right| \ (j=1, \ldots, m_1), \left|G_{zw}(\omega_j)\right| \ (j=1, \ldots, m_2)\right\}. 
\]
\end{theorem}
\begin{proof}
We have already mentioned that all the optimal values of \eqref{LMI_img2}, \eqref{dual_img2}, \eqref{dual_img3}, \eqref{LMI_img3}, \eqref{LMI_img5} and \eqref{LMI_img6} are equal to $\gamma^*$. 

We rewrite the last three constraints in \eqref{LMI_img6}. We have $h_{10} =s_j^Tb_1 + d_{11}f_j$ and $h_{20} = t_j^Tc_1 + d_{11}g_j$ for all $j$. 
%\begin{align*}
%h_{10} &=s_j^Tb_1 + d_{11}f_j \ (j=1, \ldots, m_1), %= (s_j^r b_1 + d_{11}f_j^r)+\sqrt{-1}(s_j^i b_1 + d_{11}f_j^i)\\
%%&= (h_{10}^r)_j + \sqrt{-1}(h_{10}^i)_j \ (j=1, \ldots, m_1), \\
%h_{20} = t_j^Tc_1 + d_{11}g_j \ (j=1, \ldots, m_2). %= (t_j^r c_1 + %
%d_{11}g_j^r)+\sqrt{-1}(t_j^i c_1 + d_{11}g_i^i) \\
%% &= (h_{20}^r)_j + \sqrt{-1}(h_{20}^i)_j \ (j=1, \ldots, m_2). 
%\end{align*}
From these equations, we obtain 
\[
\gamma \ge\left|\frac{s_j^Tb_1}{f_j} + d_{11}\right| \mbox{ and } \gamma\ge \left|\frac{t_j^Tc_1}{g_j} + d_{11}\right|
\] 
%\begin{align*}
%%\gamma \ge\sqrt{\displaystyle\frac{(h_{10}^r)_{j}^2 + (h_{10}^i)_{j}^2}{(f^r)_{j}^2 + (f^i)_{j}^2} }
%\gamma \ge\displaystyle
%\left|\frac{(h_{10})_{j}}{f_{j}}\right| &\iff \gamma \ge\displaystyle\left|\frac{s_j^Tb_1}{f_j} + d_{11}\right| \mbox{ and } 
%%\gamma \ge\sqrt{\displaystyle\frac{(h_{20}^r)_{j}^2 + (h_{20}^i)_{j}^2}{(g^r)_{j}^2 + (g^i)_{j}^2} } 
%\gamma\ge\displaystyle\left|\frac{(h_{20})_{j}}{g_{j}}\right|\iff \gamma\ge \displaystyle\left|\frac{t_j^Tc_1}{g_j} + d_{11}\right|. 
%\end{align*}
for all $j$. Therefore, we obtain the desired result. 

We prove the next statement. 
If $\lambda_j$ is not an eigenvalue of $A$, 
then we have $s_j^T = f_j c_1^T(\lambda_j I_n - A)^{-1}$. 
Similarly, $\omega_j$ is not an eigenvalue of $A$, 
then we have $t_j^T = g_j b_1^T(\omega_j I_n - A^T)^{-1}$. 
By using these expressions together with \eqref{zeroOnimg}, we obtain
\begin{align*}
\displaystyle\left|\frac{s_j^Tb_1}{f_j} + d_{11}\right| %&= \left|c_1^T (\lambda_j I_n-A)^{-1} b_1 + d_{11}\right| 
&= \left|G_{zw}(\lambda_j)\right| \mbox{ and } 
 \displaystyle\left|\frac{t_j^Tc_1}{g_j} + d_{11}\right| %&= \left|c_1^T (\omega_j I_n-A)^{-1} b_1 + d_{11}\right| 
 = \left|G_{zw}(\omega_j)\right|. 
\end{align*}
Therefore we obtain the desired result. 
\end{proof}

\section{Analysis of Case 4}\label{sec:case4}
Finally, we deal with the case where at least one of the conditions 
$d_{12}=0$ and $d_{21}=0$ holds. 
When $d_{12}=0$, for instance, the transfer function $G_{zu}$ may
be identically zero. Then  the transfer function $\Gcl(s, K)$ of the closed-loop system is $\Gcl(s, K) \equiv G_{zw}(s)$, and thus $\gamma^* = \|G_{zw}\|_{\infty}$. This also holds when $G_{yw}$ is identically zero. Therefore we consider the case where both transfer functions $G_{zu}$ and $G_{yw}$ are not identically zero.  

When $d_{12}=0$, %and $G_{zu}(s)$ is not identically zero, 
the relative degree $r_1$ of $G_{zu}$ is  positive and 
it follows from the assertion \ref{a4} in Theorem \ref{Weierstrass} and \ref{b2} of Remark \ref{remark1} that we have 
\begin{align*}
&\left\{
\begin{array}{cl}
d_{12} = 0, c_1^Tb_2 \neq 0 & \mbox{ if } r_1 =1, \\
 d_{12} = 0, c_1^TA^{r}b_2 = 0 \ (r=0, \ldots, r_1-2) \mbox{ and }  c_1^TA^{r_1-1}b_2 \neq 0 & \mbox{ if } r_1 >1. 
\end{array}
\right.
\end{align*}
Similarly, when $d_{21} = 0$, % and $G_{yw}(s)$ is not 
cally zero, 
then the relative degree $r_2$ of $G_{yw}$ is positive and we have 
\begin{align*}
&\left\{
\begin{array}{cl}
d_{21}  = 0, c_2^Tb_1 \neq 0 & \mbox{ if } r_2 =1, \\
 d_{21} = 0, c_2^TA^{r}b_1 = 0 \ (r=0, \ldots, r_2-2) \mbox{ and }  c_2^TA^{r_2-1}b_1 \neq 0& \mbox{ if } r_2 >1. 
\end{array}
\right. 
\end{align*}
Combining them with $S^TA + fc^T = \Lambda^TS^T$, $T^TA^T + gb^T = \Omega^TT^T$, $S^Tb_2 = 0$ and $T^Tc_2=0$, we obtain  
\begin{align}
\label{StAb}
S^TA^{r}b_2 &=0 \ (r=0, \ldots, r_1-1), \\
\label{TtAtc}
T^T(A^T)^{r}c_2 &= 0 \ (r=0, \ldots, r_2-1). 
\end{align}

We define $P_1\in\mathbb{R}^{n\times r_1}$, $p_1\in\mathbb{R}^{r_1}$, $P_2\in\mathbb{R}^{n\times r_2}$ and $p_2\in\mathbb{R}^{r_2}$ as follows:
\begin{align*} 
\begin{pmatrix}
P_1\\
p_1^T
\end{pmatrix} &=\left\{
\begin{array}{cl}
 \begin{pmatrix}
0 \\
1
\end{pmatrix} & \mbox{ if } r_1 = 1, \\
 \begin{pmatrix}
0 & c_1 & (A^T)c_1 & \ldots & (A^T)^{r_1-2}c_1\\
1 & 0 & 0 & \ldots & 0
\end{pmatrix} & \mbox{ if } r_1 > 1, 
\end{array}
\right.\\
\begin{pmatrix}
P_2\\
p_2^T
\end{pmatrix} &=\left\{
\begin{array}{cl}
 \begin{pmatrix}
0\\
1
\end{pmatrix} & \mbox{ if } r_2 = 1, \\
 \begin{pmatrix}
0 & b_1 & Ab_1 & \ldots & A^{r_2-2}b_1\\
1 & 0 & 0 & \ldots & 0
\end{pmatrix}& \mbox{ if } r_2 > 1. 
\end{array}
\right. 
\end{align*}
Using \eqref{StAb} and \eqref{TtAtc}, we have 
\begin{align}
\label{bAN1}
(b_2^T(A^{r})^TP_1)_{j} &=0 \ (r=0, 1, \ldots, r_1-1, j=1, \ldots, r_1 - r), \\
\label{cAtN2}
(c_2^T(A^{r})P_2)_{j} &=0 \ (r=0, 1, \ldots, r_2-1,  j=1, \ldots, r_2 - r). 
\end{align}

We provide  the perpendicular matrices of the vectors $(b_2^T, 0, 0)^T$ and $(c_2^T, 0, 0)^T$ in Lemma \ref{SfNp}. This lemma can be proved in a similar manner to the proof of Lemma \ref{lem:infinitevec}.  
\begin{lemma}\label{SfNp}
Suppose $d_{12}=0$ in $G_{zu}$ and its relative degree is $r_1$. Then we have 
\[
\begin{pmatrix}
b_2\\
0\\
0
\end{pmatrix}^\perp = \begin{pmatrix}
S & P_1 & 0\\
f^T & p_1^T & 0\\
0 & 0 & 1
\end{pmatrix}, 
\]
i.e., $S^Tb_2 = 0$, $P_1^Tb_2 = 0$ and the following square matrix is nonsingular: 
\[
\begin{pmatrix}
b_2 & S & P_1 & 0\\
0 & f^T & p_1^T & 0\\
0& 0 & 0 & 1
\end{pmatrix}. 
\]
Similarly, suppose $d_{21}=0$ in $G_{yw}$ and its relative degree is $r_2$. Then we have 
\[
\begin{pmatrix}
c_2\\
0\\
0
\end{pmatrix}^\perp = \begin{pmatrix}
T & P_2 & 0\\
g^T & p_2^T & 0\\
0 & 0 & 1
\end{pmatrix}. 
\]
\end{lemma}
%\begin{proof}
%We prove only the perpendicular of the vector $(b_2^T, 0, 0)^T$. We have $P_1^Tb_2=0$ due to \eqref{bAN1}. Since $S^Tb_2 =  0$ and the number of columns of the matrix in the right-hand side is $(n-r_1) + r_1 = n$, it is sufficient to prove that the following matrix is of full column rank:  
%\[
%\begin{pmatrix}
%S & P_1 \\
%f^T & p_1^T 
%\end{pmatrix}
%\]
%
%It is easy to prove that $(A^T)^kc_1\neq 0$ for all $k= 0, \ldots, r_1 - 2$ because  we have $c_1^TA^{r_1-1}b_2 \neq 0$.  Suppose that there exists $(\alpha^T, \beta^T)^T\in\mathbb{R}^{(n-r_1)+r_1}\setminus\{0\}$ such that 
%\[
%S\alpha + P_1\beta = 0 \mbox{ and } f^T\alpha + p_1^T\beta = 0. 
%\]
%Multiplying $b_2^TA^T$ from the left of the first equality, we obtain $\beta_{r_1} = 0$ from \eqref{bAN1} because $c_1^TA^{r_1-1}b_2\neq 0$. Next we multiply $b_2^T(A^T)^2$ from the left of the first equality, we obtain
%\[
%\beta_{r_1-1}(b_2^T(A^T)^{r_1-1}c_1) + \beta_{r_1}(b_2^T(A^T)^{r_1}c_1) = 0. 
%\]
%This implies $\beta_{r_1-1}=0$. Repeating this procedure, we obtain
% $\beta_j =0$ for all $j=2, \ldots, r_1$. Then we have $S\alpha = 0$
% and $f^T\alpha + \beta_1 = 0$. As $S$ is of full column rank due to
% \ref{F10} of Lemma \ref{lem:gplant}, we have $\alpha = 0$, and thus
% $\beta_1=0$. This contradicts to the assumption $(\alpha^T,
% \beta^T)^T\neq 0$. By following a similar line, we can prove that 
%the perpendicular of the vector $(c_2^T, 0, 0)^T$ 
%is given by the lemma.  
%\end{proof}

From Lemma \ref{SfNp}, \eqref{LMI} is equivalent to the following optimization problem: 
\begin{align}\label{LMI_sing2}
&\left\{
\begin{array}{cl}
\displaystyle\inf_{\gamma, X, Y} & \gamma\\
\mbox{subject to} &  - {\small \begin{pmatrix}
\He(\Lambda^T S^TXS) -\gamma ff^T &*& *\\
\hat{P}_1^TXS +P_1^TXS\Lambda-\gamma p_1f^T&\He(\hat{P}_1^TXP_1)-\gamma p_1p_1^T& *\\
h_1^T & \tilde{h}_1^T&-\gamma
\end{pmatrix}} \in\mathbb{S}_+^{n+1}, \\
& - {\small \begin{pmatrix}
\He(\Omega^T T^TYT) -\gamma gg^T &*&* \\
\hat{P}_2^TYT +P_2^TYT\Omega-\gamma p_2g^T&\He(\hat{P}_2^TYP_2)-\gamma p_2p_2^T&* \\
h_2^T & \tilde{h}_2^T&-\gamma
\end{pmatrix}} \in\mathbb{S}_+^{n+1}, \\
& \begin{pmatrix}
X & -I_n\\
-I_n & Y
\end{pmatrix}\in\mathbb{S}^{2n}_+. 
\end{array}
\right.
\end{align}
where 
\begin{align*}
h_1 &= S^Tb_1+ d_{11}f, h_2 = T^Tc_1 + d_{11}g, \\
\tilde{h}_1 &= P_1^Tb_1 + d_{11}p_1 = \begin{pmatrix}
d_{11} & c_1^Tb_1 & c_1^TAb_1 & \ldots & c_1^TA^{r_1-2}b_1
\end{pmatrix}^T\\
 \tilde{h}_2 &= P_2^Tc_1 + d_{11}p_2 = \begin{pmatrix}
d_{11} & b_1^Tc_1 & b_1^TA^Tc_1 & \ldots & b_1^T(A^T)^{r_2-2}c_1
\end{pmatrix}^T \\
\hat{P}_1 &= A^TP_1+c_1p_1^T = \begin{pmatrix}
c_1 & A^Tc_1 & \ldots & (A^T)^{r_1-1}c_1 
\end{pmatrix}, \\
\hat{P}_2 &= AP_2+b_1p_2^T = \begin{pmatrix}
b_1 & Ab_1 & \ldots & A^{r_2-1}b_1 
\end{pmatrix}. 
\end{align*}
Similar properties to \eqref{bAN1} and \eqref{cAtN2} hold in $\hat{P}_1$ and $\hat{P}_2$ as follows:
\begin{align}
\label{bAhN1}
(b_2^T(A^{r})^T\hat{P}_1)_{j} &=0 \ (r=0, \ldots, r_1-2, j=1, \ldots, r_1 - r-1), \\
\label{cAthN2}
(c_2^T(A^{r})\hat{P}_2)_{j} &=0 \ (r=0, \ldots, r_2-2,  j=1, \ldots, r_2 - r-1). 
\end{align}
%Applying Lemma \ref{lem:dual}, 
The following lemma provides the dual of \eqref{LMI_sing2}. We give a proof in \ref{subapp:lemma14}. 

\begin{lemma}\label{lem:dual_sing}
The dual of \eqref{LMI_sing2} can be formulated as follows:
\begin{align}\label{dual_sing2}
&\left\{
\begin{array}{cl}
\displaystyle\sup & \begin{pmatrix}
& &h_1\\
& &\tilde{h}_1\\
h_1^T & \tilde{h}_1^T&  
\end{pmatrix}\bullet Z+ \begin{pmatrix}
& &h_2\\
& &\tilde{h}_2\\
h_2^T & \tilde{h}_2^T&  
\end{pmatrix}\bullet V + \begin{pmatrix}
& I_n\\
I_n & 
\end{pmatrix}\bullet W\\
\mbox{subject to} &  \begin{pmatrix}
ff^T & fp_1^T & \\
p_1f^T & p_1p_1^T & \\
& & 1
\end{pmatrix}\bullet Z + \begin{pmatrix}
gg^T & gp_2^T & \\
p_2g^T & p_2p_2^T & \\
& & 1
\end{pmatrix}\bullet V = 1, \\ 
& W_{11} = \He(S\Lambda(Z_{11}S^T+ Z_{21}^TP_1^T) + \hat{P}_1(Z_{21}S^T +Z_{22}P_1^T)), \\
& W_{22} = \He(T\Omega(V_{11}T^T+ V_{21}^TP_2^T) + \hat{P}_2(V_{21}T^T + V_{22}P_2^T)), \\
&Z =  \begin{pmatrix}
Z_{11} & Z_{21}^T & Z_{31}^T\\
Z_{21} & Z_{22} & Z_{32}^T\\
Z_{31} & Z_{32} & Z_{33}
\end{pmatrix}\in\mathbb{S}^{n+1}_+, V = \begin{pmatrix}
V_{11} & V_{21}^T & V_{31}^T\\
V_{21} & V_{22} & V_{32}^T\\
V_{31} & V_{32} & V_{33}
\end{pmatrix}\in\mathbb{S}^{n+1}_+, \\
& Z_{21}\in\mathbb{R}^{(n-r_1)\times r_1}, Z_{22}\in\mathbb{R}^{r_1\times r_1},  V_{21}\in\mathbb{R}^{(n-r_2)\times r_2}, V_{22}\in\mathbb{R}^{r_2\times r_2}, \\ 
& \begin{pmatrix}
W_{11} & W_{21}^T\\
W_{21} & W_{22}
\end{pmatrix}\in\mathbb{S}^{2n}_+. 
\end{array}
\right.
\end{align}
Moreover, the duality gap between \eqref{LMI_sing2} and \eqref{dual_sing2} is zero. 
\end{lemma}

From the following lemma, we can reduce the size of \eqref{dual_sing2} equivalently. We give a proof in \ref{subapp:lemma15}.
\begin{lemma}\label{zeroDual}
Let $(Z_{ij}, V_{ij}, W_{ij})$ be a feasible solution of \eqref{dual_sing2}. Then 
$Z_{21} =O_{r_1\times(n-r_1)}$ and  $V_{21} =O_{r_2\times(n-r_2)}$. In addition, if $r_1 > 1$, we have $(Z_{22})_{kr} = 0$ for all $(k, r)\neq (1, 1)$ and $(Z_{32})_{k} = 0 \ (k=2, \ldots, r_1)$. Similarly, if $r_2 > 1$, we have $(V_{22})_{kr} = 0$  for all $(k, r)\neq (1, 1)$ and $(V_{32})_{k} = 0 \ (k=2, \ldots, r_2)$. 
\end{lemma}

\begin{remark}\label{remark3}
As well as Remark \ref{facialred} and \ref{remark2}, we have reduced the size of the matrix variables $Z$ and $V$ in \eqref{dual_sing2} in Lemma \ref{zeroDual}. This reduction also corresponds to the facial reduction for SDP as well as Lemmas \ref{lem:dstructure} and \ref{reduction}.   
\end{remark}

By applying Lemma \ref{zeroDual} to \eqref{dual_sing2} and substituting zeros in $(Z, V, W)$, we can reformulate it as follows. 
\begin{align}\label{dual_sing3}
&\left\{
\begin{array}{cl}
\displaystyle\sup & \begin{pmatrix}
& &h_1\\
& &d_{11}\\
h_1^T & d_{11}&  
\end{pmatrix}\bullet Z+ \begin{pmatrix}
& &h_2\\
& &d_{11}\\
h_2^T & d_{11}&  
\end{pmatrix}\bullet V + \begin{pmatrix}
& I_n\\
I_n & 
\end{pmatrix}\bullet W\\
\mbox{subject to} &  \begin{pmatrix}
ff^T & f & \\
f^T & 1& \\
& & 1
\end{pmatrix}\bullet Z + \begin{pmatrix}
gg^T & g & \\
g^T & 1 & \\
& & 1
\end{pmatrix}\bullet V = 1, \\ 
& W_{11} = S\He(\Lambda Z_{11})S^T, W_{22} = T\He(\Omega V_{11})T^T,  \begin{pmatrix}
W_{11} & W_{21}^T\\
W_{21} & W_{22}
\end{pmatrix}\in\mathbb{S}^{2n}_+, \\
& Z = \begin{pmatrix}
Z_{11} & O & Z_{31}^T\\
O & Z_{22} & Z_{32}^T\\
Z_{31} & Z_{32} & Z_{33}
\end{pmatrix}\in\mathbb{S}^{n-r_1+2}_+, V = \begin{pmatrix}
V_{11} & O & V_{31}^T\\
O & V_{22} & V_{32}^T\\
V_{31} & V_{32} & V_{33}
\end{pmatrix}\in\mathbb{S}^{n-r_2+2}_+. 
\end{array}
\right.
\end{align}
It follows from constraints on $W_{ij}$ in \eqref{dual_sing3} that there exists $\hat{W}_{11}\in\mathbb{S}^{n-r_1}$, $\hat{W}_{21}\in\mathbb{R}^{(n-r_2)\times (n-r_1)}$ and $\hat{W}_{22}\in\mathbb{S}^{n-r_2}$ such that  
\[
\begin{pmatrix}
\hat{W}_{11} & \hat{W}_{21}^T\\
\hat{W}_{21} & \hat{W}_{22}
\end{pmatrix}\in\mathbb{S}^{(n-r_1)+(n-r_2)}_+, W_{11}=S\hat{W}_{11}S^T, W_{22} = T\hat{W}_{22}T^T \mbox{ and } W_{21} = T\hat{W}_{21}S^T, 
\]
and thus \eqref{dual_sing3} can be reformulated as 
\begin{align}\label{dual_sing4}
&\left\{
\begin{array}{cl}
\displaystyle\sup &  \begin{pmatrix}
& &h_1\\
& &d_{11}\\
h_1^T & d_{11}&  
\end{pmatrix}\bullet Z+ \begin{pmatrix}
& &h_2\\
& &d_{11}\\
h_2^T & d_{11}&  
\end{pmatrix}\bullet V + \begin{pmatrix}
& S^TT\\
T^TS & 
\end{pmatrix}\bullet \hat{W}\\
\mbox{subject to} &  \begin{pmatrix}
ff^T & f & \\
f^T & 1& \\
& & 1
\end{pmatrix}\bullet Z + \begin{pmatrix}
gg^T & g & \\
g^T & 1 & \\
& & 1
\end{pmatrix}\bullet V = 1, \\ 
& \hat{W}_{11} = \He(\Lambda Z_{11}), \hat{W}_{22} = \He(\Omega V_{11}),  \hat{W}=\begin{pmatrix}
\hat{W}_{11} & \hat{W}_{21}^T\\
\hat{W}_{21} & \hat{W}_{22}
\end{pmatrix}\in\mathbb{S}^{n_\infty}_+, \\
& Z = \begin{pmatrix}
Z_{11} & O & Z_{31}^T\\
O & Z_{22} & Z_{32}^T\\
Z_{31} & Z_{32} & Z_{33}
\end{pmatrix}\in\mathbb{S}^{n-r_1+2}_+, V = \begin{pmatrix}
V_{11} & O & V_{31}^T\\
O & V_{22} & V_{32}^T\\
V_{31} & V_{32} & V_{33}
\end{pmatrix}\in\mathbb{S}^{n-r_2+2}_+. 
\end{array}
\right.
\end{align}
where $n_\infty = 2n-r_1-r_2$. 
%Applying a similar discussion in Section \ref{sec:case2} to \eqref{dual_sing4},  we obtain 
%\begin{align}\label{dual_sing5}
%&\left\{
%\begin{array}{cl}
%\displaystyle\sup & \begin{pmatrix}
%& &*\\
%& &*\\
%h_{1+}^T & d_{11}&  
%\end{pmatrix}\bullet Z+ \begin{pmatrix}
%& &*\\
%& &*\\
%h_{2+}^T & d_{11}&  
%\end{pmatrix}\bullet V + \begin{pmatrix}
%& *\\
%T_+^TS_+ & 
%\end{pmatrix}\bullet \hat{W}\\
%\mbox{subject to} &  \begin{pmatrix}
%f_+f_+^T & f_+ & \\
%f_+^T & 1& \\
%& & 1
%\end{pmatrix}\bullet Z + \begin{pmatrix}
%g_+g_+^T & g_+ & \\
%g_+^T & 1 & \\
%& & 1
%\end{pmatrix}\bullet V = 1, \\ 
%& \hat{W}_{11} = \He(\Lambda_+ Z_{11}), \hat{W}_{22} = \He(\Omega_+V_{11}),  \hat{W}=\begin{pmatrix}
%\hat{W}_{11} & \hat{W}_{21}^T\\
%\hat{W}_{21} & \hat{W}_{22}
%\end{pmatrix}\in\mathbb{S}^{k_1+k_2}_+, \\
%& Z = \begin{pmatrix}
%Z_{11} & O & Z_{31}^T\\
%O & Z_{22} & Z_{32}^T\\
%Z_{31} & Z_{32} & Z_{33}
%\end{pmatrix}\in\mathbb{S}^{k_1+2}_+, V = \begin{pmatrix}
%V_{11} & O & V_{31}^T\\
%O & V_{22} & V_{32}^T\\
%V_{31} & V_{32} & V_{33}
%\end{pmatrix}\in\mathbb{S}^{k_2+2}_+. 
%\end{array}
%\right.
%\end{align}
The next lemma provides the dual of \eqref{dual_sing4}. 
\begin{lemma}\label{lem:dual_sing2}
The dual of  \eqref{dual_sing4} can be formulated as follows: 
\begin{align}\label{LMI_sing4}
&\left\{
\begin{array}{cl}
\displaystyle\inf_{\gamma, \hat{X}, \hat{Y}, \xi, \eta} & \gamma\\
\mbox{subject to} & - \begin{pmatrix}
\He(\Lambda^T \hat{X}) -\gamma ff^T &\xi - \gamma f & h_{1}\\
\xi^T-\gamma f^T& -\gamma & d_{11} \\
h_{1}^T & d_{11}& -\gamma
\end{pmatrix} \in\mathbb{S}_+^{n-r_1+2},   \\
& - \begin{pmatrix}
\He(\Omega^T \hat{Y}) -\gamma gg^T & \eta-\gamma g^T& h_{}\\
\eta^T-\gamma g^T& -\gamma & d_{11} \\
h_{2}^T & d_{11} & -\gamma
\end{pmatrix} \in\mathbb{S}_+^{n-r_2+2}, \\ 
& \begin{pmatrix}
\hat{X} & -J^T\\
-J& \hat{Y}
\end{pmatrix}\in\mathbb{S}^{n_\infty}_+, \xi\in\mathbb{R}^{n-r_1}, \eta\in\mathbb{R}^{n-r_2}. 
\end{array}
\right.
\end{align}
Moreover, the duality gap between \eqref{LMI_sing4} and \eqref{dual_sing4} is zero, and \eqref{dual_sing4} has an optimal solution. 
%both \eqref{LMI_sing5} and \eqref{dual_sing4} have optimal solutions. 
\end{lemma}
\begin{proof}
We can prove by applying  similar arguments in Lemmas \ref{lem:dual} that the dual of \eqref{dual_sing4} is \eqref{LMI_sing4}. The zero duality gap  between \eqref{LMI_sing4} and \eqref{dual_sing4}, and the existence of an optimal solution of \eqref{dual_sing4} can be proved by a similar way in \ref{subapp:proof3}. %In addition, we can prove the existence of an optimal solution of \eqref{LMI_sing5} in a similar way in  \ref{subapp:proof2}. 
\end{proof}

For \eqref{LMI_sing4}, we consider the following LMI problem:
\begin{align}\label{LMI_sing6}
&\left\{
\begin{array}{cl}
\displaystyle\inf_{\gamma, \hat{X}, \hat{Y}} & \gamma\\
\mbox{subject to} & - \begin{pmatrix}
\He(\Lambda^T \hat{X}) -\gamma ff^T  &h_{1}\\
h_{1}^T &  -\gamma
\end{pmatrix} \in\mathbb{S}_+^{n-r_1+1}, 
  \begin{pmatrix}
\hat{X} & -J^T\\
-J & \hat{Y}
\end{pmatrix}\in\mathbb{S}^{n_\infty}_+, \\
& - \begin{pmatrix}
\He(\Omega^T \hat{Y}) -\gamma gg^T & h_{2}\\ 
h_{2}^T &  -\gamma
\end{pmatrix} \in\mathbb{S}_+^{n-r_2+1},  \begin{pmatrix}
\gamma & -d_{11}\\
-d_{11} & \gamma
\end{pmatrix}\in\mathbb{S}^2_+. 
\end{array}
\right.
\end{align}
The following proposition holds for \eqref{LMI_sing4} and \eqref{LMI_sing6}: 

\begin{proposition}\label{equal}
The optimal value of \eqref{LMI_sing6} is equal to the optimal value of \eqref{LMI_sing4}. 
%Let $\hat{\gamma}$ and $\tilde{\gamma}$ be the optimal values of \eqref{LMI_sing4} and \eqref{LMI_sing6}, respectively. Then we have $\hat{\gamma} = \tilde{\gamma}$. 
\end{proposition}
\begin{proof}
Let $\hat{\gamma}$ and $\tilde{\gamma}$ be the optimal values of \eqref{LMI_sing4} and \eqref{LMI_sing6}, respectively. 
Since any feasible solution of \eqref{LMI_sing4} is also feasible for \eqref{LMI_sing6} with the same objective value, we have $\hat{\gamma}\ge \tilde{\gamma}$. For this, it is sufficient to prove $\hat{\gamma}\le \tilde{\gamma}$. 

Let $(\gamma, \hat{X}, \hat{Y})$ be a feasible solution for \eqref{LMI_sing6}. 
%We can prove in a similar manner to Lemma \ref{lem:exist} that \eqref{LMI_sing6} has an optimal solution. As \eqref{LMI_sing6} has an optimal solution, 
If $\gamma=0$, then $d_{11} = 0$, $h_{1}=h_{2}=0$. In addition, as we have $\hat{X}=O$, $\hat{Y}=O$, $J=O$. This is also feasible for \eqref{LMI_sing4} by taking $\xi=\eta=0$. Hence $\gamma=0$. 

We assume that $\gamma > 0$. We define $\xi$ and $\eta$  by 
$\xi = \gamma f + \frac{d_{11}}{\gamma}h_{1}$ and $\eta = \gamma g + \frac{d_{11}}{\gamma}h_{2}$, respectively. It follows Lemma \ref{mcomp} that $(\gamma, \hat{X}, \hat{Y}, \xi, \eta)$ satisfies 
\begin{align*}
- \begin{pmatrix}
\He(\Lambda^T \hat{X}) -\gamma ff^T & *&* \\
\xi^T-\gamma f^T& -\gamma &  *\\
h_{1}^T & d_{11}& -\gamma
\end{pmatrix} & \in\mathbb{S}_+^{n-r_1+2}, \\
-  \begin{pmatrix}
\He(\Omega^T \hat{Y}) -\gamma gg^T & *&*\\
\eta^T-\gamma g^T& -\gamma & *\\
h_{2}^T & d_{11} & -\gamma
\end{pmatrix} &\in\mathbb{S}_+^{n-r_2+2}. 
\end{align*}
Hence $(\gamma, \hat{X}, \hat{Y}, \xi, \eta)$ is feasible for \eqref{LMI_sing4} with the same objective value as $\gamma$. Therefore we have $\hat{\gamma} =\tilde{\gamma}$. 
\end{proof}

It should be noted that the optimal values of optimization problems, which appear in Section \ref{sec:case4}, that is, \eqref{LMI_sing2}, \eqref{dual_sing2}, \eqref{dual_sing3}, \eqref{dual_sing4}, \eqref{LMI_sing4} and \eqref{LMI_sing6} are equal to $\gamma^*$. In fact, the optimal value of \eqref{LMI_img2} is $\gamma^*$. Since \eqref{dual_sing2} is the dual of \eqref{LMI_sing2}, it follows from Lemma \ref{lem:dual_sing} that the optimal value of \eqref{dual_sing2} is $\gamma^*$. Applying Lemma \ref{zeroDual}, we see that the optimal value of \eqref{dual_sing3} is equal to $\gamma^*$. As we have seen, the optimal value of \eqref{dual_sing4} is also equal to $\gamma^*$. It follows from Lemma \ref{lem:dual_sing2} that the optimal value of \eqref{LMI_sing4} is equal to \eqref{dual_sing4}. Finally, form Proposition \ref{equal}, we see that the optimal value of \eqref{LMI_sing6} is equal to the optimal value of \eqref{LMI_sing4}.  Hence all the optimal values of these optimization problems are equal to $\gamma^*$. 

From \eqref{LMI_sing6},  we see that 
\[
\gamma^* = \max\left\{\tilde{\gamma}, |d_{11}|\right\}, 
\]
where $\tilde{\gamma}$ is the optimal value of the following LMI problem
\begin{align}\label{LMI_sing7}
&\left\{
\begin{array}{cl}
\displaystyle\inf_{\gamma, \hat{X}, \hat{Y}} & \gamma\\
\mbox{subject to} & - \begin{pmatrix}
\He(\Lambda^T \hat{X})-\gamma ff^T  &h_{1}  \\
h_{1}^T& -\gamma
\end{pmatrix} \in\mathbb{S}_+^{n-r_1+1},  \\
& - \begin{pmatrix}
\He(\Omega^T \hat{Y})-\gamma gg^T  &h_{2}  \\
h_{2}^T& -\gamma
\end{pmatrix} \in\mathbb{S}_+^{n-r_2+1}, \\ 
&\begin{pmatrix}
\hat{X} & -J^T\\
-J & \hat{Y}
\end{pmatrix}\in\mathbb{S}^{n_\infty}_+. 
\end{array}
\right.
\end{align}
It should be noted that if $G_{zu}$ (resp. $G_{yw}$) has an invariant zero on the imaginary axis, then Theorem \ref{thm:case3} is available to obtain an expression of $\tilde{\gamma}$. Also, if $G_{zu}$ (resp. $G_{yw}$) has a stable invariant zero, then Theorem \ref{thm:case2} is available to obtain an expression of $\tilde{\gamma}$. Otherwise, Theorem \ref{thm:case1} is available. Therefore, we obtain the following theorem. % from \eqref{LMI_sing6}: 
%We obtain the main result from Proposition \ref{equal}. 
\begin{theorem}\label{thm:case4}
Let us consider Case \ref{Case1} stated at the final part of Section \ref{sec:Hinf}. We assume that invariant zeros $\lambda_1, \ldots, \lambda_{m_1}$  in
 $G_{zu}$ and $\omega_1, \ldots, \omega_{m_2}$ in $G_{yw}$ of
 \eqref{SISO} exist on the imaginary axis. Define $\hat{\gamma}:=\lambda_{\max}(E)$ where 
$E$ is given by \eqref{gammacase2} in Theorem \ref{thm:case2}. 
Then the optimal value $\gamma^*$ of \eqref{LMI_img2} is equal to  
\[
 \max\left\{\hat{\gamma}, 
%\sqrt{
%\displaystyle\frac{(h_{10}^r)_{i}^2 + (h_{10}^i)_{i}^2}{(f^r)_{i}^2 + (f^i)_{i}^2} 
%}, 
%\sqrt{
%\displaystyle\frac{(h_{20}^r)_{j}^2 + (h_{20}^i)_{j}^2}{(g^r)_{j}^2 + (g^i)_{j}^2} 
%}, 
\left|\frac{s_j^Tb_1}{f_j} + d_{11}\right| \ (j=1, \ldots, m_1),  \left|\frac{t_j^Tc_1}{g_j} + d_{11}\right| \ (j=1, \ldots, m_2), |G_{zw}(\infty)|\right\}. 
\]
If $G_{zu}$ (resp. $G_{yw})$ has no invariant zeros on the imaginary axis, then $\left|\frac{s_j^Tb_1}{f_j} + d_{11}\right|$ (resp. $\left|\frac{t_j^Tc_1}{g_j} + d_{11}\right|$) is vanished from the above expression of $\gamma^*$. 
\end{theorem} 
\begin{proof}
We have already seen that all the of optimal values of \eqref{LMI_sing2}, \eqref{dual_sing2}, \eqref{dual_sing3}, \eqref{dual_sing4}, \eqref{LMI_sing4} and \eqref{LMI_sing6} are equal to $\gamma^*$. 
%We prove that all the optimal values of \eqref{LMI_sing2}, \eqref{dual_sing2}, \eqref{dual_sing4}, \eqref{LMI_sing4} and \eqref{LMI_sing6}. The equivalence between \eqref{LMI_sing2} and  \eqref{dual_sing2} follows from Lemma \ref{lem:dual_sing}. We obtain  \eqref{dual_sing4} from \eqref{dual_sing2} by the structure of solutions in \eqref{dual_sing2}, e.g., see Lemma \ref{zeroDual}. The equivalence between \eqref{LMI_sing4} and  \eqref{dual_sing4} follows from Lemma \ref{lem:dual_sing2}. Proposition \ref{equal} ensures that \eqref{LMI_sing4} is equivalent to \eqref{LMI_sing6} when at least one of $d_{12}=0$ and $d_{21}=0$ holds. Thus we obtain $\gamma^*=\max\{\hat{\gamma}, |d_{11}|\}$. In addition, 
As $|d_{11}| = |G_{zw}(\infty)|$, we obtain the desired result. %\lim_{s\to\infty}|G_{zw}(s)|$. 
%Combining them, we obtain the desired result. 
\end{proof}

\section{Application of Theorem \ref{thm:main} to limitation analysis}\label{sec:application}
We here provide an application of Theorem \ref{thm:main} to the limitation analysis of $H_\infty$ output feedback control for sensitivity function. This application has been already considered in \cite[Theorem 5.1]{Chen00} via Nevanlinna-Pick interpolation. We provide the same result for SISO dynamical systems from Theorem \ref{thm:main}. 

Consider the following generalized plant. 
\begin{align}\label{SISOsenfun}
&\left\{
\begin{array}{lll}
\dot{x} &=& Ax + bu\\
z & = & c^Tx + w \\
y &=& c^Tx + w
\end{array}
\right. 
\end{align}
In addition to Assumption \ref{A}, we impose that \eqref{SISOsenfun} has no invariant zeros on the imaginary axis. 
Then we can rewrite \eqref{Partition.vec} as follows.
\begin{align}\label{Partition.sen.z}
\begin{pmatrix}
S_-^T&f_-\\
S_+^T&f_+
\end{pmatrix}
\begin{pmatrix}
A & b \\
c^T & 0
\end{pmatrix} &= \begin{pmatrix}
\Lambda_-^T &\\
& \Lambda_+^T
\end{pmatrix}
\begin{pmatrix}
S_-^T&0\\
S_+^T&0 
\end{pmatrix}, \\ \label{Partition.sen.p}\begin{pmatrix}
T_-^T&g_-\\
T_+^T&g_+
\end{pmatrix}
\begin{pmatrix}
A^T & c \\
0 & 1
\end{pmatrix} &= \begin{pmatrix}
\Omega_-^T &\\
& \Omega_+^T
\end{pmatrix}
\begin{pmatrix}
T_-^T&0\\
T_+^T&0 
\end{pmatrix}. 
\end{align}
Then any eigenvalue of $\Lambda_+$ (resp. $\Omega_+$) is an unstable zero (resp. pole) of \eqref{SISOsenfun}. Let $\mathcal{Z}_{++}$ and $\mathcal{P}_{++}$ be the sets of unstable zeros and poles in \eqref{SISOsenfun}, respectively. In addition, $h_{1+}$ and $h_{2+}$ are rewritten by $h_{1+} = S_+^Tb_1 + d_{11}f_+ = f_+$ and $h_{2+} = T_+^Tc_1 + d_{11}g_+ = T_+^Tc + g_+ = 0$, and thus $H_{1+} = F_+$ and $H_{2+} = O$. By using those equations, we can simplify $\gamma^*$ in Theorem \ref{thm:main} as follows. 
 \begin{align*}
 \gamma^* &= \max\left\{\hat{\gamma}, 1\right\}, \mbox{ where} \\
 \hat{\gamma} &= \lambda_{\max}\begin{pmatrix}
 O & F_+^{-1/2}J_+^TG_+^{-1/2} &I & O \\
G_+^{-1/2}J_+F^{-1/2}_+& O & O & O\\
 I & O & O &O\\
 O & O & O & O
 \end{pmatrix} \\
 &= \sqrt{
1 + \sigma_{\max}^2(G_+^{-1/2}J_+F^{-1/2}_+)
}. 
 \end{align*}
 In fact, since $F_+$ and $G_+$ are positive definite, we have 
\begin{align*}
 \gamma &\ge \lambda_{\max}\begin{pmatrix}
 O & F_+^{-1/2}J_+^TG_+^{-1/2} &I & O \\
G_+^{-1/2}J_+F^{-1/2}_+& O & O & O\\
 I & O & O &O\\
 O & O & O & O
 \end{pmatrix} \\
 &\iff \begin{pmatrix}
\gamma I & F_+^{-1/2}J_+^TG_+^{-1/2} &I & O \\
G_+^{-1/2}J_+F^{-1/2}_+& \gamma I & O & O\\
 I & O & \gamma I &O\\
 O & O & O & \gamma I
 \end{pmatrix}\succeq O\\
 &\iff \gamma I - \frac{1}{\gamma}(I + F_+^{-1/2}J_+^TG_+^{-1}J_+F_+^{-1/2}) \succeq O\\
 & \iff \gamma \ge \sqrt{
1 + \sigma_{\max}^2(G^{-1/2}_+J_+F_+^{-1/2})
}. %, 
\end{align*}

From this result, we can obtain some results in \cite{Chen00, Ebihara16b}: 
\begin{example}
We consider the case where (\ref{SISO}) has a solo unstable zero $z$ with degree 1 and a solo unstable pole $p$ with degree 1. Furthermore, assume $d=1$. Then we have 
\[
F_+ = \frac{f^2}{2z}, G_+ = \frac{g^2}{2p} \mbox{ and } J_+ = \frac{fg}{p-z}. 
\]
Hence $G_+^{-1/2}J_+F_+^{-1/2} = 2\frac{\sqrt{zp}}{p-z}$ and $\gamma^* = \left|\frac{p+z}{p-z}\right|$. 
\end{example}

\begin{example}
We consider the case where (\ref{SISO}) has $k$ unstable zero $z_1, \ldots, z_k$ with degree 1 and $k$ unstable poles $p_1, \ldots, p_k$ with degree 1. This implies that both matrices $\Lambda_+$ and $\Omega_+$ are diagonal. Furthermore, assume $d=1$. Then $F_+, G_+\in\mathbb{S}^{k}$ and $J_+$ can be written as follows:  
\[
F_+ = \left(\displaystyle \frac{f_{i}f_{j}}{z_i+z_j}\right)_{1\le i, j\le k}, 
G_+ = \left(\displaystyle \frac{g_{i}g_{j}}{p_i+p_j}\right)_{1\le i, j\le k} 
\mbox{ and } J_+ = \left(\displaystyle \frac{g_{i}f_{j}}{p_i-z_j}\right)_{1\le i, j\le k}. 
\]
In this case, the result in this subsection (seems to) coincide to \cite[Theorem 5.1]{Chen00}. Furthermore, in the case where (\ref{SISO}) has $k$ unstable zero $z_1, \ldots, z_k$ with degree 1 and one unstable pole $p$ with degree 1, we obtain by using Symbolic Math Toolbox \cite{symbolic}
\[
\sigma_{\max}(G^{-1/2}_+J_+F_+^{-1/2}) = \frac{\sqrt{p(z_1 + z_2)(p^2 + z_1z_2)}}{|p - z_1||p - z_2|}, 
\]
and thus $\gamma^* = \displaystyle\frac{|p+z_1||p+z_2|}{|p-z_1||p-z_2|}$. 
\end{example}
\begin{example}
We consider the case where (\ref{SISO}) has a solo unstable zero $z$ with degree 2 and a solo unstable poles $p$ with degree 2. This implies that both matrices $\Lambda_+$ and $\Omega_+$ have the forms of 
\[
\Lambda_+ = \begin{pmatrix}
z & 1 \\
0 & z
\end{pmatrix} \mbox{ and }\Omega_+ = \begin{pmatrix}
p & 1 \\
0 & p
\end{pmatrix}. 
\]
Then 
\[
F_+ = 
\begin{pmatrix}
\displaystyle\frac{1}{2z} &\displaystyle\frac{1}{4z^2} \\
\displaystyle\frac{1}{4z^2} &\displaystyle\frac{1}{4z^3}
\end{pmatrix}, G_+ =\begin{pmatrix}
\displaystyle\frac{1}{2p} & \displaystyle\frac{1}{4p^2}\\
\displaystyle\frac{1}{4p^2} &\displaystyle\frac{1}{4p^3}
\end{pmatrix} \mbox{ and } J_+ = \begin{pmatrix}
\displaystyle\frac{1}{z-p} & \displaystyle\frac{1}{(z-p)^2}\\
-\displaystyle\frac{1}{z-p} & \displaystyle\frac{2}{(z-p)^3}
\end{pmatrix}. 
\]
In this case, we have 
\[
\sigma_{\max}\left(G^{-1/2}_+J_+F_+^{-1/2}\right) = 2\displaystyle\frac{\sqrt{pz}}{(z-p)^3}\left(
(p+z)^4 + \sqrt{p^4+14p^2z^2+z^4}
\right). 
\]
This is the same as the result in \cite[Theorem 2]{Ebihara16b}. 
\end{example}

\section{Conclusion}\label{sec:conclusion}
We considered the LMI problem of $H_\infty$ output feedback control problem for the SISO dynamical system \eqref{SISO}. We assumed the stabilizability of $(A, b_2)$ and the detectability of $(A, c_2^T)$. In addition to these assumptions, we impose some technical assumptions for simplicity. Then we provided an explicit form of the optimal value of the LMI problem. When all invariant zeros of $G_{zu}$ and $G_{yw}$ are in the open right half plane, the Schur complement and Lyapunov equation are useful to derive the explicit form. Otherwise, we had seen that the dual problem is not strictly feasible. Then facial reduction is applicable to reduce the size of the dual problem. As a result, the LMI problem of the reduced dual problem is also simplified. 

Our explicit form of the optimal value is the unification of some results in the literature of $H_\infty$ performance limitation analysis. For instance, we had seen that we obtain the same results in \cite{Chen00, Ebihara16b}. 

We considered the case of the SISO time-invariant dynamical system. It is natural to consider the case of the MIMO time-invariant dynamical system. Our analysis will be easily extended to the case of MIMO with $m_1=m_2=p_1=p_2$. The Weierstrass form is still useful in the case. Otherwise, we will need to consider the Kronecker canonical forms of the transfer functions rather than their Weierstrass forms because their Rosenbrock system matrices are not square. A variant of Kronecker canonical form developed in \cite[Chapter 1]{SchererPhD} may also be useful in the simplification and reduction of the LMI problem obtained from a general MIMO system. \cite{SchererPhD} provided a simplification of Riccati equations and inequalities obtained from $H_\infty$ control problem. This consideration is not straightforward, and the extension of our explicit form is future study. 
\appendix

\section{Fundamental facts on semidefinite program}\label{app:sdp}
To prove the zero duality gap in some theorems and the existence of optimal solutions of some optimization problems 
%Theorem \ref{thm:duality} and  
%Lemma \ref{lem:dualdual2}, and so on, 
%and to prove the existence of optimal solutions of \eqref{LMI2}
under Assumption \ref{A1}, %the assumptions in Section \ref{sec:case1}, 
we need to use the strong duality in Theorems \ref{sdt} and
\ref{thm:fr} given below. 
For the statements of these two theorems, 
however, we first need to introduce some notation and symbols on semidefinite program.  

Let us consider the LMI problem
 \begin{align}\label{lmi_def}
 \theta^*_P&=\inf%_{y_j}
 \left\{
 d^T y : \sum_{j=1}^m y_j L_{ij} - L_{i0} \in\mathbb{S}^{n_i}_+ \ (i=1, \ldots, p), y\in\mathbb{R}^m
 \right\} 
 \end{align}
where $L_{i0}, \ldots, L_{im}\in\mathbb{S}^{n_i} \ (i=1, \ldots, p)$ and
$d\in\mathbb{R}^m$.  
The problem \eqref{lmi_def} is said to be {\itshape strictly feasible} 
if there exists $y_j \ (j=1, \ldots, m)$ such that 
$\sum_{j=1}^m y_j L_{ij} - L_{i0} \in\mathbb{S}^{n_i}_{++}$ for all
$i=1, \ldots, p$.  
On the other hand, its dual can be formulated as follows:  
 \begin{align}\label{lmi_dual_def}
 \theta^*_D&=\sup_{X_i} \left\{\sum_{i=1}^p L_{i0}\bullet X_{i} : \begin{array}{l}
 \displaystyle\sum_{i=1}^p L_{ij}\bullet X_{i} = d_j \ (j=1, \ldots, m), \\
 X_i \in\mathbb{S}^{n_i}_+ \ (i=1, \ldots, p)
 \end{array}
 \right\}.  
 \end{align}
The problem \eqref{lmi_dual_def} is said to be {\itshape strictly feasible} 
if there exists $X_i\in\mathbb{S}^{n_i}_{++} \ (i=1, \ldots, p)$ 
such that $\sum_{i=1}^p L_{ij}\bullet X_{i} = d_j $ for all $j=1, \ldots, m$. 
We call the value $\theta^*_P-\theta^*_D$ 
{\itshape the duality gap between \eqref{lmi_def} and \eqref{lmi_dual_def}}.
 
For every feasible solution $y_j$ of \eqref{lmi_def} and $X_i$ of
\eqref{lmi_dual_def}, 
we have $d^Ty \ge \sum_{i=1}^p L_{i0}\bullet X_{i}$. 
This inequality is called {\itshape the weak duality} for
\eqref{lmi_def} and \eqref{lmi_dual_def}. 
The weak duality implies $\theta^*_P\ge \theta^*_D$, i.e., the duality gap is nonnegative. 
It is well-known that the duality gap between \eqref{lmi_def} and
\eqref{lmi_dual_def} is zero, 
i.e. $\theta^*_P = \theta^*_D$ holds under a mild assumption. 
This is called {\itshape the strong duality} for \eqref{lmi_def} and
\eqref{lmi_dual_def}. 
We summarize the details of the strong duality in the next theorem.  
\begin{theorem}\label{sdt}
(see e.g., \cite[Theorem 2.2]{deKlerk02})%;  
%see also \cite[Theorem 3.2.8]{Renegar01} and \cite[Theorem 4.7.1]{Gartner12})
If \eqref{lmi_def} is strictly feasible and 
the optimal value is bounded below, 
then $\theta_P^*=\theta_D^*$ and \eqref{lmi_dual_def} has an optimal solution. 
Similarly, if \eqref{lmi_dual_def} is strictly feasible 
and the optimal value is bounded above, 
then $\theta_P^*=\theta_D^*$ and \eqref{lmi_def} has an optimal solution. 
\end{theorem}

Finally, we provide a known fact on the strict feasibility of
\eqref{lmi_def} and \eqref{lmi_dual_def}. 
\begin{theorem}\label{thm:fr}
(see e.g., \cite[Lemmas 1 and 2]{Trnovska05})%, 
%see also \cite[Lemma 28.4]{Pataki13} and \cite[Lemma 3.2]{Waki13}) 
For \eqref{lmi_def}, exactly one of 
the following two statements is true:
\begin{enumerate}
\item\label{P1} \eqref{lmi_def} is strictly feasible. 
\item\label{P2} There exist $\hat{X}_i\in\mathbb{S}^{n_i}_+$ \ $(i=1, \ldots, p)$ such that 
   at least one of $\hat{X}_i$ is nonzero, 
   $\sum_{i=1}^p L_{i0}\bullet \hat{X}_i \ge 0$ and 
   $\sum_{i=1}^p L_{ij}\bullet \hat{X}_i = 0$ for all $j= 1, \ldots, m$.
\end{enumerate}
In particular, if \ref{P2} holds and 
$\sum_{i=1}^p L_{i0}\bullet \hat{X}_i >0$, 
then \eqref{lmi_def} is infeasible. 
Similarly, for \eqref{lmi_dual_def}, 
exactly one of the following two statements is true:
\begin{enumerate}
\item\label{P3} \eqref{lmi_dual_def} is strictly feasible. 
\item\label{P4} There exists $\hat{y}\in\mathbb{R}^{m}\setminus\{0\}$
   such that 
   $\sum_{j=1}^m L_{ij}\hat{y}_j\in\mathbb{S}^{n_i}_+$ for all $i=1, \ldots, p$ and $d^T\hat{y} \le 0$.
\end{enumerate}
In particular, if \ref{P4} holds and $d^T\hat{y} < 0$, 
then \eqref{lmi_dual_def} is infeasible. 
\end{theorem}
\section{Proofs on the strong duality}\label{app:proofs}
\subsection{Proof on the zero duality gap in Theorem \ref{thm:duality}}\label{subapp:proof1}

It is clear that the optimal value $\gamma^*$ 
of \eqref{LMI} is nonnegative. 
Therefore from Theorem \ref{sdt} all the assertions of Theorem
\ref{thm:duality}
can be verified by proving that (I) $\Leftrightarrow$ (II). 
We first prove (I) $\Rightarrow$ (II).  
Since $(A,b_2)$ is stabilizable, 
there exists $K\in\mathbb{R}^{1\times n}$ and $X_0-I_n\in \bbS_{++}^n$ such that
$-\He((A+b_2K)X_0)\in \bbS_{++}^n$.  
Similarly, 
since $(A,c_2)$ is detectable, 
there exists $L\in\mathbb{R}^{n\times 1}$ and $Y_0-I_n\in \bbS_{++}^n$ such that
$-\He(Y_0(A+Lc_2^T))\in \bbS_{++}^n$.  
It follows that for sufficiently large $\gamma$ we have
\begin{align*}
-\begin{pmatrix}
\He((A+b_2K)X_0) & X_0(c_1^T+d_{12}K)^T& b_1 \\
(c_1^T+d_{12}K)X_0&-\gamma & d_{11}\\
b_1^T & d_{11} & -\gamma 
\end{pmatrix}&\in\mathbb{S}_{++}^{n+2}, \\
-\begin{pmatrix}
\He(Y_0(A+Lc_2^T) & Y_0(b_1+Ld_{21}) & c_1 \\
(b_1+Ld_{21})^TY_0&-\gamma & d_{11}\\
c_1^T & d_{11} & -\gamma 
\end{pmatrix}&\in\mathbb{S}_{++}^{n+2},\\
\begin{pmatrix}
X_0 & -I_n\\
-I_n & Y_0
\end{pmatrix}&\in\mathbb{S}^{2n}_{++}. 
\end{align*}
These can be restated equivalently as
\begin{align*}
-\left(
\begin{pmatrix}
\He(AX_0) & X_0c_1& b_1 \\
c_1^TX_0&-\gamma & d_{11}\\
b_1^T & d_{11} & -\gamma 
\end{pmatrix}+
\He\left(
\begin{pmatrix}
b_2\\
d_{12}\\
0
\end{pmatrix}
\begin{pmatrix}
KX_0 & 0 & 0
\end{pmatrix}
\right)
\right) &\in\mathbb{S}_{++}^{n+2}, \\
-\left(
\begin{pmatrix}
\He(Y_0A) & Y_0b_1 & c_1 \\
b_1^TY_0&-\gamma & d_{11}\\
c_1^T & d_{11} & -\gamma 
\end{pmatrix}+
\He\left(
\begin{pmatrix}
c_2\\
d_{21}\\
0
\end{pmatrix}
\begin{pmatrix}
L^TY_0 & 0 & 0
\end{pmatrix}
\right) 
\right) &\in\mathbb{S}_{++}^{n+2}, \\
\begin{pmatrix}
X_0 & -I_n\\
-I_n & Y_0
\end{pmatrix}&\in\mathbb{S}_{++}^{2n}. 
\end{align*}
The above matrix inequalities clearly show that 
 LMI problem \eqref{LMI} is strictly feasible and hence (II) holds.  

To prove (I)$\Leftarrow$(II), suppose (II) holds.  
Then, from Elimination Lemma \cite{Gahinet94,Iwasaki94}, 
there exist $X,Y,\gamma$ and
$F_1,F_2,F_3$ and $G_1,G_2,G_3$ of appropriate size such that
\begin{align*}
-\left(
\begin{pmatrix}
\He(AX) & Xc_1& b_1 \\
c_1^TX&-\gamma & d_{11}\\
b_1^T & d_{11} & -\gamma 
\end{pmatrix}+
\He\left(
\begin{pmatrix}
b_2\\
d_{12}\\
0
\end{pmatrix}
\begin{pmatrix}
F_1 & F_2 & F_3
\end{pmatrix}
\right)
\right) &\in\mathbb{S}_{++}^{n+2}, \\
-\left(
\begin{pmatrix}
\He(YA) & Yb_1 & c_1 \\
b_1^TY&-\gamma & d_{11}\\
c_1^T & d_{11} & -\gamma 
\end{pmatrix}+
\He\left(
\begin{pmatrix}
c_2\\
d_{21}\\
0
\end{pmatrix}
\begin{pmatrix}
G_1 & G_2 & G_3
\end{pmatrix}
\right) 
\right) &\in\mathbb{S}_{++}^{n+2}, \\
\begin{pmatrix}
X & -I_n\\
-I_n & Y
\end{pmatrix}&\in\mathbb{S}_{++}^{2n}. 
\end{align*}
This in particular implies that
$X\in\bbS_{++}^n$ and $-\He((A+b_2K)X)\in \bbS_{++}^n$ hold with
$K=F_1X^{-1}$ and hence $(A,b_2)$ is stabilizable.  
Similarly, we have
$Y\in\bbS_{++}^n$ and $-\He(Y(A+Lc_2^T))\in \bbS_{++}^n$ hold with
$L=Y^{-1}G_1^T$ and hence $(A,c_2)$ is detectable. 
It follows that (I) holds, and this completes the proof.

\subsection{Proof of Lemma \ref{lem:exist}}\label{subapp:proof2}
We will prove that the dual of \eqref{LMI2} is strictly feasible. 
Theorem \ref{thm:duality} and the weak duality on LMI problems introduced in 
\ref{app:sdp} imply that the optimal value of the dual is bounded above. 
Hence, if the dual is strictly feasible, 
then it follows from Theorem \ref{sdt} that \eqref{LMI2} 
has an optimal solution. 
To prove this we use \ref{P3} and \ref{P4} of 
the second part of Theorem \ref{thm:fr}, i.e., 
we prove that there exists no solution 
$(\gamma, X, Y)\in\mathbb{R}\times\mathbb{S}^n_+\times\mathbb{S}^n_+$ such that 
\begin{align}\label{redcert}
&\left\{
\begin{array}{l}
-\begin{pmatrix}
\He(\Lambda_+^TX) - \gamma f_+f_+^T & 0\\
0 & -\gamma 
\end{pmatrix}\in\mathbb{S}^{n+1}_+, \\
-\begin{pmatrix}
\He(\Omega_+^TY) - \gamma g_+g_+^T & 0\\
0 & -\gamma 
\end{pmatrix}\in\mathbb{S}^{n+1}_+, \\
\gamma \le 0, (\gamma, X, Y)\neq (0, O_{n\times n}, O_{n\times n})
\end{array}
\right.
\end{align}
It is clear that $\gamma = 0$ is necessary for \eqref{redcert} being
valid, 
and by substituting it, we obtain
\begin{align}\label{redcert2}
-\He(\Lambda_+^TX)&\in\mathbb{S}^n_+, -\He(\Omega_+^TY)\in\mathbb{S}^{n}_+, (X, Y)\neq (O_{n\times n}, O_{n\times n})
\end{align}
As both $-\Lambda_+$ and $-\Omega_+$ are Hurwitz stable, %diagonal with positive elements, 
\eqref{redcert2} has no solutions, 
and thus \eqref{redcert} has no solutions.  
It follows that the dual of \eqref{LMI2} is strictly feasible, and hence
\eqref{LMI2} has an optimal solution. 

\subsection{Proof on the zero duality gap in Lemma \ref{lem:dualdual2}}\label{subapp:proof3}
We use Theorem \ref{thm:fr} for this proof. 
The condition \ref{P2} in the first part of Theorem \ref{thm:fr} can be described by
\begin{align}\label{conditionCase2}
&\left\{
\begin{array}{l}
\begin{pmatrix}
Z_{11} & Z_{21}^T\\
Z_{21} &Z_{22}
\end{pmatrix}\in\mathbb{S}^{k_1+1}_+, \begin{pmatrix}
V_{11} & V_{21}^T\\
V_{21} & V_{22}
\end{pmatrix}\in\mathbb{S}^{k_2+1}_+, \\
\begin{pmatrix}
W_{11} & W_{21}^T\\
W_{21} & W_{22}
\end{pmatrix}\in\mathbb{S}^{k_1+k_2}_+, W_{11} = \He(\Lambda_+Z_{11}), W_{22}=\He(\Omega_+V_{11}), \\
f_+^TZ_{11}f_+ +Z_{22} + g_+^TV_{11}g_+ + V_{22} = 0, \\
2(J_+\bullet W_{21} + h_{1+}^T\bullet Z_{21} + h_{2+}\bullet V_{21})\ge 0
\end{array}
\right. 
\end{align}
Any solution of \eqref{conditionCase2} satisfies $Z_{22} = V_{22}=0$. Substituting them, we reformulate \eqref{conditionCase2} into 
\begin{align}\label{conditionCase2_2}
&\left\{
\begin{array}{l}
Z_{11} \in\mathbb{S}^{k_1}_+, V_{11}\in\mathbb{S}^{k_2}_+, \begin{pmatrix}
W_{11} & W_{21}^T\\
W_{21} & W_{22}
\end{pmatrix}\in\mathbb{S}^{k_1+k_2}_+, J_+\bullet W_{21} \ge 0, \\
W_{11} = \He(\Lambda_+Z_{11}), W_{22}=\He(\Omega_+V_{11}), f_+^TZ_{11}f_+ = g_+^TV_{11}g_+ = 0 
\end{array}
\right. 
\end{align}

The next lemma is useful in analyzing \eqref{conditionCase2_2}: 
\begin{lemma}\label{lem:cntobs}
\begin{enumerate}
\item\label{D0s} If $(A, b_2)$ is stabilizable, 
   then $(\Lambda_+, f_+)$ is controllable. 
\item\label{D0d} If $(A, c_2^T)$ is detectable, 
   then $(\Omega_+, g_+)$ is controllable. 
\item\label{D1} If $(\Lambda_+, f_+)$ is controllable, 
   then there does not exist any 
   $Z_{11}\in\mathbb{S}^{k_1}_+\setminus\{O_{k_1}\}$ such that 
   $f_+^TZ_{11}f_+ = 0$ and $\He(\Lambda_+Z_{11})\in\mathbb{S}^{k_1}_+$. 
\item\label{D2} If $(\Omega_+, g_+)$ is controllable, 
   then there does not exist any 
   $V_{11}\in\mathbb{S}^{k_2}_+\setminus\{O_{k_2}\}$ such that 
   $g_+^TV_{11}g_+ = 0$ and $\He(\Omega_+V_{11})\in\mathbb{S}^{k_2}_+$. 
\end{enumerate}
\end{lemma}
\begin{proof}
We prove \ref{D0s} and \ref{D1} only because 
we can prove \ref{D0d} and \ref{D2} by a similar manner. 

To prove \ref{D0s} by contradiction, suppose that there exists 
$\lambda>0$ such 
that $v^T\Lambda_+ = \lambda v^T$ and $v^Tf_+=0$. 
Then, it follows from \eqref{Partition.vec} that $(v^TS_+^T)A = \lambda (v^TS_+^T)$ and 
$(v^TS_+^T)b_2=0$. This implies that $(A, b_2)$ is not stabilizable, 
and thus we obtain a contradiction.  

Again to prove \ref{D1} by contradiction, suppose that there exists a nonzero 
$Z_{11}\in\mathbb{S}^n_+$ such that $f_+^TZ_{11}f_+ = 0$ and 
$\He(\Lambda_+Z_{11})\in\mathbb{S}^n_+$. 
Then, since the latter condition can be seen as the 
Lyapunov equation $\He(\Lambda_+Z_{11})=W$ by introducing $W\in\mathbb{S}^n_+$
and since $-\Lambda_+$ is Hurwitz stable, we can solve this equation explicitly as 
\[
Z_{11} = \int_{0}^\infty\exp(-\Lambda_+ t)W\exp(-\Lambda_+^T t)\,dt. 
\]
It follows from $f_+^TZ_{11}f_+ = 0$ that 
$W\exp(-\Lambda^T_+ t)f_+=0$ for all $t\ge 0$. 
If $W$ is the zero matrix, then $Z_{11}$ is also zero, 
which contradicts the assumption that $Z_{11}$ is nonzero. 
Hence $W$ is nonzero. This implies that 
there exists $p\in\mathbb{C}^{k_1}\setminus\{0\}$ such that 
$p^T\exp(-\Lambda_+^T t)f_+=0$ for all $t\ge 0$. 
This contradicts the controllability of $(\Lambda_+, f_+)$.  
\end{proof}

It follows from Lemma \ref{lem:cntobs} and Assumption \ref{A1} that \eqref{conditionCase2} has no nonzero solutions, and thus we see from Theorem \ref{thm:fr} that \eqref{LMI4case2} is strictly feasible. Therefore the zero duality gap between \eqref{LMI4case2} and \eqref{dual3} holds from Theorem \ref{sdt}. This implies that the existence of an optimal solution of \eqref{dual3}. 

The existence of an optimal solution in \eqref{LMI4case2} can be proved similarly to Lemma \ref{lem:exist}.

\section{Proofs of technical lemmas}\label{app:lemproofs}
%\subsection{Proof of Lemma \ref{lemma:genvec34}}\label{subapp:genvec34}
%We prove only the case where $(A, c^T)$ is observable. The other case can be proved similarly.   
%Assume that $\lambda\in\mathbb{C}$ is an invariant zero with the multiplicity $m>1$ of \eqref{fundSISO} and that the geometric multiplicity $m_{\mathrm g}=1$. 
%
%If $(A, c^T)$ is observable, then there exist $q_k\in\mathbb{C}^n\setminus\{0\}$ and $\hat{q}_k\in\mathbb{C}$  \ $(k=1, \ldots, m)$ such that 
%\begin{align}\label{1steq}
%\begin{pmatrix}
%A-\lambda I_n & b\\
%c^T & d
%\end{pmatrix}\begin{pmatrix}
%q_1\\
%\hat{q}_1
%\end{pmatrix} &= \begin{pmatrix}
%0\\
%0
%\end{pmatrix},  \\
%\label{2ndeq}
%\begin{pmatrix}
%A-\lambda I_n & b\\
%c^T & d
%\end{pmatrix}\begin{pmatrix}
%q_k\\
%\hat{q}_k
%\end{pmatrix} &= \begin{pmatrix}
%q_{k-1}\\
%0
%\end{pmatrix} \ (k=2, \ldots, m)
%\end{align}
%It follows from 3 of Lemma \ref{lem:zero} that $\hat{q}_1 \neq 0$. If $\hat{q}_2$

\subsection{Proof of Lemma \ref{lem:infinitevec}}\label{subapp:lemma2}
Since the relative degree of the dual system \eqref{dualfundSISO} is equal to $r$, we prove only the first part of the Lemma \ref{lem:infinitevec}. We note that the submatrices $Q_{11}$ and $\left(
\begin{smallmatrix}
Q_{11}\\
Q_{21}
\end{smallmatrix}
\right)$ are of full column rank. 

It follows from \ref{b3} in Remark \ref{remark1} that we can set $\left(\begin{smallmatrix}
Q_{12} \\ Q_{22}
\end{smallmatrix}\right) = \left(
\begin{smallmatrix}
0\\
1
\end{smallmatrix}
\right)$ when $r=0$. We consider the case $r>0$. Clearly, $\left(\begin{smallmatrix}
Q_{12} \\ Q_{22}
\end{smallmatrix}\right)$ in \eqref{infiniteRvec} satisfies \eqref{eq:nilpotent3}. Thus it is sufficient to prove that $Q$ is non-singular. For this, we consider $\alpha\in\mathbb{C}^{n-r}$ and $\beta = (\beta_1, \ldots, \beta_{r+1})^T\in\mathbb{C}^{r+1}$ that satisfies
\[
\begin{pmatrix}
Q_{11} & Q_{12}\\
Q_{21} & Q_{22}
\end{pmatrix}\begin{pmatrix}
\alpha\\
\beta
\end{pmatrix} = 0. 
\]
We note $d=0$ and thus $c^TQ_{11} = 0$ because of \eqref{eq:finitezero}. Using this, we obtain $c^TQ_{12}\beta = 0$. Furthermore, from the definition of the relative degree, we obtain $\beta_{r+1} = 0$. Next, we see $c^TAQ_{11} = 0$ from \eqref{cqf2}. Thus we obtain $\beta_{r} = 0$ from $c^TAQ_{12}\beta = 0$ and $\beta_{r+1}=0$. Applying this procedure repeatedly, we obtain
\[
\begin{pmatrix}
Q_{11} & 0\\
Q_{21} & 1
\end{pmatrix}\begin{pmatrix}
\alpha\\
\beta_1
\end{pmatrix} = 0. 
\]
Since $Q_{11}$ is of full column rank, $\alpha = 0$ and $\beta_1=0$. Therefore, $Q$ is non-singular because $Q$ is square and all the columns are linearly independent. 

\subsection{Proof of Lemma \ref{lem:zerocase1}}\label{subapp:lemma5}
We first prove the ``if'' part. 
Suppose $h_{1+} = 0$, $h_{2+} = 0$ and $J_+ = O_{n\times n}$.  
Then LMI problem \eqref{LMI2} can be reformulated as 
\begin{align*}
&\left\{
\begin{array}{cl}
\displaystyle\inf%_{\gamma, X, Y} 
& \gamma\\
\mbox{subject to} & - \begin{pmatrix}
\He(\Lambda_+^T X) -\gamma f_+f_+^T & 0^T\\
0 & -\gamma
\end{pmatrix} \in\mathbb{S}_+^{n+1}, \\
& - \begin{pmatrix}
\He(\Omega_+^T Y) -\gamma g_+g_+^T & 0^T\\
0 & -\gamma
\end{pmatrix} \in\mathbb{S}_+^{n+1}, X, Y\in\mathbb{S}^n_+ 
\end{array}
\right.
\end{align*}
From this form it is very clear that
$(\gamma, X, Y) = (0, O_{n\times n}, O_{n\times n})$ is an optimal
solution achieving $\gamma^*=0$. 

We next prove the ``only if'' part. 
If $\gamma^*=0$, 
it follows from Lemma \ref{lem:exist} that \eqref{LMI2} 
has an optimal solution $(0, X, Y)$. 
Then $(X, Y)$ satisfies
\[
 - \begin{pmatrix}
\He(\Lambda_+^T X) & h_{1+}\\
h_{1+}^T & 0
\end{pmatrix} \in\mathbb{S}_+^{n+1}, - \begin{pmatrix}
\He(\Omega_+^T Y) & h_{2+}\\
h_{2+}^T & 0
\end{pmatrix} \in\mathbb{S}_+^{n+1}, \begin{pmatrix}
X & -J_{+}^T\\
-J_{+}& Y
\end{pmatrix}\in\mathbb{S}^{2n}_+. 
\]
From these matrix inequalities 
we readily obtain $h_{1+}=0$, $h_{2+} = 0$. In addition, we obtain the following equations.
\[
\left\{
\begin{array}{lcll}
\He((-\Lambda_+)^TX) &=& \tilde{X}, \tilde{X}\in\mathbb{S}^n_+, X\in\mathbb{S}^n_+,\\
\He((-\Omega_+)^TY) &=& \tilde{Y}, \tilde{Y}\in\mathbb{S}^n_+, Y\in\mathbb{S}^n_+,
\end{array}
\right. 
\]
These equations can be seen as the Lyapunov equations, and thus we have 
\[
X = -\int_{0}^{+\infty}\exp(-\Lambda_+^Tt)\tilde{X}\exp(-\Lambda_+t)\, dt \mbox{ and } Y = -\int_{0}^{+\infty}\exp(-\Omega_+^Tt)\tilde{Y}\exp(-\Omega_+t)\, dt. 
\]
Therefore $X=O_{n\times n}$ and $Y=O_{n\times n}$ because $\tilde{X}, \tilde{Y}\in\mathbb{S}^n_+$ and $X, Y\in\mathbb{S}^n_+$. Then $J_+$ must be the zero matrix. This completes the proof.  

\subsection{Proof of Lemma \ref{lem:dual}}\label{subapp:lemma6}
We define the Lagrange function $L$ for \eqref{LMI2} as follows: 
\begin{align*}
L(\gamma, \hat{X}, \hat{Y}, Z, V, W) &= \gamma + Z\bullet \begin{pmatrix}
\He(\Lambda^T\hat{X})-\gamma ff^T &h_1 \\
h_1^T & -\gamma
\end{pmatrix}- W\bullet \begin{pmatrix}
\hat{X} & -J^T\\
-J & \hat{Y}
\end{pmatrix} \\
&\quad + V\bullet \begin{pmatrix}
\He(\Omega^T\hat{Y})-\gamma gg^T &h_2 \\
h_2^T & -\gamma
\end{pmatrix}. 
\end{align*}
%for $\gamma\in\mathbb{R}, \hat{X}, \hat{Y}\in\mathbb{S}^n$ and 
%\[
%Z = \begin{pmatrix}
%Z_{11} & Z_{21}^T\\
%Z_{21} & Z_{22}
%\end{pmatrix}\in\mathbb{S}^{n+1}_+, V = \begin{pmatrix}
%V_{11} & V_{21}^T\\
%V_{21} & V_{22}
%\end{pmatrix}\in\mathbb{S}^{n+1}_+ \mbox{ and } W = \begin{pmatrix}
%W_{11} & W_{21}^T\\
%W_{21} & W_{22}
%\end{pmatrix}\in\mathbb{S}^{2n}_+. 
%\]
The Lagrange function $L$ can be reformulated as follows: 
\begin{align*}
L(\gamma, \hat{X}, \hat{Y}, Z, V, W) &= 2(Z_{21}\bullet h_1^T + V_{21}\bullet h_2^T + W_{21}\bullet J)\\
&\quad + \hat{X}\bullet (\He(\Lambda Z_{11}) - W_{11})+ \hat{Y}\bullet (\He(\Omega V_{11}) - W_{22})\\
&\quad + \gamma(1-f^TZ_{11}f-Z_{22}-g^TV_{11}g-V_{22}). 
\end{align*}
Then for any $(Z, V, W)\in\mathbb{S}^{n+1}_+\times\mathbb{S}^{n+1}_+\times\mathbb{S}^{2n}_+$, we consider the following Lagrange relaxation problem:
\begin{align*}
%\nonumber
&\inf%_{\gamma, \hat{X}, \hat{Y}}
\left\{
L(\gamma, \hat{X}, \hat{Y}, Z, V, W) : \gamma\in\mathbb{R}, \hat{X}, \hat{Y}\in\mathbb{S}^n
\right\} \\
%\label{LagRelax}
=& \left\{
\begin{array}{cl}
2(Z_{21}\bullet h_1^T + V_{21}\bullet h_2^T + W_{21}\bullet J) & \mbox{ if } f^TZ_{11}f+Z_{22}+g^TV_{11}g+V_{22}=1, \\
& \quad W_{11} = \He(\Lambda Z_{11}), W_{22} = \He(\Omega V_{11}) , \\
-\infty & \mbox{ o.w.}
\end{array}
\right. 
\end{align*}
Hence we see that \eqref{dual2} is the dual of \eqref{LMI2}. On the other hand, 
it follows from Theorem \ref{thm:duality} that 
the duality gap between \eqref{LMI2} and \eqref{dual2} is zero, and \eqref{dual2} has an optimal solution. This completes the proof.  

\subsection{Proof of Lemma \ref{reduction}}\label{subapp:lemma9}
First of all, we prove \eqref{Z11} and the first equation in \eqref{W11}. Any feasible solution $(W_{ij}, Z_{ij}, V_{ij})$ satisfies
\begin{align*}
W_{11} &= \He\left(\begin{pmatrix}
\Lambda_0 & \\
 & \Lambda
\end{pmatrix}
\begin{pmatrix}
Z_{11} & Z_{21}^T\\
Z_{21} &Z_{22}
\end{pmatrix}
\right)=\begin{pmatrix}
\He(\Lambda_0 Z_{11}) &\Lambda_0Z_{21}^T+Z_{21}^T\Lambda^T \\
\Lambda Z_{21} + Z_{21}\Lambda_0^T & \He(\Lambda Z_{22})
\end{pmatrix}. 
\end{align*} 
For simplicity, we assume $m_1=2$. Other cases can be proved by a similar manner. Then $\Lambda_0$ and $Z_{11}$ can be written by 
\begin{align*}
\Lambda_0&=\begin{pmatrix}
F(\lambda_1) & \\
& F(\lambda_2) 
\end{pmatrix}, Z_{11} = \begin{pmatrix}
Z_{11}^1 & (Z_{11}^2)^T \\
Z_{11}^2 & Z_{11}^3 
\end{pmatrix} \mbox{ and } Z_{11}^1 = \begin{pmatrix}
a_1 & a_2\\
a_2 & a_3
\end{pmatrix}. 
%Z_{11}^2 = \begin{pmatrix}
%a_4 & a_6\\
%a_5 & a_7
%\end{pmatrix}. 
\end{align*}
Hence we have 
\begin{align}
\label{W1}
\He(\Lambda_0 Z_{11})&=\begin{pmatrix}
\He(F(\lambda_1)Z_{11}^1) &* \\
F(\lambda_2) Z_{11}^2 + Z_{11}^2F(\lambda_1)^T & \He(F(\lambda_2) Z_{11}^3) 
\end{pmatrix}, \\
\label{W2}
\He(F(\lambda_1)Z_{11}^1) &=\begin{pmatrix}
2\Im(\lambda_1) a_2 & \Im(\lambda_1)(a_3-a_1)\\
 \Im(\lambda_1)(a_3-a_1) & -2\Im(\lambda_1) a_2
\end{pmatrix}. %, \\
%\label{W3}
%F(\lambda_2) Z_{11}^2 + Z_{11}^2F(\lambda_1)^T &= \begin{pmatrix}
%a_5(\Im(\lambda_1)+\Im(\lambda_2)) & \Im(\lambda_2)a_7 - \Im(\lambda_1)a_4\\
% \Im(\lambda_1)a_7 - \Im(\lambda_2)a_4 & -a_6(\Im(\lambda_1)+\Im(\lambda_2))
%\end{pmatrix}. 
\end{align}
Since \eqref{W1} is positive semidefinite, the diagonal elements in
 \eqref{W2} must be nonnegative, and thus $a_2=0$ and $a_1=a_3$. This
 implies that \eqref{W2} is the zero matrix. Therefore we obtain
 $F(\lambda_2) Z_{11}^2 + Z_{11}^2F(\lambda_1)^T = O_{2\times 2}$ from the positive semidefiniteness of
 \eqref{W1}. The first equation can be regarded as the Sylvester
 equation on $Z_{11}^2$. As we have assumed that invariant zeros on the imaginary axis are distinct from each other, we have $\lambda_1\neq\lambda_2$, and thus $F(\lambda_1)$ has no common eigenvalues of
 $-F(\lambda_2)^T$. Thus $Z_{11}^2$ the zero matrix, see \cite[Theorem 2.4.4.1]{Horn12} for the existence and uniqueness of the solution in the Sylvester equation.
 %see \ref{app:sylvester} for details. 
 Similarly, we obtain $Z_{11}^3$ is diagonal with nonnegative elements.

Also since $\He(\Lambda_0 Z_{11}) = O_{4\times 4}$, the matrix $\Lambda Z_{21} + Z_{21}\Lambda_0^T$ is zero. This equation is also the Sylvester equation. Since $\Lambda$ and $-\Lambda_0^T$ have no common eigenvalues, the unique solution of the Sylvester equation is $Z_{21} = O_{(n-2m_1)\times 2m_1}$.  
%This is equivalent to the following equation 
%\[
%(I_{2m_1+1}\otimes \Lambda + \Lambda_0\otimes I_{n-2m_1-1})\mbox{vec}(Z_{21}) = 0, 
%\]
%where $\mbox{vec}(Z_{21})$ is the vecterization of $Z_{21}$. Since the coefficient matrix of this linear system is nonsingular, $\mbox{vec}(Z_{21})$ is zero, and thus $Z_{21} = O_{(n-2m_1-1)\times (2m_1+1)}$. 
%For $i=1, \ldots, (n-2m_1-1), j=1, \ldots, 2m_1+1$, we have 
%\[
%\left\{
%\begin{array}{cl}
%(\Lambda)_{ii}z_{ij} - z_{i j+1}\Im(\lambda_j) = 0, (\Lambda)_{ii}z_{i j+1} + z_{i j}\Im(\lambda_j) = 0 & (j \mbox{ is odd and not } 2m_1+1), \\
%(\Lambda)_{ii}z_{ij} - z_{i j-1}\Im(\lambda_j) = 0, (\Lambda)_{ii}z_{i j-1} + z_{i j}\Im(\lambda_j) = 0 & (j \mbox{ is even and not } 2m_1+1), \\
%(\Lambda)_{ii}z_{ij} = 0 & (j = 2m_1+1), 
%\end{array}
%\right. 
%\]
%where $Z_{21} = (z_{ij})_{1\le i \le (n-2m_1-1), 1\le j\le 2m_1+1}$. We obtain $z_{i j}
% = 0$ for both cases. 
%We also have $Z_{11}^4F(\lambda_2) = Z_{11}^5F(\lambda_2) = 0$, and thus $Z_{11}^4 = Z_{11}^5 = 0$.  
 Hence we obtain \eqref{Z11}. The first equation of \eqref{W11} is proved from the positive semidefiniteness of $W$ and $\He(\Lambda_0 Z_{11}) = O$. 
By following similar lines, the rest equations \eqref{V11} and those in 
\eqref{W11} can be proved.  

\subsection{Proof of Lemma \ref{lem:dualdual}}\label{subapp:lemma10}
We prove by s similar manner to Lemma \ref{lem:dual} that the Lagrange dual of \eqref{LMI_img3} is \eqref{dual_img3}. We define the Lagrange function $L$ for \eqref{LMI_img3}: 
{\small 
\begin{align*}
&L(\gamma, \hat{X}, \hat{Y}, U^X_{ij}, U^Y_{ij}, Z, V, \hat{W}, z, v) \\
=& \gamma + Z\bullet \begin{pmatrix}
U_{11}^X -\gamma f_0f_0^T &(U_{21}^X)^T-\gamma f_0f^T &h_{10} \\
U_{21}^X-\gamma ff^T_0&\He(\Lambda^T \hat{X})-\gamma ff^T &h_{1} \\
h_{10}^T & h_{1}^T& -\gamma
\end{pmatrix} \\
&\quad + V\bullet \begin{pmatrix}
U_{11}^Y -\gamma g_0g_0^T &(U_{21}^Y)^T-\gamma g_0g^T &h_{20} \\
U_{21}^Y-\gamma gg^T_0&\He(\Omega^T \hat{Y})-\gamma gg^T &h_{2} \\
h_{20}^T & h_{2}^T& -\gamma
\end{pmatrix} -\hat{W}\bullet \begin{pmatrix}
\hat{X} & -J^T\\
-J & \hat{Y}
\end{pmatrix} \\ 
&\quad + \sum_{j=1}^{m_1}z_j((U_{11}^X)_{2j-1, 2j-1} + (U_{11}^X)_{2j, 2j}) + \sum_{j=1}^{m_2}v_j((U_{11}^Y)_{2j-1, 2j-1} + (U_{11}^Y)_{2j, 2j})\\
=&2\left(h_{10}^T\bullet Z_{31} + h_{1}^T\bullet Z_{32} +
 h_{20}^T\bullet V_{31} + h_{2}^T\bullet V_{32} + J\bullet \hat{W}_{21}
\right) \\
&\quad + \hat{X}\bullet (\He(\Lambda Z_{22})-\hat{W}_{11}) + \hat{Y} \bullet (\He(\Omega V_{22})-\hat{W}_{22})\\
&\quad +\gamma \left(1-\begin{pmatrix}
f_0\\
f
\end{pmatrix}^T\begin{pmatrix}
Z_{11} & *\\
Z_{21} & Z_{22}
\end{pmatrix}\begin{pmatrix}
f_0\\
f
\end{pmatrix} -Z_{33}- \begin{pmatrix}
g_0\\
g
\end{pmatrix}^T\begin{pmatrix}
V_{11} & *\\
V_{21} & V_{22}
\end{pmatrix}\begin{pmatrix}
g_0\\
g
\end{pmatrix}-V_{33}\right)\\
&\quad + \sum_{i \neq j}(U_{11}^X)_{ij}(Z_{11})_{ij} + \sum_{i \neq j}(U_{11}^Y)_{ij}(V_{11})_{ij} + \sum_{j=1}^{m_1}(U_{11}^X)_{2j-1, 2j-1}((Z_{11})_{2j-1, 2j-1}+z_{j}) \\
&\quad + \sum_{j=1}^{m_1}(U_{11}^X)_{2j, 2j}((Z_{11})_{2j, 2j}+z_{j})+ \sum_{j=1}^{m_2}(U_{11}^Y)_{2j-1, 2j-1}((V_{11})_{2j-1, 2j-1}+v_{j})\\
&\quad+ \sum_{j=1}^{m_2}(U_{11}^Y)_{2j, 2j}((V_{11})_{2j, 2j}+v_{j}) + 2(U_{21}^X\bullet Z_{21} + U_{21}^Y\bullet V_{21}). 
\end{align*}
}
%for $\gamma\in\mathbb{R}$, $\hat{X}_{22}\in\mathbb{S}^{k_1}$, $\hat{Y}_{22}\in\mathbb{S}^{k_2}$, $U_{21}^X\in\mathbb{R}^{k_1\times 2m_1}$, $U_{21}^Y\in\mathbb{R}^{k_2\times 2m_2}$, $z\in\mathbb{R}^{m_1}$, $v\in\mathbb{R}^{m_2}$, 
%\begin{align*}
% W&=\begin{pmatrix}
%\hat{W}_{11} & \hat{W}_{21}^T\\
%\hat{W}_{21} & \hat{W}_{22}
%\end{pmatrix}\in\mathbb{S}^{k_1+k_2}_+, Z = \begin{pmatrix}
%Z_{11} & Z_{21}^T & Z_{31}^T\\
%Z_{21} & \hat{Z}_{22} & \hat{Z}_{32}^T\\
%Z_{31} & \hat{Z}_{32} & Z_{33}
%\end{pmatrix}\in\mathbb{S}^{n_1}_+ \mbox{ and }  V = \begin{pmatrix}
%V_{11} & V_{21}^T & V_{31}^T\\
%V_{21} & \hat{V}_{22} & \hat{V}_{32}^T\\
%V_{31} & \hat{V}_{32} & V_{33}
%\end{pmatrix}. 
%\end{align*}
Then for any $(Z, V, \hat{W}, z, v)\in\mathbb{S}^{n+1}_+\times\mathbb{S}^{n+1}_+\times\mathbb{S}^{n_0}_+\times\mathbb{R}^{m_1}\times\mathbb{R}^{m_2}$, the Lagrange relaxation problem can be formulated as 
\begin{align*}
&\inf%_{\begin{subarray}{c}
%\gamma, \hat{X}_{22}, \hat{Y}_{22}, U^X_{ij}, U^Y_{ij}
%\end{subarray}}
\left\{
L(\gamma, \hat{X}, \hat{Y}, U^X_{ij}, U^Y_{ij}, Z, V, \hat{W}, z, v) 
: \begin{array}{l}
\gamma\in\mathbb{R}, \hat{X}\in\mathbb{S}^{n-2m_1}, \hat{Y}\in\mathbb{S}^{n-2m_2}, \\
U^X_{11}\in\mathbb{S}^{2m_1}, U^Y_{11}\in\mathbb{S}^{2m_2}, \\
U^X_{21}\in\mathbb{R}^{(n-2m_1)\times 2m_1}, \\
U^Y_{21}\in\mathbb{R}^{(n-2m_2)\times 2m_2}
\end{array}
\right\}\\
=& \left\{
\begin{array}{cl}
2\left(h_{10}^T\bullet Z_{31} + h_{1}^T\bullet Z_{32} +
 h_{20}^T\bullet V_{31} \right. \\
 \left. + h_{2}^T\bullet V_{32} + J\bullet \hat{W}_{21}
\right)& \mbox{if } (Z, V, \hat{W}, z, v)\in\mathcal{F} \\
-\infty & \mbox{o.w.}. 
\end{array}
\right.
\end{align*}
%where $\mathcal{F}_d$ is a subset of $\mathbb{S}^{n+1}_+\times\mathbb{S}^{n+1}_+\times\mathbb{S}^{2n}_+\times\mathbb{R}^{2m_1+1}\times\mathbb{R}^{2m_2+1}$. Eliminating $\alpha$ and $\beta$ from $\mathcal{F}_d$, we obtain \eqref{dual_img4}. 
As it is easy to obtain the dual \eqref{dual_img3} from
 $\mathcal{F}$, we omit the detail. The zero duality gap and the existence of an optimal solution of \eqref{dual_img3} can be proved similarly to \ref{subapp:proof3}.  

\subsection{Proof of Lemma \ref{lem:dual_sing}}\label{subapp:lemma14}
We define the Lagrange function $L$ for \eqref{LMI_sing2}: 
{\small 
\begin{align*}
&L(\gamma, X, Y, Z, V, W) \\
=& \gamma + Z\bullet \begin{pmatrix}
\He(\Lambda^T S^TXS) -\gamma ff^T &*&* \\
\hat{P}_1^TXS +P_1^TXS\Lambda-\gamma p_1f^T&\He(\hat{P}_1^TXP_1)-\gamma p_1p_1^T&* \\
h_1^T & \tilde{h}_1^T&-\gamma
\end{pmatrix} \\
&\quad + V\bullet \begin{pmatrix}
\He(\Omega^T T^TYT) -\gamma gg^T &*&* \\
\hat{P}_2^TYT +P_2^TYT\Omega-\gamma p_2g^T&\He(\hat{N}_2^TYP_2)-\gamma p_2p_2^T&* \\
h_2^T & \tilde{h}_2^T&-\gamma
\end{pmatrix} \\
&\quad - W \bullet \begin{pmatrix}
X & -I_n\\
-I_n & Y
\end{pmatrix}\\
=& \begin{pmatrix}
& I_n\\
I_n & 
\end{pmatrix}\bullet W + \begin{pmatrix}
& &h_1\\
& &\tilde{h}_1\\
h_1^T & \tilde{h}_1^T&  
\end{pmatrix}\bullet Z+ \begin{pmatrix}
& &h_2\\
& &\tilde{h}_2\\
h_2^T & \tilde{h}_2^T&  
\end{pmatrix}\bullet V\\
&\quad + \gamma\left(
1 - \begin{pmatrix}
ff^T & fp_1^T & \\
p_1f^T & p_1p_1^T & \\
& & 1
\end{pmatrix}\bullet Z - \begin{pmatrix}
gg^T & gp_2^T & \\
p_2g^T & p_2p_2^T & \\
& & 1
\end{pmatrix}\bullet V
\right)\\
&\quad + X\bullet \left(
W_{11} - \He\left(
\begin{pmatrix}
S\Lambda & \hat{P}_1
\end{pmatrix}\begin{pmatrix}
Z_{11} & Z_{21}^T\\
Z_{21} & Z_{22}
\end{pmatrix}\begin{pmatrix}
S^T\\
P_1^T
\end{pmatrix}
\right)
\right) \\
&\quad + Y\bullet \left(
W_{22} - \He\left(
\begin{pmatrix}
T\Omega & \hat{P}_2
\end{pmatrix}\begin{pmatrix}
V_{11} & V_{21}^T\\
V_{21} & V_{22}
\end{pmatrix}\begin{pmatrix}
T^T\\
P_2^T
\end{pmatrix}
\right)
\right) 
\end{align*}
}
By applying a similar discussion in Lemma \ref{lem:dual}, we can obtain the Lagrange relaxation problem for $(Z, V, W)\in\mathbb{S}^{n+1}_+\times\mathbb{S}^{n+1}_+\times\mathbb{S}^{2n}_+$ and the Lagrange dual problem. We see that the Lagrange dual is equivalent to \eqref{dual_sing2}. The zero duality gap between \eqref{LMI_sing2} and \eqref{dual_sing2} follows from Theorem \ref{thm:duality}. In fact, the proof is independent in the computation of the perpendicular matrices in \eqref{LMI}. In addition, it follows from Assumption \ref{A1} and Theorem \ref{thm:duality} that \eqref{dual_sing2} has an optimal solution.

\subsection{Proof of Lemma \ref{zeroDual}}\label{subapp:lemma15}
We prove only the statement on $Z_{ij}$ because we can also prove the statement on $V_{ij}$ in a similar manner.  
Since we use equations \eqref{StAb}, \eqref{bAN1} and \eqref{bAhN1} in this proof, we rewrite here.
\begin{align}
\label{rStAb} S^TA^r b_2 &= 0 \ (r=0, \ldots, r_1-1) \tag{63}\\
\label{rbAN1} (P_1^TA^rb_2)_j &= 0 \ (r=0, \ldots, r_1-1, j=1, \ldots, r_1-r) \tag{65}\\
\label{rbAhN1} (\hat{P}_1^TA^rb_2)_j &= 0 \ (r=0, \ldots, r_1-2, j=1, \ldots, r_1-r-1). \tag{68}
\end{align}
We focus on the following constraint of \eqref{dual_sing2}.
\[
W_{11} = \He(S\Lambda(Z_{11}S^T+ Z_{21}^TP_1^T) + \hat{P}_1(Z_{21}S^T +Z_{22}P_1^T)). 
\]
In addition, we use the positive semidefiniteness of $Z_{22}$ and $W_{11}$, and the following well-known facts.
\begin{fact}\label{wfact}
If there exists nonzero $v\in\mathbb{R}^n$ such that $v^TW_{11}v = 0$, then $W_{11}v = 0$. 
\end{fact} 
 \begin{fact}\label{zfact}
Let $k \in \{1, \ldots, r_1\}$. If $(Z_{22})_{kk}=0$, then $(Z_{22})_{kj} = (Z_{22})_{jk}= 0$ for all $j=1, \ldots, r_1$. 
\end{fact} 
% \begin{align}
% \label{StAb}
% b_2^T(A^T)^{r}S&= 0 \mbox{ and } \\
% \label{bAN1}
% b_2^T(A^T)^{r}P_1 &= \bordermatrix{
% &(r_1 - r) & r\cr
% &0 & b_2^T(A^T)^{r_1-1}c_1, \ldots, b_2^T(A^T)^{r+r_1-2}c_1
% }
% \end{align}
% for $r = 0, 1, \ldots, r_1-1$ and 
% \begin{align}
% \label{bAhN1}
% b_2^T(A^T)^{r}\hat{P}_1 &= \bordermatrix{
% &(r_1 - r-1) & (r+1)\cr
% &0 & b_2^T(A^T)^{r_1-1}c_1, \ldots, b_2^T(A^T)^{r+r_1-1}c_1
% }
% \end{align}
% for $r = 0, 1, \ldots, r_1-2$, where $*$ stands for real numbers. In particular, we have $b_2^T(A^T)^{r+1}P_1 = b_2^T(A^T)^{r}\hat{P}_1$ for $r = 0, \ldots, r_1-2$. 

First of all, it follows from \eqref{StAb} and \eqref{bAN1} that we have $S^Tb_2 = 0$ and $P_1^Tb_2=0$, and thus $b_2^TW_{11}b_2 = 0$. $W_{11}b_2 = 0$ follows from Fact \ref{wfact}. 

If $r_1=1$, we have $P_1=0$ and $\hat{P}_1 = c$, and thus $W_{11} = \He((S\Lambda Z_{11} + c_1Z_{21}^T)S^T)$. Multiplying $b_2^T$ from the left side, we obtain $b_2^TW_{11} = (b_2^Tc_1)Z_{21}^TS^T = 0$. This implies $(Z_{21})_{r_1, j} = 0$ for all $j=1, \ldots, n-r_1$ because $S$ is of full column rank and $c_1^Tb_2 \neq 0$. This is the proof of the case $r_1=1$. 

Secondly, we consider the case $r_1>1$. We prove the statement on $Z_{22}$ by induction. For this, we first prove $(Z_{22})_{r_1 r} = (Z_{22})_{r r_1} = 0$ for all $r=1, \ldots, r_1$. \eqref{bAN1} and \eqref{bAhN1}, respectively are equal to 
\begin{align}
\label{eq6_1}
b_2^TA^TP_1 &= b_2^T\hat{P}_1 = \bordermatrix{
&(r_1 - 1) & 1\cr
&0 & b_2^T(A^T)^{r_1-1}c_1
}  = (b_2^T(A^T)^{r_1-1}c_1)e_{r_1}^T, 
\end{align}
where $e_{r_1}\in\mathbb{R}^{r_1}$ is the $r_1$th unit vector. It follows from $W_{11}b_2=0$, $S^Tb_2=0$, $P_1^Tb_2 = 0$ and \eqref{eq6_1} that we have 
$
W_{11}b_2 = (SZ_{21}^T+ P_1Z_{22})(b_2^T(A^T)^{r_1-1}c_1)e_{r_1} = 0
$. Multiplying $(Ab_2)^T$ into the left side of the above equation, we obtain 
 $(b_2^T(A^T)^{r_1-1}c_1)^2 e_{r_1}^TZ_{22}e_{r_1}=0$ from \eqref{eq6_1}. 
 Hence it follows from Fact \ref{zfact} that $(Z_{22})_{r_1 r} = (Z_{22})_{r r_1} = 0$ for all $r=1, \ldots, r_1$. Morevoer, $(Z_{21})_{r_1 r} = 0$ for all $r=1, \ldots, n-r_1$ and $(Z_{32})_{r_1} = 0$ because we have $b_2^T(A^T)^{r_1-1}c_1 \neq 0$ and the positive semidefiniteness of $Z$.  

Thirdly, we fix $1\le r\le r_1-2$ arbitrary, and assume that $Z_{21}$ and $Z_{22}$ form   
\[
Z_{21}^T = \bordermatrix{
&(r_1 - r) & r\cr
&\tilde{Z}_{21}^T & O_{(n-r_1) \times r}
} \mbox{ and } Z_{22} = \bordermatrix{
&(r_1 - r) & r\cr
(r_1 - r)&\tilde{Z}_{22} & O_{(r_1-r) \times r}\cr
r&O_{r\times(r_1-r) } & O_{r\times r}
}. 
\]
Then it follows from $S^TA^{r}b_2=0$ from \eqref{StAb} that  we have
\[
(A^{r}b_2)^T\He(S\Lambda(Z_{11}S^T + Z_{21}^TP_1^T)+\hat{P}_1Z_{21}S^T)(A^{r}b_2) = 0. 
\]
In addition, we see from \eqref{bAN1} that the first $(r_1-r)$ elements of the vector $P_1^TA^{r}b_2$ are zero. Hence, it follows from the structure of $Z_{22}$ that we have 
 %\[
%(A^{r}b_2)^TP_1 = \bordermatrix{
%&(r_1 -r) &r  \cr
%& 0 & *\cr
%} \mbox{ and } (A^{r}b_2)^T\hat{P}_1 = \bordermatrix{
%&(r_1 -r-1) &r+1 \cr
% & 0 & * \cr
%}, 
%\]
%where $*$ stands for real numbers. 
\begin{align}\label{eq6d}
(A^{r}b_2)^TW_{11}(A^{r}b_2) &= \He((A^{r}b_2)^T\hat{P}_1Z_{22} P_1^T(A^{r}b_2)) =0. 
%&= \He\left(\bordermatrix{
%&(r_1 -r-1) &(r+1) \cr
% & 0 & * \cr
%}\begin{pmatrix}
%\tilde{Z}_{22} & O_{(r_1-r) \times r}\\
%O_{r\times(r_1-r)} & O_{r\times r}
%\end{pmatrix} \bordermatrix{
%& \cr
%(r_1 -r)& 0 \cr
%r& *
%}
%\right)\\
\end{align}
%We remark that this equality also holds for $r=r_1-1$ and use this equality to prove $Z_{21}=O_{r_1\times (n-r_1)}$. 
$W_{11}(A^{r}b_2)=0$ follows from Fact \ref{wfact}. \eqref{rStAb}, \eqref{rbAN1} and this equation imply that 
\begin{align}\label{eq6e}
0 = W_{11}(A^rb_2) = (SZ_{21}^T + P_1Z_{22})\hat{P}_1^T(A^{r}b_2).
\end{align}
 In addition, we have from \eqref{bAN1} and \eqref{bAhN1}, 
\begin{align*}
\begin{pmatrix}
b_2^T(A^{r+1})^TP_1\\
%\hat{P}_1^T(A^{r}b_2)
b_2^T(A^{r})^T\hat{P}_1
\end{pmatrix}& = \bordermatrix{
&(r_1 -r-1)&1 &(r+1) \cr
&0& c_1^TA^{r_1-1}b_2&* \cr
&0& c_1^TA^{r_1-1}b_2&* \cr
}. %, \\
%\hat{P}_1^T(A^{r}b_2) &=\bordermatrix{
%&(r_1 -r-1) &1&(r+1) \cr
%&0& c_1^TA^{r_1-1}b_2&* \cr
%}^T
\end{align*}
Multiplying $(A^{r+1}b_2)^T$ from the left side of \eqref{eq6e}, we obtain %from \eqref{bAN1} and \eqref{bAhN1}, 
\begin{align*}
 0 &= b_2^T(A^T)^{r+1}P_1Z_{22}\hat{P}_1^T(A^{r}b_2) %\\
% &=\bordermatrix{
%&(r_1 -r-1)&1 &(r+1) \cr
%&0& c_1^TA^{r_1-1}b_2&* \cr
%} \begin{pmatrix}
%\tilde{Z}_{22} & O_{(r_1-r) \times r}\\
%O_{r\times(r_1-r)} & O_{r\times r}
%\end{pmatrix} \bordermatrix{
%&(r_1 -r-1) &1&(r+1) \cr
%&0& c_1^TA^{r_1-1}b_2&* \cr
%^T\\
%
=(Z_{22})_{(r_1-r), (r_1-r)} (c_1^TA^{r_1-1}b_2)^2. 
\end{align*}
Therefore from Fact \ref{zfact}, we obtain $(Z_{22})_{(r_1-r), j} = (Z_{22})_{j, (r_1-r)} = 0$ for all $j=1, \ldots, r_1-r$. This means that $(Z_{22})_{ij} = 0$ for all $1\le i, j\le r_1$ except for $(i, j) = (1, 1)$, $(Z_{21})_{ij} = 0$ for all $2\le i\le r_1, 1\le j\le n-r_1$ and $(Z_{32})_{r} = 0$ for $r=2, \ldots, r_1$. 

Finally, we prove that $(Z_{21})_{1j} = 0$ for all $j=1, \ldots, n-r_1$. As \eqref{StAb} and \eqref{bAN1} hold for $r=r_1-1$, we have $(A^{r_1-1}b_2)^TW_{11}(A^{r_1-1}b_2) = 0$, and thus $W_{11}(A^{r_1-1}b_2) = 0$. We have already known from the induction that 
\[
Z_{21}^T = \begin{pmatrix}
\tilde{z} & O_{ (n-r_1)\times(r_1-1)}
\end{pmatrix} \mbox{ and } Z_{22} = \begin{pmatrix}
z & 0\\
0 & O_{(r_1-1)\times (r_1-1)}
\end{pmatrix}
\]
for some $\tilde{z}\in\mathbb{R}^{n-r_1}$ and $z\in\mathbb{R}$. We substitute them to $W_{11}(A^{r_1-1}b_2) = 0$. Then as we have 
\begin{align*}
Z_{21}^TP_1^T &= \begin{pmatrix}
\tilde{z} & O_{(n-r_1)\times (r_1-1)}
\end{pmatrix}\begin{pmatrix}
0\\
*
\end{pmatrix} = O_{(n-r_1)\times n}, \\
Z_{21}^T\hat{P}_1^T &= \begin{pmatrix}
\tilde{z} & O_{(n-r_1)\times (r_1-1)}
\end{pmatrix}\begin{pmatrix}
c_1^T\\
*
\end{pmatrix} = \tilde{z}c_1^T, \\
Z_{22}P_1^T &= \begin{pmatrix}
z & 0\\
0 & O_{(r_1-1)\times (r_1-1)}
\end{pmatrix}\begin{pmatrix}
0\\
*
\end{pmatrix} = O_{r_1\times n}, 
\end{align*}
we obtain
\begin{align*}
W_{11}(A^{r_1-1}b_2) &=(S\Lambda Z_{21}^TP_1^T + SZ_{21}^T\hat{P}_1^T + \hat{P}_1Z_{22}P_1^T + P_1Z_{22}\hat{P}_1^T)(A^{r_1-1}b_2) \\
&=(c_1^TA^{r_1-1}b_2)S\tilde{z}^T = 0.
\end{align*}
 As $S$ is of full column rank and $c_1^TA^{r_1-1}b_2\neq 0$, we have $\tilde{z} = 0$. Therefore $Z_{21} = O$. %_{r_1\times (n-r_1)}$. %In addition, $(Z_{32})_{k} = 0$ for $k=2, \ldots, r_1$ due to the positive semidefiniteness of $Z = (Z_{ij})_{1\le i\le j\le3}$. 

\end{document}